\documentclass[10pt,aps,prb, preprint,reqno]{amsart}
\usepackage{amsmath}
\usepackage{amssymb}
\usepackage{yfonts}
\usepackage{hyperref}
\usepackage{mathrsfs}
\usepackage {latexsym}
\usepackage{graphics}
\usepackage{txfonts}
\usepackage{txfonts}	
\usepackage{breqn}
\usepackage{cite}
\usepackage{color}
\usepackage[shortlabels]{enumitem}

\newtheorem{theorem}{Theorem}[section]
\newtheorem{remark}{Remark}[section]
\newtheorem{lemma}[theorem]{Lemma}

\newtheorem{proposition}[theorem]{Proposition}

\newtheorem{define}{Definition}[section]  
 
\newtheorem{hypothesis}{Hypothesis}[section]

\DeclareMathOperator*{\Id}{Id}
\DeclareMathOperator*{\esssup}{esssup}
\DeclareMathOperator*{\supp}{supp}
\DeclareMathOperator*{\divergence}{div}

\allowdisplaybreaks

\begin{document}
\title[Non-uniqueness in law for the SQG equations]{Non-uniqueness in law of the surface quasi-geostrophic equations: the case of linear multiplicative noise}

 
\author[Kazuo Yamazaki]{Kazuo Yamazaki}  
\address{Department of Mathematics, University of Nebraska, Lincoln, 243 Avery Hall, PO Box, 880130, Lincoln, NE 68588-0130, U.S.A.; Phone: 402-473-3731; Fax: 402-472-8466}
\email{kyamazaki2@unl.edu}
\date{}

\begin{abstract}
The momentum formulation of the surface quasi-geostrophic equations consists of two nonlinear terms, besides the pressure term, one of which cannot be written in a divergence form. When the anti-divergence operator is applied to such nonlinear terms, in general, one cannot take advantage of the differentiation operator of order minus one unless the nonlinear terms are compactly supported away from the origin in Fourier frequency. Moreover, the two nonlinear terms of the momentum surface quasi-geostrophic equations are one derivative more singular than that of the Navier-Stokes equations. Upon employing the convex integration technique to the random partial differential equations corresponding to the momentum surface quasi-geostrophic equations forced by linear multiplicative noise, these issues create various difficulties, unseen in the deterministic scenario and even in the case the noise is additive. By making a key observation in case the solution is a shear flow and rewriting the difficult stochastic commutator error in terms of the oscillation error, we prove its non-uniqueness in law. 
\end{abstract}

\maketitle

\section{Introduction}
\subsection{Motivation}\label{Section 1.1} 
As far back as 1956, Landau and Lifschitz \cite{LL56} introduced the concept of fluctuating hydrodynamics in which fluctuations can be described by replacing a force term of the partial differential equations (PDEs) with certain random noise. Since then, the study of the stochastic Navier-Stokes equations has seen many developments (e.g. \cite{F08} and references therein). Very recently, the breakthrough technique of convex integration has led to many non-uniqueness results  in the deterministic case. Especially, the Nash-type iteration scheme seems very fit for the stochastic case, and we have seen successful convex integration schemes for various stochastic PDEs that resulted in their non-uniqueness in law. Expanding the list of equations on which convex integration may be applied is strongly desired, especially to the equations of mathematical physics such as the stochastic Yang-Mills equation; e.g. we refer to  \cite[Section 1.2]{HZZ22a}. However, many of these PDEs have nonlinear terms that are not of divergence form, and, to the best of the author's knowledge, successful convex integration schemes for equations of non-divergence form seem to be extremely limited.

In this connection, initiated by Charney \cite{C48} in 1948, the two-dimensional (2D) surface quasi-geostrophic (SQG) equations has attracted interest from many physicists and engineers due to its applications in geophysics, atmospheric sciences, and meteorology (e.g. \cite{HPGS95, MW06}). Application of convex integration on the SQG equations was considered to be especially difficult among active scalars (e.g. \cite{IV15}), and the first breakthrough approach by Buckmaster, Shkoller, and Vicol \cite{BSV19} was to consider its momentum form (see Equation \eqref{momentum}); interestingly, the nonlinear terms of the momentum SQG equations as a sum cannot be written in a divergence form in contrast to the nonlinear terms of the Navier-Stokes and the SQG equations (see Equations \eqref{SQG}-\eqref{NS}). In \cite{Y23a}, the author considered such momentum SQG equations forced by additive noise and employed convex integration to prove its non-uniqueness in law; however, due to multiple technical difficulties that arise precisely due to the nonlinear term of non-divergence form, it seemed difficult to extend its result to the case of linear multiplicative noise (see Remark \ref{Remark 1.3} for detailed explanations). In this manuscript, we overcome such difficulties and prove the non-uniqueness in law of the SQG equations in momentum formulation forced by linear multiplicative noise. 

\subsection{Review of past works}\label{Section 1.2}
We write ``$n$D'' for $n$-dimensional, denote $\partial_{t} \triangleq \frac{\partial}{\partial t}$, $\mathcal{R} = (\mathcal{R}_{1}, \mathcal{R}_{2})$ as the Riesz transform vector, and $(\Omega, \mathcal{F}, \textbf{P})$ as the standard probability space. For any $x = (x_{1}, x_{2}) \in \mathbb{R}^{2}$, we define $x^{\bot} \triangleq (-x_{2}, x_{1})$, as well as $\Lambda \triangleq (-\Delta)^{\frac{1}{2}}$ defined via Fourier transform so that $\Lambda^{\gamma} g(x) \triangleq \sum_{k\in\mathbb{Z}^{2}} \lvert k \rvert^{\gamma} \hat{g}(k) e^{ik\cdot x}$, $\gamma \in \mathbb{R}$, where $\hat{g}$ is the Fourier transform of $g$. Finally, we write $A \lesssim B$ and $A \approx B$ when there exist some universal constants $C_{1}, C_{2} \geq 0$ such that $A \leq C_{1}B$, and $C_{2} B \leq A \leq C_{1} B$, respectively. 

We let $\theta: \mathbb{R}_{+} \times \mathbb{T}^{2} \mapsto \mathbb{R}$ represent potential temperature and $\nu \geq 0$ so that the generalized SQG equations forced by $f: \mathbb{R}_{+} \times \mathbb{T}^{2} \mapsto \mathbb{R}$ can be written as 
\begin{subequations}\label{SQG}
\begin{align}
& \partial_{t} \theta + \divergence (u\theta) + \nu \Lambda^{\gamma_{1}} \theta = f,\\
& u = \Lambda^{1-\gamma_{2}} \mathcal{R}^{\bot} \theta 
\end{align}
\end{subequations}
with initial data $\theta^{\text{in}}(x) = \theta(0,x)$, where $\gamma_{1} \in (0, 2]$ and $\gamma_{2} \in [1,2]$. The diffusive term $\nu \Lambda^{\gamma_{1}} \theta$  comes from Ekman pumping at the boundary, while $f$ may represent sources or damping. 

For the purpose of comparison, let us denote the velocity and pressure fields respectively by $u: \mathbb{R}_{+} \times \mathbb{T}^{n} \mapsto \mathbb{R}^{n}$ and $p: \mathbb{R}_{+} \times \mathbb{T}^{n} \mapsto \mathbb{R}$ where $n \in \mathbb{N} \setminus \{1\}$ so that the system 
\begin{subequations}\label{NS}
\begin{align}
&\partial_{t} u + \divergence (u\otimes u) + \nabla p - \nu \Delta u = f, \\
& \nabla\cdot u = 0, 
\end{align}
\end{subequations} 
represents the Navier-Stokes and the Euler equations when $\nu > 0$, and $\nu = 0$, respectively. Such generalized SQG equations was introduced in \cite[Equation (1.3)]{K10} (see also \cite{CIW08}); hereafter, we refer to special case when $\gamma_{2} = 1$ as the SQG equations. In the following discussions, we focus on the unforced case $f \equiv 0$. 

As a pioneering work, Constantin, Majda, and Tabak \cite{CMT94} in 1994 analytically and numerically demonstrated various similarities in the structures and the properties of the solutions to the Navier-Stokes, the Euler, and the SQG equations. E.g., the generalized SQG equations has the rescaling property such that if $\theta(t,x)$ solves it, then so does $\theta_{\lambda}(t,x) = \lambda^{\gamma_{1} + \gamma_{2} - 2} \theta (\lambda^{\gamma_{1}} t, \lambda x)$ for any $\lambda \in \mathbb{R}_{+}$; moreover, its smooth solution has a uniform $L^{\infty}(\mathbb{T}^{2})$-bound by that of initial data just like the vorticity formulation of the 2D Euler equations. Many important works concerning the global regularity issue of the SQG equations appeared, most notably, \cite{CW99} in the $L^{\infty}(\mathbb{T}^{2})$-subcritical case $\gamma_{1}  >1, \gamma_{2} = 1$, \cite{CV10, CV12, KN10, KNV07} in the $L^{\infty}(\mathbb{T}^{2})$-critical case $\gamma_{1} = 1, \gamma_{2} = 1$; the global regularity issue in the $L^{\infty}(\mathbb{T}^{2})$-supercritical case $\gamma_{1} + \gamma_{2} < 2$ remains largely unknown. Concerning weak solutions, Resnick \cite[Theorem 2.1]{R95} constructed global-in-time weak solutions to the non-diffusive SQG equations starting from initial data $\theta^{\text{in}} \in L^{2}(\mathbb{T}^{2})$, and Marchand \cite[Theorem 1.1]{M08} subsequently improved it to $\theta^{\text{in}} \in L^{p}(\mathbb{T}^{2})$ for $p \in (\frac{4}{3},\infty)$. 

The mathematical theory of the stochastic SQG equations was launched by Zhu \cite{Z12a} and Zhu \cite{Z12b} who obtained many important results. We highlight those of the most relevance to our current manuscript: the existence of Markov selection if $\gamma_{1} > 0$ in \cite[Theorem 4.2.5]{Z12a}; the existence of a probabilistically weak martingale solution proven via Galerkin approximation if $\gamma_{1} > 0$ in \cite[Theorem 4.3.2]{Z12b}; the existence of a path-wise unique and consequently probabilistically strong solution starting from $\theta^{\text{in}} \in L^{p}(\mathbb{T}^{2})$ for $ p> (\gamma_{1} - \frac{1}{2})^{-1}$ where $\gamma_{1} > \frac{1}{2}$ in \cite[Theorem 4.4.4]{Z12b}. We also refer to \cite{BNSW18, RZZ15, Y22a} and references therein.  

Let us review the recent developments of the convex integration technique that has originated in the deterministic case. De Lellis and Sz$\acute{\mathrm{e}}$kelyhidi Jr. proved non-uniqueness of $n$D Euler equations for all $n \in \mathbb{N}\setminus \{1\}$ via differential inclusion approach in \cite{DS09}, and followed with a series of influential works \cite{DS10, DS13} where \cite{DS13}, for the first time, proved the negative direction of the Onsager's conjecture at the $C_{t,x}^{0}$-level. De Lellis and Sz$\acute{\mathrm{e}}$kelyhidi Jr., along with Choffrut, also provided the 2D analog of \cite{DS13} in \cite{CDS12a}.  Refinements and improvements of the convex integration technique through many works (e.g. \cite{BDIS15}) eventually led to \cite{I18} in which Isett settled the negative direction of the Onsager's conjecture in any dimension $n \geq 3$; this problem was open for more than two decades since the positive direction was solved by Constantin, E, and Titi \cite{CET94} and Eyink \cite{E94} in 1994. We mention that the applications of the convex integration technique generally become more difficult in lower dimensions, essentially due to the lack of Mikado flows in the 2D case, and it required a very recent breakthrough by Giri and Radu \cite{GR23} to prove the negative direction of the Onsager's conjecture in the 2D case. Besides Onsager's conjecture, Buckmaster and Vicol \cite{BV19a} settled Serrin's conjecture concerning whether a non-constant weak solution to the 3D Navier-Stokes equations can come to rest in finite time; non-uniqueness of the 2D Navier-Stokes equations diffused via $(-\Delta)^{m}$ for $m < 1$ was subsequently proven by Luo and Qu \cite{LQ20}. We also refer to the following works and references therein for applications of the convex integration technique with an emphasis that all of their nonlinear terms are of divergence form: \cite{LZZ22} on the magnetohydrodynamics (MHD) system; \cite{MS20} on transport-diffusion equation; \cite{BMS21} on the power-law model; \cite{LTZ20} on the Boussinesq system. We also refer to \cite{BV19b} for an excellent review. 

In direct relevance to the SQG equations, we now recall that C$\acute{\mathrm{o}}$rdoba, Faraco, and Gancedo \cite{CFG11} first applied the convex integration technique on the non-diffusive incompressible porous media equations. Subsequently, Isett and Vicol \cite{IV15} employed convex integration on a general family of active scalars; however, the odd Fourier symbol of the velocity field $u = \mathcal{R}^{\bot} \theta$ in \eqref{SQG} created technical difficulties and led to an exclusion of the SQG equations in \cite{IV15}. Several years later, the research direction of convex integration on the SQG equations  started with \cite{BSV19} by Buckmaster, Shkoller, and Vicol. Specifically, they considered the momentum SQG equations solved by $v$ such that 
\begin{equation}\label{est 1}
\theta = - \nabla^{\bot} \cdot v; 
\end{equation}
i.e., 
\begin{subequations}\label{momentum} 
\begin{align}
&\partial_{t} v + (u\cdot\nabla) v - (\nabla v)^{T} \cdot u + \nabla p + \Lambda^{\gamma_{1}} v = 0, \\
&u =\Lambda v, \hspace{3mm} \nabla\cdot v = 0, 
\end{align}
\end{subequations} 
for which due to an identity of 
\begin{equation}\label{est 38}
(u\cdot\nabla) v - (\nabla v)^{T} \cdot u = u^{\bot} (\nabla^{\bot} \cdot v),
\end{equation} 
the global-in-time weak solution $v$ such that $\lVert v(t) \rVert_{\dot{H}_{x}^{\frac{1}{2}}}^{2} \leq \lVert v^{\text{in}} \rVert_{\dot{H}_{x}^{\frac{1}{2}}}^{2}$ starting from $v(0) = v^{\text{in}} \in \dot{H}^{\frac{1}{2}}(\mathbb{T}^{2})$ was known due to \cite[Theorem 1.2 (ii)]{M08}. The authors of \cite{BSV19} were able to prescribe initial energy $e(t)$ and construct a solution $v$ to \eqref{momentum} such that $\lVert v(t) \rVert_{\dot{H}_{x}^{\frac{1}{2}}}^{2} = e(t)$ via convex integration. Subsequently, Isett and Ma \cite{IM21}, as well as Cheng, Kwon, and Li \cite{CKL20} also provided different lines of attack in applying convex integration on the SQG equations; we will briefly elaborate on the latter approach in Remark \ref{Remark 1.1}.  

Extensions of the convex integration technique to the stochastic PDEs was initiated in the works of \cite{BFH20} by Breit, Feireisl, and Hofmanov$\acute{\mathrm{a}}$, and \cite{CFF19} by Chiodaroli, Feireisl, and Flandoli  in which path-wise non-uniqueness was proven for certain stochastic Euler equations. Hofmanov$\acute{\mathrm{a}}$, Zhu, and Zhu \cite{HZZ19} were the first to proven non-uniqueness in law via convex integration and they achieved this on the 3D Navier-Stokes equations forced by noise of additive, linear multiplicative, and nonlinear types. Let us point out that due to the classical Yamada-Watanabe theorem, non-uniqueness in law implies non-uniqueness path-wise. Subsequently, we have witnessed rapid developments of convex integration to prove non-uniqueness in law for stochastic PDEs: \cite{HZZ20} on the Euler equations; \cite{RS21, Y20a, Y20c, Y21c} on the fractional Navier-Stokes equations in 2D and 3D cases; \cite{Y21a} on the Boussinesq system; \cite{Y21d} on the MHD system; \cite{KY22} on the transport-diffusion equation; \cite{LZ22} on the power-law model. We also mention that Hofmanov$\acute{\mathrm{a}}$, Zhu, and Zhu \cite{HZZ21a} were able to prescribe energy or initial data following the approach of \cite{BMS21} for the 3D Navier-Stokes equations forced by additive noise; the same authors were able to treat the case of space-time white noise in \cite{HZZ21b}. In all these works, the authors somehow transformed the stochastic PDEs to random PDEs and employed convex integration on the latter to eventually claim the non-uniqueness by transforming back to the original stochastic PDEs. In contrast, Chen, Dong, and Zhu \cite{CDZ22} employed the convex integration technique directly on the stochastic Navier-Stokes equations although after applying mathematical expectation. For more recent works on convex integration applied to stochastic PDEs, we refer to \cite{B23, HLP23, HPZZ23a, HPZZ23b, HZZ22b, HZZ23a, BJLZ23, LZ23, MS23, P23}. 

Convex integration on stochastic SQG equations has been accomplished in \cite{BLW23, HZZ22a, HLZZ23, Y23a}. First, Hofmanov$\acute{\mathrm{a}}$, Zhu, and Zhu \cite{HZZ22a} employed the convex integration scheme of \cite{CKL20} on the SQG equations forced by an additive white-in-space noise. Independently, the author \cite{Y23a} employed the convex integration scheme of \cite{BSV19} on the momentum SQG equations forced by an additive white-in-time noise. It was elaborated in \cite[Section 2.2]{Y23a} why one cannot simply transform \eqref{momentum} forced by an additive noise to a random PDE by subtracting out a stochastic Stokes equation forced by the same additive noise and employing convex integration from \cite{BSV19}; in short, the inductive estimates of \cite{BSV19} included a material derivative estimate on the Reynolds stress 
\begin{equation}\label{est 13}
\text{``} \lVert (\partial_{t} + u_{q} \cdot \nabla) \mathring{R}_{q} \rVert_{C^{0}} \leq \lambda_{q}^{2} \delta_{q}^{\frac{1}{2}} \lambda_{q+1} \delta_{q+1}\text{''}
\end{equation} 
(see \cite[Equation (3.7)]{BSV19}) and while this was indispensable, the Reynold stress in the stochastic case cannot be differentiable in time, inherently from the fact that Brownian path is $\textbf{P}$-almost surely ($\textbf{P}$-a.s.) only locally H$\ddot{\mathrm{o}}$lder continuous with any exponent in $(0,\frac{1}{2})$. Finally, very recently, Hofmanov$\acute{\mathrm{a}}$, Luo, Zhu, and Zhu \cite{HLZZ23} were able to consider the case of space-time white noise as well (see also \cite{BLW23}). 

As we discussed in Section \ref{Section 1.1}, the examples of convex integration on PDEs of non-divergence form seem to be extremely rare. To the best of the author's knowledge, the only exception is the work of \cite{VK18} in which Vo and Kim worked on deterministic conservation models that included the Burgers' equation but rewrote it in the form of Kardar-Parisi-Zhang equation in its proof, effectively proving non-uniqueness for the latter equation. Nevertheless, the non-uniqueness result from  \cite{VK18} was in the 1D, non-diffusive, and unforced case, and relied on the approach of differential inclusion from \cite{DS09} rather than Nash-type iteration, and it seems difficult to adapt to the stochastic case with diffusion in higher dimensions.  
 
In this manuscript we employ convex integration on the 2D momentum SQG equations \eqref{momentum} forced by linear multiplicative noise, specifically  
\begin{subequations}\label{est 30}
\begin{align} 
&dv + [(u\cdot\nabla) v - (\nabla v)^{T} \cdot u + \nabla p + \Lambda^{\gamma_{1}} v]dt = v dB, \label{est 30a}\\
&u = \Lambda^{2-\gamma_{2}} v,  \hspace{2mm} \nabla\cdot v = 0  \label{est 30b}
\end{align}
\end{subequations}
with initial data $v^{\text{in}}(x) = v(0,x)$, where $B$ is a $\mathbb{R}$-valued Wiener process. We consider this case $u = \Lambda^{2-\gamma_{2}} v$ for $\gamma_{2} \in [1, 2)$ (see Equation \eqref{est 315a}) for generality and possible use in future works although we must compromise to the case $\gamma_{2} = 1$ at the last step of the proof of Theorem \ref{Theorem 2.1}. Let us explain three of our main motivation aspects separately in the following remarks. 

\begin{remark}\label{Remark 1.1}
As we mentioned in Section \ref{Section 1.2}, the convex integration technique has been successfully employed on the stochastic SQG equations in \cite{BLW23, HZZ22a, HLZZ23, Y23a} and all of \cite{BLW23, HZZ22a, HZZ23a} have relied on the convex integration scheme of \cite{CKL20} that studies the formulation of the SQG equations solved by $f$ such that $\Lambda f = \theta$; cf. $-\nabla^{\bot}\cdot v = \theta$ in \eqref{est 1}. There are multiple advantages in the convex integration scheme of \cite{CKL20}. On the other hand, we recall from Section \ref{Section 1.1} the significant desire to employ convex integration on PDEs in mathematical physics that consist of nonlinear terms in non-divergence form. Because the nonlinear terms $(u\cdot\nabla) v - (\nabla v)^{T} \cdot u$ of \eqref{est 30a} are not of divergence form, the momentum SQG equations \eqref{momentum} seems to offers a rare opportunity to understand convex integration scheme on PDEs with nonlinear terms of non-divergence form. Furthermore, the convex integration scheme of \cite{BSV19} is interesting because, to the best of the authors' knowledge,  
\begin{enumerate}
\item it is the only successful attempt of convex integration on fully dissipative system with the frequency parameter that grows only exponentially rather than super-exponentially, i.e., 
\begin{quotation} 
\begin{equation}\label{est 11}
\text{``}\lambda_{q} = \lambda_{0}^{q} \text{ for some integer } \lambda_{0} \gg 1\text{''} 
\end{equation}
\end{quotation} 
on \cite[p. 1823]{BSV19},
\item while being a rare example of doing so without relying on intermittency, and consequently all the estimates are based on H$\ddot{\mathrm{o}}$lder spaces. 
\end{enumerate} 
We note that Colombo, De Lellis, and De Rosa employed convex integration to construct weak solutions to the 3D Navier-Stokes equations diffused via fractional Laplacian $(-\Delta)^{m}$ without intermittency in \cite[Corollary 2.3]{CDD18}. Nevertheless, their result required super-exponential frequency parameter $\lambda_{q} = a^{b^{q}}$ (see \cite[Equation (16)]{CDD18}) and the range of $m$ was limited to $m \in (0, \frac{1}{2})$ in which $\frac{1}{2}$ can can be understood as the natural threshold up to which $\mathcal{B} (-\Delta)^{m}$ can informally be dominated by an identity function, where $\mathcal{B}$ is the anti-divergence operator from  Lemma \ref{Inverse-divergence lemma} (see also \cite{D19}). In contrast, \cite{BSV19} was able to cover the full dissipation range of $\gamma_{1} \in [0, \frac{3}{2})$.
\end{remark}

\begin{remark}
Linear multiplicative noise is one of the most fundamental examples of random noise that appears in various real-world applications, e.g. \cite{F75} by Fleming on population models forced by such linear multiplicative noise (see also \cite{D72}).  The special regularization effect from linear multiplicative noise in general is widely established. E.g., although the 3D deterministic Euler equations does not admit a global-in-time unique solution starting from sufficiently small initial data, Kim \cite{K11} observed that certain multiplicative noise can play the role of damping and therefore PDEs forced by such random noise, even in the absence of any diffusion or damping terms, can admit a global-in-time unique solution starting from sufficiently small initial data. Buckmaster, Nahmod, Staffilani, and Widmayer \cite{BNSW18} obtained such a result for the SQG equations forced by multiplicative noise with the size condition in terms of the noise coefficient; we mention that the size condition may also be imposed on a coefficient of its nonlinear term (e.g. \cite{YM19} for the Hall-MHD system forced by L$\acute{\mathrm{e}}$vy-type linear multiplicative noise). We also refer to \cite{RZZ14} for other regularizing effects from multiplicative noise. In light of such phenomena, it is of interest to employ convex integration on various stochastic PDEs forced by linear multiplicative noise in hope to improve to general multiplicative noise in future works. 
\end{remark}

\begin{remark}\label{Remark 1.3}
In effort to employ convex integration on \eqref{est 30}, we may consider 
\begin{equation}\label{est 325} 
\Upsilon(t) \triangleq e^{B(t)} \hspace{1mm} \text{ and } \hspace{1mm} y(t) \triangleq \Upsilon^{-1} v 
\end{equation} 
which lead us from \eqref{est 30} to investigate  
\begin{align*} 
&\partial_{t} y + \frac{1}{2} y + \Upsilon [ (\Lambda^{2-\gamma_{2}} y \cdot \nabla) y - (\nabla y)^{T} \cdot \Lambda^{2- \gamma_{2}} y] + \nabla (\Upsilon^{-1} p) + \Lambda^{\gamma_{1}} y = 0, \\
&\nabla\cdot y = 0,
\end{align*}
and thus for all $q \in \mathbb{N}_{0} \triangleq \mathbb{N} \cup \{0\}$ 
\begin{subequations}\label{est 320}
\begin{align}
& \partial_{t} y_{q} + \frac{1}{2} y_{q} + \Upsilon [(\Lambda^{2- \gamma_{2}} y_{q} \cdot \nabla) y_{q} - (\nabla y_{q})^{T} \cdot \Lambda^{2- \gamma_{2}} y_{q} ] + \nabla p_{q} + \Lambda^{\gamma_{1}} y_{q}= \divergence \mathring{R}_{q}, \label{est 320a} \\
& \nabla\cdot y_{q} = 0, \label{est 320b}
\end{align}
\end{subequations}
where $\mathring{R}_{q}$ is a symmetric, trace-free matrix. We define   
\begin{equation}\label{est 35}
\lambda_{q} \triangleq a^{b^{q}} \text{ for } a \in 5 \mathbb{N} \text{ and }  b \in \mathbb{N} \text{ sufficiently large, as well as } \delta_{q} \triangleq \lambda_{q}^{-2\beta}, 
\end{equation}
(cf. Equation \eqref{est 11}). We set the convention that $\sum_{1 \leq \iota \leq 0} \triangleq 0$, define 
\begin{equation}\label{est 463}
t_{q} \triangleq -2 + \sum_{1 \leq \iota \leq q} \delta_{\iota}^{\frac{1}{2}} < -1 
\end{equation} 
where the inequality can be guaranteed as long as 
\begin{equation}\label{est 442} 
2 \leq a^{b\beta},  
\end{equation} 
(see Equation \eqref{est 153b}), and consider \eqref{est 320} inductively over $[t_{q}, T_{L}]$ for $T_{L}$ to be determined in \eqref{est 312}. We let 
\begin{equation}\label{est 32}
l \triangleq \lambda_{q+1}^{-\alpha}  
\end{equation} 
where $\alpha$ will be optimized in \eqref{est 315c}, and $\{\phi_{l}\}_{l > 0}$ and $\{ \varphi_{l}\}_{l > 0}$ be the families of standard mollifiers with mass one and compact support respectively on $\mathbb{R}^{2}$ and $(\tau_{q+1}, 2 \tau_{q+1}] \subset \mathbb{R}_{+}$, where $\tau_{q+1}$ will be chosen in \eqref{est 210}. Then we extend $B(t)= B(0)$ for all $t \in [-2, 0]$ and mollify $y_{q}, \mathring{R}_{q}$, and $\Upsilon$ to obtain 
\begin{equation}\label{est 337} 
y_{l} \triangleq y_{q} \ast_{x} \phi_{l} \ast_{t} \varphi_{l}, \hspace{1mm} \mathring{R}_{l} \triangleq \mathring{R}_{q} \ast_{x} \phi_{l} \ast_{t} \varphi_{l}, \hspace{1mm} \text{ and } \hspace{1mm} \Upsilon_{l} \triangleq \Upsilon \ast_{t} \varphi_{l}. 
\end{equation} 
It follows from \eqref{est 38}, \eqref{est 320}, and \eqref{est 337} that 
\begin{equation}\label{est 341} 
\partial_{t} y_{l} + \frac{1}{2} y_{l} + \Upsilon_{l} [ ( \Lambda^{2- \gamma_{2}} y_{l} \cdot\nabla) y_{l} - (\nabla y_{l} )^{T} \cdot \Lambda^{2- \gamma_{2}} y_{l} ] + \nabla p_{l} + \Lambda^{\gamma_{1}} y_{l} = \divergence \mathring{R}_{l} + \divergence R_{\text{Com1}} 
\end{equation} 
if we define 
\begin{subequations}\label{est 338} 
\begin{align}
R_{\text{Com1}} \triangleq& \mathcal{B} \left[ \Upsilon_{l} [ \Lambda^{2-\gamma_{2}} y_{l}^{\bot} (\nabla^{\bot} \cdot y_{l})] -  (\Upsilon [  \Lambda^{2- \gamma_{2}} y_{q}^{\bot} (\nabla^{\bot} \cdot y_{q} ) ]) \ast_{x} \phi_{l} \ast_{t} \varphi_{l} \right],  \\
p_{l} \triangleq& p_{q} \ast_{x} \phi_{l} \ast_{t} \varphi_{l}.  
\end{align}
\end{subequations} 
Now in the case of the Navier-Stokes equations, once convex integration has been successfully employed in the case of additive noise, extending to the case of linear multiplicative noise has been considered to be relatively straight-forward, as shown in \cite[Equation (6.20)]{HZZ19}. In short, if $a_{k,j}(t,x)$ was the amplitude function for the perturbation 
\begin{equation}\label{est 164} 
w_{q+1} \triangleq y_{q+1} - y_{l}
\end{equation}    
(cf. Equation \eqref{est 447}), then concerning the $w_{q+1}$ part of the main nonlinear term $(\Lambda^{2-\gamma_{2}} y_{q+1} \cdot \nabla) y_{q+1} - (\nabla y_{q+1})^{T} \cdot \Lambda^{2-\gamma_{2}} y_{q+1}$, namely 
\begin{equation*}
( \Lambda^{2- \gamma_{2}} w_{q+1} \cdot\nabla) w_{q+1} - (\nabla w_{q+1} )^{T} \cdot \Lambda^{2- \gamma_{2}} w_{q+1}, 
\end{equation*}
an analogous approach following \cite[Equation (6.20)]{HZZ19} in the case of \eqref{est 320a} would be to define 
\begin{align}\label{est 454} 
\bar{a}_{k,j} (t,x) \triangleq \Upsilon_{l}^{-\frac{1}{2}} a_{k,j}
\end{align}
so that it can play the crucial role within 
\begin{align*}
&\Upsilon_{l}[( \Lambda^{2- \gamma_{2}} w_{q+1} \cdot\nabla) w_{q+1} - (\nabla w_{q+1} )^{T} \cdot \Lambda^{2- \gamma_{2}} w_{q+1}]  \nonumber \\
=& (\Lambda^{2- \gamma_{2}} \Upsilon_{l}^{\frac{1}{2}} w_{q+1} \cdot\nabla) \Upsilon_{l}^{\frac{1}{2}}w_{q+1} - (\nabla \Upsilon_{l}^{\frac{1}{2}}w_{q+1} )^{T} \cdot \Lambda^{2- \gamma_{2}} \Upsilon_{l}^{\frac{1}{2}}w_{q+1}. 
\end{align*}
This approach effectively reduces the oscillation term in the linear multiplicative case with $\bar{a}_{k,j}$ to the oscillation term in the additive case with $a_{k,j}$ because $\Upsilon_{l}^{\frac{1}{2}} \bar{a}_{k,j} = a_{k,j}$ due to \eqref{est 454}, and makes the estimates in the additive case  become directly applicable in the linear multiplicative case. 

This nice strategy turns out to face multiple difficulties for other models that are more complex than the Navier-Stokes equations. E.g., the author in \cite{Y21d} considered the following generalized MHD system forced by linear multiplicative noise:
\begin{align*}
&du + [\nu_{1} (-\Delta)^{m_{1}} u + \text{div} (u\otimes u - b \otimes b) + \nabla p] dt = u dB_{1}, \hspace{2mm} \nabla\cdot u = 0, \\
&db + [\nu_{2} (-\Delta)^{m_{2}} b + \text{div} (b\otimes u - u\otimes b)]dt = b dB_{2}, \hspace{10mm} \nabla\cdot b = 0, 
\end{align*}
where $b: \mathbb{R}_{+} \times \mathbb{T}^{3} \mapsto \mathbb{R}^{3}$ denotes the magnetic field, $B_{1}$ and $B_{2}$ are $\mathbb{R}$-valued independent Wiener processes, and $m_{1}, m_{2} \in (0,1)$. We can define 
\begin{equation*}
\Upsilon_{1} \triangleq e^{B_{1}}, \Upsilon_{2} \triangleq e^{B_{2}}, v \triangleq \Upsilon_{1}^{-1} u, \Xi \triangleq \Upsilon_{2}^{-1}b 
\end{equation*}
analogously to \eqref{est 325}, which leads us to search for solutions $(v_{q}, \Xi_{q}, \mathring{R}_{q}^{v}, \mathring{R}_{q}^{\Xi})$ for $q \in \mathbb{N}_{0}$ that satisfies 
\begin{align*}
&\partial_{t} v_{q} + \frac{1}{2} v_{q} + (-\Delta)^{m_{1}} v_{q} + \divergence ( \Upsilon_{1} (v_{q} \otimes v_{q}) - \Upsilon_{1}^{-1} \Upsilon_{2}^{2} (\Xi_{q} \otimes \Xi_{q} ))  + \nabla p_{q} = \divergence \mathring{R}_{q}^{v},  \hspace{2mm} \nabla\cdot v_{q}  = 0,\\
&\partial_{t} \Xi_{q} + \frac{1}{2}\Xi_{q} + (-\Delta)^{m_{2}} \Xi_{q} + \Upsilon_{1} \divergence (\Xi_{q} \otimes v_{q} - v_{q} \otimes \Xi_{q} ) = \divergence \mathring{R}_{q}^{\Xi},  \hspace{22mm} \nabla\cdot \Xi_{q} = 0, 
\end{align*}
where $\mathring{R}_{q}^{v}$ is a symmetric trace-free matrix and $\mathring{R}_{q}^{\Xi}$ is a skew-symmetric matrix, analogously to \eqref{est 320}. We now mollify to obtain 
\begin{align*}
&v_{l} \triangleq v_{q} \ast_{x} \phi_{l} \ast_{t} \varphi_{l}, \hspace{3mm} \Theta_{l} \triangleq \Theta_{q} \ast_{x} \phi_{l} \ast_{t} \varphi_{l}, \hspace{3mm} \Upsilon_{k,l} \triangleq \Upsilon_{k}\ast_{t} \varphi_{l}, k\in \{1,2\},\\
&\mathring{R}_{l}^{v} \triangleq \mathring{R}_{q}^{v} \ast_{x} \phi_{l} \ast_{t} \varphi_{l}, \hspace{2mm} \mathring{R}_{l}^{\Theta} \triangleq \mathring{R}_{q}^{\Theta} \ast_{t} \phi_{l} \ast_{t} \varphi_{l} 
\end{align*}
analogously to \eqref{est 337}, and deduce an appropriate mollified system of 
\begin{subequations}\label{est 457}
\begin{align}
&\partial_{t} v_{l} + \frac{1}{2} v_{l} + (-\Delta)^{m_{1}} v_{l} \nonumber\\
& \hspace{20mm} + \divergence (\Upsilon_{1,l} (v_{l} \otimes v_{l}) - \Upsilon_{1,l}^{-1} \Upsilon_{2,l}^{2} (\Xi_{l} \otimes \Xi_{l}) ) + \nabla p_{l} = \divergence (\mathring{R}_{l}^{v}  + R_{\text{Com1}}^{v}), \label{est 457a}\\
& \partial_{t} \Xi_{l} + \frac{1}{2} \Xi_{l} + (-\Delta)^{m_{2}} \Xi_{l} + \Upsilon_{1,l} \divergence (\Xi_{l} \otimes v_{l} - v_{l} \otimes \Xi_{l} ) = \divergence( \mathring{R}_{l}^{\Xi}  + R_{\text{Com1}}^{\Xi}),  \label{est 457b}
\end{align}
\end{subequations}
where 
\begin{align*}
p_{l} \triangleq&  \left(\Upsilon_{1} \frac{ \lvert v_{q} \rvert^{2}}{3} \right)\ast_{x} \phi_{l} \ast_{t} \varphi_{l} - \left(\Upsilon_{1}^{-1} \Upsilon_{2}^{2} \frac{ \lvert \Xi_{q} \rvert^{2}}{3} \right)\ast_{x} \phi_{l} \ast_{t} \varphi_{l} \nonumber\\
& \hspace{20mm} + p_{q} \ast_{x} \phi_{l} \ast_{t} \varphi_{l} - \Upsilon_{1,l} \frac{ \lvert v_{l} \rvert^{2}}{3} + \Upsilon_{1,l}^{-1} \Upsilon_{2,l}^{2} \frac{ \lvert \Xi_{l} \rvert^{2}}{3},  \\
R_{\text{Com1}}^{v}\triangleq& - (\Upsilon_{1} (v_{q} \mathring{\otimes} v_{q}) ) \ast_{x} \phi_{l} \ast_{t} \varphi_{l} + (\Upsilon_{1}^{-1} \Upsilon_{2}^{2} (\Xi_{q} \mathring{\otimes} \Xi_{q} ))\ast_{x} \phi_{l} \ast_{t} \varphi_{l} \nonumber\\
& \hspace{20mm} + \Upsilon_{1,l} (v_{l} \mathring{\otimes} v_{l}) - \Upsilon_{1,l}^{-1} \Upsilon_{2,l}^{2} (\Xi_{l} \mathring{\otimes} \Xi_{l} ),  \\
R_{\text{Com1}}^{\Xi} \triangleq& -(\Upsilon_{1}(\Xi_{q} \otimes v_{q} )) \ast_{x} \phi_{l} \ast_{t} \varphi_{l} + (\Upsilon_{1} (v_{q} \otimes \Xi_{q}) )\ast_{x} \phi_{l} \ast_{t} \varphi_{l} \nonumber \\
& \hspace{20mm} + \Upsilon_{1,l} (\Xi_{l} \otimes v_{l} ) - \Upsilon_{1,l} (v_{l} \otimes \Xi_{l}). 
\end{align*}
A glance at \eqref{est 457} shows that there is no way to reduce all four nonlinear terms therein to the corresponding oscillation terms in its additive case.  Indeed, in order to cancel out $\Upsilon_{1,l}$ in $\Upsilon_{1,l} v_{l} \otimes v_{l}$ of \eqref{est 457a}, $v_{l}$ must eliminate $\Upsilon_{1,l}^{-\frac{1}{2}}$. With that fixed, in order to cancel out $\Upsilon_{1,l}$ in $\Upsilon_{1,l} \Theta_{l} \otimes v_{l} $ and $\Upsilon_{1,l} v_{l} \otimes \Theta_{l}$ in \eqref{est 457b}, we will have to ask $\Theta_{l}$ to eliminate $\Upsilon_{1,l}^{-\frac{1}{2}}$ as well. However, with such a choice on $v_{l}$ and $\Theta_{l}$, $\Upsilon_{1,l}^{-1} \Upsilon_{2,l}^{2} $ within $\Upsilon_{1,l}^{-1} \Upsilon_{2,l}^{2} \Theta_{l} \otimes \Theta_{l}$ of \eqref{est 457a} does not cancel out; in fact, it becomes $\Upsilon_{1,l}^{-2} \Upsilon_{2,l}^{2}$. We refer to  \cite[Remarks 2.1 and 5.1]{Y21d} for detailed discussions on the problems and new ideas to overcome them. 

The challenges in the momentum SQG equations forced by linear multiplicative noise are even more serious. There are two major difficulties and the first one starts with the inductive step $q=0$. In order to explain, let us suppose that the inductive hypothesis (cf. Hypothesis \ref{Hypothesis 5.1}) is simply  
\begin{subequations}\label{est 431}
\begin{align}
&\lVert y_{q} \rVert_{C_{t,x,q}} \leq F_{1}(t,L) \left(1+ \sum_{1 \leq \iota \leq q} \delta_{\iota}^{\frac{1}{2}} \right), \label{est 431a} \\
& \lVert \mathring{R}_{q} \rVert_{C_{t,x,q}} \leq \epsilon_{\gamma} F_{2}(t,q,L) \label{est 431b}
\end{align} 
\end{subequations} 
for some functions $F_{1}(t,L)$ and $F_{2}(t,q,L)$, where $\epsilon_{\gamma} > 0$ is a universal constant from the Geometric Lemma of Lemma \ref{Geometric Lemma} and $\lVert \cdot \rVert_{C_{t,x,q}}$ will be introduced in \eqref{est 00}. A typical choice of a shear flow at step $q = 0$ such as 
\begin{equation}\label{est 432}
y_{0}(t,x) \approx F_{1}(t,L) 
\begin{pmatrix}
\sin(x_{2})\\
0
\end{pmatrix}
\end{equation}  
satisfies $(\Lambda^{2-\gamma_{2}} y_{0} \cdot \nabla ) y_{0} = 0$ but 
\begin{equation}\label{est 434} 
(\nabla y_{0})^{T} \cdot \Lambda^{2-\gamma_{2}} y_{0} \neq 0. 
\end{equation} 
Therefore, the Reynolds stress at step $q= 0$, namely $\mathring{R}_{0}$, must consist of $\mathcal{B} ( ( \nabla y_{0})^{T} \cdot \Lambda^{2-\gamma_{2}} y_{0})$ where $\mathcal{B}$ is the anti-divergence operator from  Lemma \ref{Inverse-divergence lemma}. Unfortunately, the differentiation operator of order -1 is wasted and we end up with  
\begin{align} 
\lVert \mathring{R}_{0} \rVert_{C_{t,x,0}} \geq C_{1} F_{1}(t,L)^{2} \label{est 478}
\end{align}
for some $C_{1} \geq 0$ due to \eqref{est 431a}. Along with \eqref{est 431b}, this implies 
\begin{equation}\label{est 435}
\epsilon_{\gamma} F_{2}(t,0, L) \geq \lVert \mathring{R}_{0} \rVert_{C_{t,x,0}} \geq C_{1} F_{1}(t,L)^{2}. 
\end{equation} 
With $\gamma_{k} \in C^{\infty} (B( \Id, \epsilon_{\gamma}))$ from the Geometric Lemma of Lemma \ref{Geometric Lemma}, $\mathring{R}_{q,j}$ from \eqref{est 355}, and $\tau_{q+1}$ from \eqref{est 210}, we will need to define the amplitude function of the form 
\begin{equation}\label{est 436}
a_{k,j} (t,x) = D(j,q) \gamma_{k} \left( \Id - \frac{ \mathring{R}_{q,j} (t,x)}{F_{2}(\tau_{q+1} j, q,L)} \right) 
\end{equation} 
for some function $D(j,q)$ so that $\Id - \frac{\mathring{R}_{q,j}}{F_{2}(\tau_{q+1} j, q, L)} \overset{\eqref{est 431b}}{\in} B( \Id, \epsilon_{\gamma}) = Dom(\gamma_{k})$. Then, at the application of the Geometric Lemma, we have for $\chi_{j}$ to be defined in \eqref{est 209},   
\begin{align*}
& \sum_{j} \chi_{j}^{2} \mathring{R}_{q,j} + \frac{ \gamma_{2} \lambda_{q+1}^{2-\gamma_{2}}}{2}  \sum_{j,k}  k^{\bot} \otimes k^{\bot} \chi_{j}^{2} a_{k,j}^{2} - \frac{\gamma_{2} \lambda_{q+1}^{2-\gamma_{2}}}{2}  \sum_{j,k}  \chi_{j}^{2} \Id a_{k,j}^{2} \\
\overset{\eqref{est 436}}{=}& \sum_{j} \chi_{j}^{2} \lambda_{q+1}^{2- \gamma_{2}} \left[ \frac{ \mathring{R}_{q,j}}{\lambda_{q+1}^{2- \gamma_{2}}}   + \frac{\gamma_{2}}{2} \sum_{k}  k^{\bot} \otimes k^{\bot} D^{2}(j,q) \gamma_{k}^{2} \left( \Id - \frac{ \mathring{R}_{q,j}}{F_{2}(\tau_{q+1} j, q, L)} \right) - \frac{\gamma_{2}}{2} \sum_{k} \Id a_{k,j}^{2} \right]  \\
\overset{\eqref{est 265}}{=}& \sum_{j} \chi_{j}^{2} \lambda_{q+1}^{2- \gamma_{2}} \left[ \frac{ \mathring{R}_{q,j}}{\lambda_{q+1}^{2- \gamma_{2}}} + \gamma_{2} D^{2}(j,q) \left( \Id - \frac{ \mathring{R}_{q,j} }{F_{2} (\tau_{q+1} j, q, L)} \right) - \frac{\gamma_{2}}{2} \sum_{k} \Id a_{k,j}^{2} \right]
\end{align*}
(see Equation \eqref{est 459} for details).  Thus, to cancel out the previous error of $\frac{\mathring{R}_{q,j}}{\lambda_{q+1}^{2- \gamma_{2}}}$ with $- \gamma_{2} D^{2} (j,q) \frac{ \mathring{R}_{q,j}}{F_{2}(\tau_{q+1} j, q, L)}$, we must choose 
\begin{equation*}
D(j,q) = \sqrt{\frac{F_{2}(\tau_{q+1} j, q, L)}{\gamma_{2} \lambda_{q+1}^{2-\gamma_{2}}}}
\end{equation*} 
which implies  
\begin{align}
D(j,q)  = \sqrt{\frac{F_{2}(\tau_{q+1} j, q, L)}{\gamma_{2} \lambda_{q+1}^{2-\gamma_{2}}}} \overset{\eqref{est 435}}{\gtrsim} F_{1}(\tau_{q+1} j) \lambda_{q+1}^{-1+ \frac{\gamma_{2}}{2}}.\label{est 437} 
\end{align}
Consequently, \eqref{est 436} implies that $a_{k,j}(t,x)$ would have the same lower bound, modulus constants, for all $t \in \supp \chi_{j}$; in turn, this implies that the perturbation $w_{q+1}$ from \eqref{est 164} would have the same lower bound as well. This, together with \eqref{est 431a} implies 
\begin{equation}\label{est 438}
\lambda_{q+1}^{-\beta} F_{1}(t, L) \overset{\eqref{est 35}}{=} \delta_{q+1}^{\frac{1}{2}} F_{1}(t, L) \overset{\eqref{est 431a}}{\gtrsim} w_{q+1} (t) \gtrsim  F_{1}(\tau_{q+1} j, L) \lambda_{q+1}^{-1 + \frac{\gamma_{2}}{2}}
\end{equation}  
which seems to be ``almost possible'' considering the fact that $\beta > 1 - \frac{\gamma_{2}}{2}$ due to \eqref{est 315b}. If, in \eqref{est 431b}, we use the typical upper bound of 
\begin{equation*}
\lVert \mathring{R}_{q} \rVert_{C_{t,x,q}} \leq \epsilon_{\gamma} M_{0}(t) \lambda_{q+1}^{2-\gamma_{2}} \delta_{q+1} 
\end{equation*} 
for $M_{0}(t)$ to be defined in \eqref{est 326},  then the difficulty described in \eqref{est 438} cannot be overcome. In the case of an additive noise, \cite[Equation (90)]{Y23a}, in order to close all the necessary estimates, the author chose instead 
\begin{equation}\label{est 460}
\lVert \mathring{R}_{q} \rVert_{C_{t,x,q}} \leq \bar{C} L M_{0}(t) \lambda_{q+1}^{2-\gamma_{2}} \delta_{q+1} 
\end{equation} 
for a carefully selected constant $\bar{C} > 0$ (see \cite[Equation (85)]{Y23a}) and required $L > 0$ to satisfy 
\begin{equation}\label{est 461}
a^{-b (1- \frac{\gamma_{2}}{2} - \beta)}  \lesssim L^{\frac{1}{2}}
\end{equation} 
in \cite[Equation (94b)]{Y23a} where the exponent ``$1 - \frac{\gamma_{2}}{2} - \beta$'' is negative (see also Equation \eqref{est 315b}). 

Yet, tracing these computations in the case of the linear multiplicative noise shows that the problem is more severe and surprisingly cannot be overcome by any extension of the approach described in \eqref{est 460}-\eqref{est 461}. First, in contrast to \eqref{est 434}, we now obtain  
\begin{equation*}
\Upsilon (\nabla y_{0})^{T} \cdot \Lambda^{2-\gamma_{2}} y_{0} \neq 0. 
\end{equation*} 
As we will see in \eqref{est 324}, we may assume the bound of $\lvert B(t) \rvert \leq L^{\frac{1}{4}}$ over $[-2, T_{L}]$ and therefore, we now have 
\begin{equation*}
\epsilon_{\gamma} F_{2}(t,0, L) \geq \lVert \mathring{R}_{0} \rVert_{C_{t,x,0}} \geq C_{1} e^{L^{\frac{1}{4}}} F_{1}(t, L)^{2}
\end{equation*}
for some $C_{1} \geq 0$ due to \eqref{est 431},  in contrast to \eqref{est 435} which leads to 
\begin{align}
D(j,q)   \gtrsim e^{\frac{1}{2} L^{\frac{1}{4}}} F_{1}(\tau_{q+1} j, L) \lambda_{q+1}^{-1 + \frac{\gamma_{2}}{2}} 
\end{align}
in contrast to \eqref{est 437}. This leads ultimately to, in contrast to \eqref{est 438}, 
\begin{equation}
\lambda_{q+1}^{-\beta} F_{1}(t, L)  \overset{\eqref{est 431a}}{\gtrsim} w_{q+1} \gtrsim  e^{\frac{1}{2} L^{\frac{1}{4}}} F_{1}(\tau_{q+1} j, L) \lambda_{q+1}^{-1 + \frac{\gamma_{2}}{2}}.
\end{equation} 
Informally, any bound in terms of $L$ should already be part of $F_{1}(t, L)$ and thus this presents a gap that is more significant than \eqref{est 438} and cannot be closed by simply carefully choosing a constant and introducing some function of $L$ and asking it to cover the gap between the exponents $1 - \frac{\gamma_{2}}{2}$ and $\beta$. 

An immediate next idea is that perhaps a shear flow $y_{0}(t,x) \approx F_{1}(t,L) 
\begin{pmatrix}
\sin(x_{2})\\
0
\end{pmatrix}$ in \eqref{est 432} at the step $q = 0$ is no longer the correct choice and we need an alternative choice that satisfies $(\nabla y_{0})^{T} \cdot \Lambda^{2-\gamma_{2}} y_{0} = 0$ in contrast to \eqref{est 434}. However, we could not come up with a good candidate for this endeavor. Of course, we can write the two nonlinear terms together via \eqref{est 38} as $(\Lambda^{2-\gamma_{2}} y_{0})^{\bot} (\nabla^{\bot} \cdot y_{0})$; unfortunately, for this to vanish, it seems that $y_{0}$ must be curl-free and this, in addition to the divergence-free condition that is already imposed from \eqref{est 320b}, seemed too restrictive. After attempting multiple other ideas all of which failed, at the time of writing \cite{Y23a} we felt that this difficulty due to the nonlinear term in non-divergence form requires new approaches and chose to reconsider it in a future work. 

The second major difficulty is the singularity of the nonlinear term that now includes the operator  $\Lambda^{2-\gamma_{2}}$, where $\gamma_{2} \in [1, 2)$ for generality (see \eqref{est 315a}), in contrast to other PDEs such as the Navier-Stokes equations. In \eqref{est 345e} we will define the second stochastic commutator error of 
\begin{align}\label{est 462}
\divergence R_{\text{Com2}} = (\Upsilon - \Upsilon_{l}) [(\Lambda^{2- \gamma_{2}} y_{q+1} \cdot \nabla) y_{q+1} - (\nabla y_{q+1})^{T} \cdot \Lambda^{2-\gamma_{2}} y_{q+1} ]; 
\end{align}
we refer to $R_{\text{Com2}}$ as the stochastic commutator error because such an error is absent in the deterministic case. In all the previous works, an analogous estimate on $R_{\text{Com2}}$ has always been relatively easy, consisting of proving an analog of Hypothesis \ref{Hypothesis 5.1} \ref{Hypothesis 5.1 (a)} at level $q+1$ first and then a straight-forward mollifier estimate; e.g. we recall the estimate of 
\begin{align}
\lVert R_{\text{Com2}} \rVert_{C_{t,q+1}L_{x}^{1}} =& \lVert (\Upsilon_{l} - \Upsilon) (v_{q+1} \mathring{\otimes} v_{q+1}) \rVert_{C_{t, q+1}} \nonumber \\
\lesssim& l^{\frac{1}{2} - 2\delta} m_{L}^{2} \lVert v_{q+1} \rVert_{C_{t,q+1}L_{x}^{2}}^{2} \lesssim l^{\frac{1}{2} - 2 \delta} m_{L}^{4} M_{0}(t) \leq \frac{M_{0}(t) c_{R} \delta_{q+2}}{5}  \label{est 479}
\end{align}
from \cite[Equation (173)]{Y20a} in the case of the Navier-Stokes equations. Surprisingly, an analogous attempt fails in the case of the momentum SQG equations as follows: for any $\delta \in (0, \frac{1}{4})$, we only obtain 
\begin{align}
\lVert R_{\text{Com2}} \rVert_{C_{t,x,q+1}} \overset{\eqref{est 38}}{=}&\lVert \mathcal{B} \left( ( \Upsilon - \Upsilon_{l}) \Lambda^{2- \gamma_{2}} y_{q+1}^{\bot} (\nabla^{\bot} \cdot y_{q+1}) \right) \rVert_{C_{t,x,q+1}}  \nonumber \\
\lesssim&\lVert \Upsilon - \Upsilon_{l} \rVert_{C_{t,q+1}} \lVert \Lambda^{2- \gamma_{2}} y_{q+1}\rVert_{C_{t,x,q+1}} \lambda_{q+1}^{\gamma_{2} -1} \lVert \Lambda^{2-\gamma_{2}} y_{q+1} \rVert_{C_{t,x,q+1}} \nonumber \\
\overset{\eqref{est 324} \eqref{est 317}\eqref{est 32}}{\lesssim}& \lambda_{q+1}^{-\alpha (\frac{1}{2} - 2 \delta)} m_{L}^{4}M_{0}(t)^{\frac{3}{2}}  \lambda_{q+1}^{3- \gamma_{2}} \delta_{q+1} \label{est 346} 
\end{align} 
which is far too large. Indeed, to bound the right hand side of \eqref{est 346} by an appropriate inductive hypothesis of $\lVert \mathring{R}_{q+1} \rVert_{C_{t,x,q+1}}$, we found out that we will need $\alpha$ larger than the limit of range that is allowed from other estimates. Specifically, ultimately we will optimize over all other constraints and choose $\alpha = 1 + \frac{\beta}{2}$ in \eqref{est 315c}; yet, we realize that while we will need the exponent of 
\begin{align*}
- \left(1+ \frac{\beta}{2} \right)\left( \frac{1}{2} - 2 \delta \right) + 3 - \gamma_{2} - 2 \beta
\end{align*}
from $\lambda_{q+1}^{-\alpha (\frac{1}{2} - 2 \delta)} \lambda_{q+1}^{3-\gamma_{2}} \delta_{q+1}$ in \eqref{est 346} to be negative, as $\beta > 1 - \frac{\gamma_{2}}{2}$ will be taken to be very close to $1- \frac{\gamma_{2}}{2}$,  even for $\delta \in (0, \frac{1}{4})$ arbitrary small, we have at best 
\begin{align*}
-\left( 1 + \frac{\beta}{2} \right) \left(\frac{1}{2} - 2 \delta \right) + 3 - \gamma_{2} - 2 \beta \approx \frac{1}{4} + \frac{\gamma_{2}}{8} \geq \frac{3}{8}. 
\end{align*}
We emphasize again that the difficulty here stems from the nonlinear term $(\nabla y_{q+1})^{T} \cdot \Lambda^{2-\gamma_{2}} y_{q+1}$ in \eqref{est 462} that is not in a divergence form, forcing $\mathcal{B}$ to go to waste in the first inequality of \eqref{est 346}, in sharp contrast to \eqref{est 479} where the $R_{\text{Com2}}$ in the case of the Navier-Stokes equations does not even have the divergence operator to start with. 

We summarize why this second difficulty due to $\Lambda^{2-\gamma_{2}}$ within the nonlinear terms never appeared in previous works. 
\begin{enumerate} 
\item First, such difficulty never existed in the case of the Navier-Stokes equations forced by linear multiplicative noise (see Equation \eqref{est 479}) because the nonlinear term therein did not have derivatives such as $\Lambda^{2-\gamma_{2}}$ in our current $R_{\text{Com2}}$ and it was in a divergence form allowing one to take full advantage of $\mathcal{B}$. 
\item Second, such difficult terms did not exist even in the case of the SQG equations forced by additive noise (e.g. see \cite[Equations (245)--(246)]{Y23a}) because in the case of additive noise, within the product of 
\begin{align*}
(\Lambda^{2-\gamma_{2}} f \cdot \nabla)g - (\nabla f)^{T} \cdot \Lambda^{2-\gamma_{2}} g, 
\end{align*}
$f$ or $g$ was the difference of the Ornstein-Uhlenbeck process; e.g. 
\begin{align*}
\text{``}R_{\text{Com2,1}} \triangleq \mathcal{B} ( \Lambda^{2-\gamma_{2}} y_{q+1} \cdot \nabla (z_{q+1} - z_{q}) + \Lambda^{2-\gamma_{2}} (z_{q+1} - z_{q}) \cdot \nabla y_{q+1} )\text{''}
\end{align*}
from \cite[Equation (246a)]{Y23a}. In the case of the linear multiplicative noise, the \emph{difference} within $R_{\text{Com2}}$ of \eqref{est 462} is $(\Upsilon - \Upsilon_{l})$ which can only provide us $l^{\frac{1}{2} - 2 \delta} = \lambda_{q+1}^{\alpha(\frac{1}{2} - 2 \delta)}$ for $\delta \in (0, \frac{1}{4})$ while both spatial derivatives of $\Lambda^{2-\gamma_{2}}$ and $\nabla$ fall on $y_{q+1}$ which cost $\lambda_{q+1}^{2-\gamma_{2}}$ and $\lambda_{q+1}$, respectively. This makes the $R_{\text{Com2}}$ in \eqref{est 462} much more singular than the  $R_{\text{Com2}}$ in the case of additive noise. 
\item Lastly, this difficulty did not exist in the convex integration scheme for the deterministic SQG equations in \cite{BSV19} because the stochastic commutator error $R_{\text{Com2}}$ is an error due to the exponential Brownian motion $\Upsilon(t)$; i.e., $B \equiv 0$ turns $R_{\text{Com2}} \equiv 0$.  
\end{enumerate} 
\end{remark} 

In this work we overcome these major difficulties and succeed on employing convex integration on \eqref{est 30}; we present our main results in the next section. 

\section{Statement of main results}
\subsection{Statements of Theorems \ref{Theorem 2.1}-\ref{Theorem 2.2}}
First, given any deterministic initial data $v^{\text{in}} \in \dot{H}^{\frac{1}{2}}(\mathbb{T}^{2})$, via a Galerkin approximation one can construct a solution $v$ to \eqref{est 30} with $\gamma_{2} = 1$ that satisfies the inequality of 
\begin{equation}\label{est 37} 
\mathbb{E}^{\textbf{P}} [\lVert v(t) \rVert_{\dot{H}_{x}^{\frac{1}{2}}}^{2}] \leq e^{t} \lVert v^{\text{in}} \rVert_{\dot{H}_{x}^{\frac{1}{2}}}^{2} \hspace{3mm} \forall \hspace{1mm} t \geq 0.
\end{equation} 

\begin{theorem}\label{Theorem 2.1} 
Suppose that $\gamma_{1} \in (0, \frac{3}{2})$ and $\gamma_{2} = 1$. Then, given any $T> 0, K > 1$, and $\kappa \in (0,1)$, there exist a sufficiently small $\iota > 0$ and a $\textbf{P}$-a.s. strictly positive stopping time $\mathfrak{t}$ such that
\begin{equation}\label{est 148}
\textbf{P} (\{\mathfrak{t} \geq T \}) > \kappa 
\end{equation} 
and the following is additionally satisfied. There exists an $(\mathcal{F}_{t})_{t\geq 0}$-adapted process $v$ that is a weak solution to \eqref{est 30} starting from a deterministic initial data $v^{\text{in}}$ that satisfies  
\begin{equation}\label{est 9}
\esssup_{\omega \in \Omega} \lVert v(\omega) \rVert_{C_{\mathfrak{t}} C_{x}^{\frac{1}{2} + \iota}} + \esssup_{\omega \in \Omega} \lVert v(\omega) \rVert_{C_{\mathfrak{t}}^{\eta}C_{x}}  < \infty \hspace{3mm} \forall \hspace{1mm} \eta \in \left(0, \frac{1}{3}\right], 
\end{equation} 
and on the set $\{ \mathfrak{t} \geq T \}$,  
\begin{equation}\label{est 313} 
\lVert v(T) \rVert_{\dot{H}_{x}^{\frac{1}{2}}}  > K e^{\frac{T}{2}} \lVert v^{\text{in}} \rVert_{\dot{H}_{x}^{\frac{1}{2}}}. 
\end{equation} 
\end{theorem} 

\begin{theorem}\label{Theorem 2.2} 
Suppose that  $\gamma_{1} \in (0, \frac{3}{2})$ and $\gamma_{2} = 1$. Then non-uniqueness in law holds for \eqref{est 30} on $[0,\infty)$. Moreover, for all $T > 0$ fixed, non-uniqueness in law holds for \eqref{est 30} on $[0,T]$. 
\end{theorem} 
\noindent Previous applications of convex integration technique on the stochastic SQG equations, namely \cite{BLW23, HZZ22a, HLZZ23, Y23a}, were in the case of an additive noise. To the best of the author's knowledge, Theorems \ref{Theorem 2.1}-\ref{Theorem 2.2} represent first non-uniqueness results in the case of linear multiplicative noise. 

\subsection{New ideas and novelties to overcome the difficulties}
Let us briefly describe new ideas and novelties of the proof of Theorems \ref{Theorem 2.1}-\ref{Theorem 2.2} that were needed to overcome the difficulties described in Remark \ref{Remark 1.3}. 

Concerning the first difficulty of $(\nabla y_{0})^{T} \cdot \Lambda^{2-\gamma_{2}} y_{0}$ in \eqref{est 434}, even though this term certainly does not vanish, remarkably, it turns out that when $y_{0}$ is a shear flow, this problematic term can be written in a gradient form. This allows us to use the unique loophole of convex integration schemes, namely to choose pressure to absorb it so that the Reynolds stress $\mathring{R}_{0}$ does not need to take this term into account (see Equations \eqref{est 322} and \eqref{est 444}). Consequently, $\mathring{R}_{0}$ is freed from the restriction of the lower bound that we saw in \eqref{est 478}.   

Concerning the second issue of the stochastic commutator error $R_{Com2}$ in \eqref{est 462}, our idea is to first use \eqref{est 164} to rewrite it as 
\begin{subequations}\label{est 389}
\begin{align}
\divergence R_{\text{Com2}} \overset{\eqref{est 462}\eqref{est 164}}{=}&  (\Upsilon - \Upsilon_{l}) [ \Lambda^{2- \gamma_{2}} w_{q+1} \cdot \nabla w_{q+1} - (\nabla w_{q+1})^{T} \cdot \Lambda^{2- \gamma_{2}} w_{q+1}  \label{est 389a} \\
&+ \Lambda^{2- \gamma_{2}} w_{q+1} \cdot \nabla y_{l}  - (\nabla w_{q+1})^{T} \cdot \Lambda^{2- \gamma_{2}} y_{l}    \label{est 389b}\\
&+ \Lambda^{2- \gamma_{2}} y_{l} \cdot \nabla w_{q+1} - (\nabla y_{l})^{T} \cdot \Lambda^{2- \gamma_{2}} w_{q+1}    \label{est 389c} \\
&+ \Lambda^{2- \gamma_{2}} y_{l} \cdot \nabla y_{l} - (\nabla y_{l})^{T} \cdot \Lambda^{2- \gamma_{2}} y_{l} ].  \label{est 389d}
\end{align}
\end{subequations}
First, for the products of $w_{q+1}$ and $y_{l}$ in \eqref{est 389b}-\eqref{est 389c}, we can equivalently consider them with $\tilde{P}_{\approx \lambda_{q+1}}$ (see Section \ref{Section 5.1} for precise definition) applied in consideration of their Fourier support from Hypothesis \ref{Hypothesis 5.1}(a) and \eqref{est 163} and thereby gain $\lambda_{q+1}^{-1}$ from $\mathcal{B}$. Second, the products of $y_{l}$ with itself in \eqref{est 389d} are less of a concern due to a lack of singularity in $y_{q}$ in comparison to $y_{q+1}$. Finally, the remaining products of $w_{q+1}$ with itself in \eqref{est 389a} are the most singular and problematic. Upon a closer look, it turns out that these terms are very similar to the usual most problematic oscillation error $R_{O}$ from \eqref{est 345b}. Therefore, we subtract and add back $\mathring{R}_{l}$ and rewrite $R_{\text{Com2}}$ from \eqref{est 389} as follows: 
\begin{equation}\label{est 390}
R_{\text{Com2}} \overset{\eqref{est 389} \eqref{est 38}}{=} \sum_{k=1}^{3}  R_{\text{Com2,k}} 
\end{equation} 
where 
\begin{subequations}\label{est 391}
\begin{align}
& R_{\text{Com2,1}} \triangleq - (\Upsilon - \Upsilon_{l}) \Upsilon_{l}^{-1}  \mathring{R}_{l}, \label{est 391a}\\
& R_{\text{Com2,2}} \triangleq (\Upsilon - \Upsilon_{l}) \Upsilon_{l}^{-1}  R_{O},\label{est 391b} \\
& R_{\text{Com2,3}} \triangleq  (\Upsilon - \Upsilon_{l})  \mathcal{B}[\Lambda^{2-\gamma_{2}} w_{q+1}^{\bot} \nabla^{\bot} \cdot y_{l} + \Lambda^{2-\gamma_{2}} y_{l}^{\bot} \nabla^{\bot} \cdot w_{q+1} + \Lambda^{2-\gamma_{2}} y_{l}^{\bot} \nabla^{\bot} \cdot y_{l} ].\label{est 391c}
\end{align}
\end{subequations}
In contrast to typical convex integration schemes where the importance of canceling out the previous error $\mathring{R}_{l}$ is always emphasized, we are now left with $\mathring{R}_{l}$ in $R_{\text{Com2,1}}$ of \eqref{est 391a}. Nevertheless, it turns out that the difference of $\Upsilon - \Upsilon_{l}$ that can provide us with $l^{\frac{1}{2} - 2 \delta}$ is enough to handle this term. As a result, we will estimate the oscillation error $\lVert R_{O} \rVert_{C_{t,x,q+1}}$ in \eqref{est 307} and effectively make use of it twice, the second time in \eqref{est 394} to estimate $R_{\text{Com2,2}}$ that is defined via \eqref{est 391b}. This approach seems to be new. Let us estimate all of $R_{\text{Com2,k}}$ for $k \in \{1,2,3\}$ separately and carefully in Section \ref{Section 5.2.6}.

\begin{remark}\label{Remark 2.1} 
Additionally, we continue to adopt all the new ideas from \cite{Y23a} to adjust the deterministic convex integration scheme of \cite{BSV19} to the stochastic case as follows. 
\begin{enumerate} 
\item We replace the ``$\rho_{j}$'' which is part of the definition of the amplitude function $a_{k,j}$ in \cite[Equations (4.14)-(4.15)]{BSV19} by $\delta_{q+1}$. This simplification allows us to drop both inductive hypotheses \cite[Equations (3.8)--(3.9)]{BSV19} concerning the prescribed energy $e(t)$. 
\item We replace $\mathring{R}_{q}(\tau_{q+1} j, x)$ in \cite[Equation (4.12b)]{BSV19} by $\mathring{R}_{l}(\tau_{q+1} j, x)$; this is crucial allowing us to drop one of the inductive hypothesis \eqref{est 13} from \cite[Equation (3.7)]{BSV19} because the Reynolds stress in the stochastic case cannot be differentiable in time (see Equation \eqref{est 383} and the discussion thereafter). It also allows us to drop the inductive hypothesis \cite[Equation (3.6)]{BSV19} (see \cite[Section 2.2]{Y23a} for details). 
\item Finally, we define differently from \cite[Equation (136)]{Y23a}, 
\begin{equation}\label{est 349} 
D_{t,q} \triangleq \partial_{t} + \Upsilon_{l} \Lambda^{2-\gamma_{2}} y_{l} \cdot \nabla 
\end{equation}
to handle the Reynolds stress transport error. The multiplication by ``$\Upsilon_{l}$'' in \eqref{est 349} is crucial upon employing the $C_{x}^{\alpha}$-type convex integration schemes to stochastic PDEs forced by linear multiplicative noise, as the author discovered in \cite[Equation (195)]{Y21c}. On the other hand,  ``$\Lambda^{2-\gamma_{2}}$'' in \eqref{est 349} is new and important considering the fact that the nonlinear term of the momentum SQG equations is more singular than those of many other PDEs on which convex integration was successfully employed. 
\end{enumerate} 
\end{remark}

We introduce a minimum amount of notations that are immediately necessary in Section \ref{Section 3}, leaving the rest in the Appendix for completeness. We prove Theorem \ref{Theorem 2.2} assuming Theorem \ref{Theorem 2.1} in Section \ref{Section 4}, and prove Theorem \ref{Theorem 2.1} in Section \ref{Section 5}.  

\section{Preliminaries}\label{Section 3}
Let us denote the spatial mean of $f$ by $\fint_{\mathbb{T}^{2}} f dx \triangleq \lvert \mathbb{T}^{2} \rvert^{-1} \int_{\mathbb{T}^{2}} f(x) dx$ and then $C_{0}^{\infty} \triangleq \{f \in C^{\infty} (\mathbb{T}^{2}): \fint_{\mathbb{T}^{2}} f dx = 0 \}$ while $C_{0,\sigma}^{\infty} \triangleq \{f \in C_{0}^{\infty} (\mathbb{T}^{2}): \nabla\cdot f = 0 \}$. We denote the Sobolev spaces by $\dot{H}_{\sigma}^{s} \triangleq \{f \in \dot{H}^{s}(\mathbb{T}^{2}): \fint_{\mathbb{T}^{2}} f dx = 0, \nabla\cdot f = 0 \}$ and $H_{\sigma}^{s}$ similarly for any $s \in \mathbb{R}$. We let $\mathbb{P} \triangleq \Id + \mathcal{R} \otimes \mathcal{R}$ represent the Leray projection onto the space of divergence-free vector fields. We define $\lVert \cdot \rVert_{C_{t,x}^{N}} \triangleq \sum_{0 \leq k + \lvert \alpha \rvert \leq N} \lVert \partial_{t}^{k} D^{\alpha} \cdot \rVert_{L_{t,x}^{\infty}}$ where $k \in \mathbb{N}_{0}$ and $\alpha$ is a multi-index. For all $p \in [1,\infty]$, $\iota_{1} \in \mathbb{N}_{0}$, and $\iota_{2} \in \mathbb{R}$ and any sufficiently smooth function $f$, we define the following quantities:
\begin{subequations}\label{est 00}
\begin{align}
& \lVert f \rVert_{C_{t,x,q}^{\iota_{1}}} \triangleq \lVert f \rVert_{C_{[t_{q}, t], x}^{\iota_{1}}} \triangleq \sup_{s \in [t_{q}, t]} \sum_{0 \leq j + \lvert \alpha \rvert \leq \iota_{1}} \lVert \partial_{s}^{j} D^{\alpha} f(s) \rVert_{C_{x}}, \hspace{2mm} \lVert f \rVert_{C_{t,q} L_{x}^{p}} \triangleq \sup_{s \in [t_{q}, t]} \lVert f(s) \rVert_{L_{x}^{p}} \\
&\lVert f \rVert_{C_{t,q}^{j} C_{x}^{\iota_{1}}} \triangleq \sum_{0\leq k \leq j, 0 \leq \lvert \alpha \rvert \leq \iota_{1}} \lVert \partial_{s}^{k} D^{\alpha} f(s) \rVert_{C_{[t_{q}, t], x}}, \hspace{15mm} \lVert f \rVert_{C_{t,q} H_{x}^{\iota_{2}}} \triangleq \sup_{s \in [t_{q}, t]} \lVert f(s) \rVert_{H_{x}^{\iota_{2}}}, 
\end{align}
\end{subequations} 
where $t_{q}$ was defined in \eqref{est 463}. 

Given any Polish space $H$, we let $\mathcal{B}(H)$ denote the $\sigma$-algebra of Borel sets in $H$. We denote the $L^{2}(\mathbb{T}^{2})$-inner products, a duality pairing of $\dot{H}^{-\frac{1}{2}}(\mathbb{T}^{2})-\dot{H}^{\frac{1}{2}}(\mathbb{T}^{2})$, and a quadratic variation of $A$ and $B$, respectively by $\langle \cdot, \cdot \rangle$, $\langle \cdot, \cdot \rangle_{\dot{H}_{x}^{-\frac{1}{2}}-\dot{H}_{x}^{\frac{1}{2}}}$, 
and $\langle\langle A, B \rangle \rangle$. Additionally, we define $\langle \langle A \rangle \rangle \triangleq \langle \langle A, A \rangle \rangle$. For any Hilbert space $U$, we denote by $L_{2}(U, \dot{H}_{\sigma}^{s})$ with $s \in \mathbb{R}_{+}$ the space of all Hilbert-Schmidt operators from $U$ to $\dot{H}_{\sigma}^{s}$ with norm $\lVert \cdot \rVert_{L_{2}(U, \dot{H}_{\sigma}^{s})}$. We impose on $G: \dot{H}_{\sigma}^{1- \frac{\gamma_{2}}{2}} \mapsto L_{2} (U, \dot{H}_{\sigma}^{1 - \frac{\gamma_{2}}{2}})$ to be $\mathcal{B} (\dot{H}_{\sigma}^{1- \frac{\gamma_{2}}{2}})/\mathcal{B} (L_{2} (U, \dot{H}_{\sigma}^{1- \frac{\gamma_{2}}{2}}))$-measurable and satisfy 
\begin{equation}\label{est 60} 
\lVert G(\phi ) \rVert_{L_{2}(U, \dot{H}_{\sigma}^{1- \frac{\gamma_{2}}{2}})} \leq C(1+ \lVert \phi \rVert_{\dot{H}_{x}^{1- \frac{\gamma_{2}}{2}}}), \hspace{3mm} \lim_{j\to\infty} \lVert G(\psi^{j})^{\ast} \phi - G(\psi)^{\ast} \phi \rVert_{U} = 0 
\end{equation} 
for all $\phi, \psi^{j}, \psi \in C^{\infty} (\mathbb{T}^{2}) \cap \dot{H}_{\sigma}^{1- \frac{\gamma_{2}}{2}}$ such that $\lim_{j\to\infty} \lVert \psi^{j} - \psi \rVert_{\dot{H}_{x}^{1- \frac{\gamma_{2}}{2}}} = 0$. We assume the existence of another Hilbert space $U_{1}$ such that the embedding $U \subset U_{1}$ is Hilbert-Schmidt and define 
\begin{equation*} 
\bar{\Omega} \triangleq C([0,\infty); (H_{\sigma}^{4})^{\ast} \times U_{1}) \cap L_{\text{loc}}^{\infty} ([0,\infty); \dot{H}_{\sigma}^{1- \frac{\gamma_{2}}{2}} \times U_{1}) 
\end{equation*} 
where $(H_{\sigma}^{4})^{\ast}$ denotes the dual of $H_{\sigma}^{4}$. Additionally, we let $\mathcal{P} (\bar{\Omega})$ represent the set of all probability measures on $(\bar{\Omega}, \bar{\mathcal{B}})$ where $\bar{\mathcal{B}}$ is the Borel $\sigma$-algebra on $\bar{\Omega}$. Furthermore, we define the canonical process on $\bar{\Omega}$ by $(\xi, \theta): \bar{\Omega} \mapsto (H_{\sigma}^{4})^{\ast} \times U_{1}$ to satisfy $(\xi, \theta)_{t}(\omega) \triangleq \omega(t)$. Finally, for $t \geq 0$ we define 
\begin{equation*}
\bar{\mathcal{B}}^{t} \triangleq \sigma \{ (\xi, \theta)(s): s \geq t \}, \hspace{3mm} \bar{\mathcal{B}}_{t}^{0} \triangleq \sigma \{(\xi, \theta)(s): s \leq t \}, \hspace{3mm} \bar{\mathcal{B}}_{t} \triangleq \cap_{s > t} \bar{\mathcal{B}}_{s}^{0}. 
\end{equation*} 

\section{Proof of Theorem \ref{Theorem 2.2}}\label{Section 4}
Throughout this Section \ref{Section 4} we only consider the case $\gamma_{2} = 1$ in \eqref{est 30} because Proposition \ref{Proposition 4.1} can be proven only for this special case. We start with the definitions of a solution to 
\begin{subequations}\label{est 480} 
\begin{align} 
&dv + [(u\cdot\nabla) v - (\nabla v)^{T} \cdot u + \nabla p + \Lambda^{\gamma_{1}} v]dt = G(v)dD, \\
&u = \Lambda v,  \hspace{2mm} \nabla\cdot v = 0, 
\end{align}
\end{subequations}
represents a stochastic force because the materials presented in Definitions \ref{Definition 4.1}-\ref{Definition 4.2}, Proposition \ref{Proposition 4.1}, and Lemmas \ref{Lemma 4.2}-\ref{Lemma 4.3} apply to a general infinite-dimensional stochastic perturbation of the system \eqref{est 30}; its special case includes $G(v) dD = vdB$ in \eqref{est 30}. 

\begin{define}\label{Definition 4.1}
Fix $\iota \in (0,1)$. Let $s \geq 0$, $\xi^{\text{in}} \in \dot{H}_{\sigma}^{\frac{1}{2}}$, and $\theta^{\text{in}} \in U_{1}$. Then $P \in \mathcal{P} (\bar{\Omega})$ is a probabilistically weak solution to \eqref{est 480} with initial data $(\xi^{\text{in}}, \theta^{\text{in}})$ at initial time $s$ if 
\begin{enumerate}[wide=0pt]
\item [](M1) $P (\{ \xi(t) = \xi^{\text{in}}, \theta(t) = \theta^{\text{in}} \hspace{1mm} \forall \hspace{1mm} t \in [0,s] \}) = 1$ and for all $l \in \mathbb{N}$, 
\begin{equation*}
P \left( \left\{ (\xi,\theta) \in \bar{\Omega}: \int_{0}^{l} \lVert G(\xi(r)) \rVert_{L_{2} (U, \dot{H}_{\sigma}^{\frac{1}{2}})}^{2} dr < \infty \right\} \right) = 1, 
\end{equation*} 
\item [](M2) under $P$, $\theta$ is a cylindrical $(\bar{\mathcal{B}}_{t})_{t\geq s}$-Wiener process on $U$ starting from initial data $\theta^{\text{in}}$ at initial time $s$ and for all $\psi^{k} \in C^{\infty} (\mathbb{T}^{2}) \cap \dot{H}_{\sigma}^{\frac{1}{2}}$ and $t\geq s$,  
\begin{align}
&\langle \xi(t) - \xi(s), \psi^{k} \rangle - \int_{s}^{t} \sum_{i,j=1}^{2} \langle \Lambda \xi_{i}, \partial_{i} \psi_{j}^{k} \xi_{j} \rangle_{\dot{H}_{x}^{-\frac{1}{2}} -\dot{H}_{x}^{\frac{1}{2}}} \nonumber \\
& \hspace{20mm} - \frac{1}{2} \langle \partial_{i} \xi_{j}, [\Lambda,  \psi_{i}^{k}] \xi_{j} \rangle_{\dot{H}_{x}^{-\frac{1}{2}}-\dot{H}_{x}^{\frac{1}{2}}} - \langle \xi, \Lambda^{\gamma_{1}} \psi^{k} \rangle dr = \int_{s}^{t} \langle \psi^{k}, G(\xi(r)) d\theta(r) \rangle, \label{est 2}
\end{align}
\item [](M3) for any $q \in \mathbb{N}$, there exists a function $t \mapsto C_{t,q} \in\mathbb{R}_{+}$ for all $t \geq s$ such that
\begin{equation}\label{est 70}
\mathbb{E}^{P} \left[ \sup_{r \in [0,t]} \lVert \xi(r) \rVert_{\dot{H}_{x}^{\frac{1}{2}}}^{2q} + \int_{s}^{t} \lVert \xi(r) \rVert_{\dot{H}_{x}^{\frac{1}{2} + \iota}}^{2} dr\right] \leq C_{t,q} (1+ \lVert \xi^{\text{in}} \rVert_{\dot{H}_{x}^{\frac{1}{2}}}^{2q}). 
\end{equation} 
\end{enumerate}
The set of all such probabilistically weak solutions with the same constant $C_{t,q}$ in \eqref{est 70} for every $q \in \mathbb{N}$ and $t\geq s$ will be denoted by $\mathcal{W} ( s, \xi^{\text{in}}, \theta^{\text{in}}, \{C_{t,q} \}_{q\in\mathbb{N}, t \geq s})$.  
\end{define} 

\begin{define}\label{Definition 4.2}
Fix $\iota \in (0,1)$. Let $s \geq 0$, $\xi^{\text{in}} \in \dot{H}_{\sigma}^{\frac{1}{2}}$, and $\theta^{\text{in}} \in U_{1}$. Let $\tau \geq s$ be a $(\bar{\mathcal{B}}_{t})_{t\geq s}$-stopping time and set 
\begin{equation*}
\bar{\Omega}_{\tau} \triangleq \{\omega( \cdot \wedge \tau(\omega)): \omega \in \bar{\Omega} \} = \{ \omega \in \bar{\Omega}: (\xi, \theta) (t,\omega) = (\xi, \theta) (t\wedge \tau(\omega), \omega)  \hspace{1mm} \forall \hspace{1mm} t \geq 0\}. 
\end{equation*} 
Then $P \in \mathcal{P} (\bar{\Omega}_{\tau})$ is a probabilistically weak solution to \eqref{est 480} on $[s,\tau]$ with initial data $(\xi^{\text{in}}, \theta^{\text{in}})$ at initial time $s$ if 
\begin{enumerate}[wide=0pt]
\item [](M1) $P (\{ \xi(t) = \xi^{\text{in}}, \theta(t) = \theta^{\text{in}} \hspace{1mm} \forall \hspace{1mm} t \in [0,s] \}) = 1$ and for all $l \in \mathbb{N}$, 
\begin{equation*}
P \left( \left\{ (\xi,\theta) \in \bar{\Omega}: \int_{0}^{l \wedge \tau} \lVert G(\xi(r)) \rVert_{L_{2} (U, \dot{H}_{\sigma}^{\frac{1}{2}})}^{2} dr < \infty \right\} \right) = 1, 
\end{equation*} 
\item [](M2) under $P$, $\langle \theta (\cdot \wedge \tau), l^{i} \rangle_{U}$ where $\{l^{i}\}_{i\in\mathbb{N}}$ is an orthonormal basis of $U$, is a continuous, square-integrable $(\bar{\mathcal{B}}_{t})_{t\geq s}$-martingale with initial data $\langle \theta^{\text{in}}, l^{i}  \rangle$ at initial time $s$ with a quadratic variation process given by $(t\wedge \tau - s) \lVert l^{i} \rVert_{U}^{2}$ and for every $\psi^{k} \in C^{\infty} (\mathbb{T}^{2}) \cap \dot{H}_{\sigma}^{\frac{1}{2}}$ and $t \geq s$,   
\begin{align*}
&\langle \xi(t \wedge \tau) - \xi(s), \psi^{k} \rangle - \int_{s}^{t\wedge \tau} \sum_{i,j=1}^{2} \langle \Lambda \xi_{i}, \partial_{i} \psi_{j}^{k} \xi_{j} \rangle_{\dot{H}_{x}^{-\frac{1}{2}} -\dot{H}_{x}^{\frac{1}{2}}}  \\
& \hspace{12mm} - \frac{1}{2} \langle \partial_{i} \xi_{j}, [\Lambda,  \psi_{i}^{k}] \xi_{j} \rangle_{\dot{H}_{x}^{-\frac{1}{2}}-\dot{H}_{x}^{\frac{1}{2}}} - \langle \xi, \Lambda^{\gamma_{1}} \psi^{k} \rangle dr  = \int_{s}^{t \wedge \tau} \langle \psi^{k}, G(\xi(r)) d\theta(r) \rangle, \nonumber
\end{align*} 
\item [](M3) for any $q \in \mathbb{N}$, there exists a function $t \mapsto C_{t,q} \in\mathbb{R}_{+}$ for all $t \geq s$ such that
\begin{equation*}
\mathbb{E}^{P} \left[ \sup_{r \in [0,t \wedge \tau]} \lVert \xi(r) \rVert_{\dot{H}_{x}^{\frac{1}{2}}}^{2q} + \int_{s}^{t \wedge \tau} \lVert \xi(r) \rVert_{\dot{H}_{x}^{\frac{1}{2} + \iota}}^{2} dr\right] \leq C_{t,q} (1+ \lVert \xi^{\text{in}} \rVert_{\dot{H}_{x}^{\frac{1}{2}}}^{2q}). 
\end{equation*} 
\end{enumerate} 
\end{define} 

Let $\bar{\mathcal{B}}_{\tau}$ denote the $\sigma$-algebra associated to any given stopping time $\tau$. The following proposition can be deduced from  \cite[Proposition 4.1]{Y23a} identically to how \cite[Theorem 5.1]{HZZ19} was deduced from \cite[Theorem 3.1]{HZZ19}; we leave its proof in the Appendix for completeness.

\begin{proposition}\label{Proposition 4.1}
\begin{enumerate}
\item For every $(s, \xi^{\text{in}}, \theta^{\text{in}}) \in [0,\infty) \times \dot{H}_{\sigma}^{\frac{1}{2}} \times U_{1}$, there exists a probabilistically weak solution $P \in \mathcal{P} (\bar{\Omega})$ to \eqref{est 30} with initial data $(\xi^{\text{in}}, \theta^{\text{in}})$ at initial time $s$ according to Definition \ref{Definition 4.1}. 
\item Moreover, if there exists a family $\{(s_{l}, \xi_{l}, \theta_{l}) \}_{l \in \mathbb{N}} \subset [0,\infty) \times \dot{H}_{\sigma}^{\frac{1}{2}} \times U_{1}$ such that $\lim_{l\to\infty} \lVert (s_{l}, \xi_{l},\theta_{l}) - (s, \xi^{\text{in}},\theta^{\text{in}}) \rVert_{\mathbb{R} \times \dot{H}_{x}^{\frac{1}{2}} \times U_{1}} = 0$ and $P_{l} \in \mathcal{W} ( s_{l}, \xi_{l}, \theta_{l}, \{C_{t,q} \}_{q\in \mathbb{N}, t \geq s_{l}} )$, then there exists a subsequence $\{P_{l_{k}} \}_{k\in\mathbb{N}}$ and $P \in \mathcal{W} (s, \xi^{\text{in}},\theta^{\text{in}}, \{C_{t,q} \}_{q\in\mathbb{N}, t \geq s})$ such that $P_{l_{k}}$ converges weakly to $P$.
\end{enumerate} 
\end{proposition}
Proposition \ref{Proposition 4.1} leads to the following two lemmas due to \cite[Propositions 5.2 and 5.3]{HZZ19}. 
\begin{lemma}\label{Lemma 4.2}
\rm{(\cite[Proposition 5.2]{HZZ19})} Let $\tau$ be a bounded $(\bar{\mathcal{B}}_{t})_{t \geq 0}$-stopping time. Then, for every $\omega \in \bar{\Omega}$, there exists $Q_{\omega} \in \mathcal{P}(\bar{\Omega})$ such that for $\omega \in \{\xi(\tau) \in \dot{H}_{\sigma}^{\frac{1}{2}} \}$ 
\begin{subequations}
\begin{align}
& Q_{\omega} ( \{ \omega' \in \bar{\Omega}: \hspace{0.5mm} \hspace{1mm}  ( \xi, \theta) (t, \omega') = (\xi, \theta) (t,\omega) \hspace{1mm} \forall \hspace{1mm} t \in [0, \tau(\omega)] \}) = 1, \\
& Q_{\omega} (A) = R_{\tau(\omega), \xi(\tau(\omega), \omega), \theta(\tau(\omega), \omega)} (A) \hspace{1mm} \forall \hspace{1mm} A \in \bar{\mathcal{B}}^{\tau(\omega)}, 
\end{align} 
\end{subequations}
where $R_{\tau(\omega), \xi(\tau(\omega), \omega), \theta(\tau(\omega), \omega)} \in \mathcal{P} (\bar{\Omega})$ is a probabilistically weak solution to \eqref{est 30} with initial data $(\xi(\tau(\omega), \omega), \theta(\tau(\omega), \omega))$ at initial time $\tau(\omega)$. Moreover, for every $A \in \bar{\mathcal{B}}$, the mapping $\omega \mapsto Q_{\omega}(A)$ is $\bar{\mathcal{B}}_{\tau}$-measurable.  
\end{lemma} 

\begin{lemma}\label{Lemma 4.3}
\rm{(\cite[Proposition 5.3]{HZZ19})} Let $\tau$ be a bounded $(\bar{\mathcal{B}}_{t})_{t\geq 0}$-stopping time, $\xi^{\text{in}}  \in \dot{H}_{\sigma}^{\frac{1}{2}}$, and $P\in \mathcal{P} (\bar{\Omega})$ be a probabilistically weak solution to \eqref{est 30} on $[0,\tau]$ with initial data $(\xi^{\text{in}}, 0)$ at initial time $0$ according to Definition \ref{Definition 4.2}. Suppose that there exists a Borel set $\mathcal{N} \subset \bar{\Omega}_{\tau}$ such that $P(\mathcal{N}) = 0$ and that $Q_{\omega}$ from Lemma \ref{Lemma 4.2} satisfies for every $\omega \in \bar{\Omega}_{\tau} \setminus \mathcal{N}$ 
\begin{equation}
Q_{\omega} (\{ \omega' \in \bar{\Omega}: \hspace{1mm}  \tau(\omega') = \tau(\omega) \}) = 1. 
\end{equation} 
Then the probability measure $P\otimes_{\tau}R \in \mathcal{P} (\bar{\Omega})$ defined by 
\begin{equation}\label{est 309} 
P \otimes_{\tau} R (\cdot) \triangleq \int_{\bar{\Omega}} Q_{\omega} (\cdot) P(d \omega)
\end{equation} 
satisfies $P\otimes_{\tau}R = P$ on the $\sigma$-algebra $\sigma \{ \xi(t\wedge \tau), \theta (t\wedge \tau), t \geq 0 \}$ and it is a probabilistically weak solution to \eqref{est 30} on $[0,\infty)$ with initial data $(\xi^{\text{in}}, 0)$ at initial time $0$. 
\end{lemma} 

Now we fix an $\mathbb{R}$-valued Wiener process $B$ on $(\Omega, \mathcal{F}, \textbf{P})$ with $(\mathcal{F}_{t})_{t\geq 0}$ as its normal filtration and $U = U_{1} = \mathbb{R}$. For $\lambda \in \mathbb{N}, L > 1$, and $\delta \in (0, \frac{1}{4})$ we define 
\begin{subequations}\label{est 127} 
\begin{align}
\tau_{L}^{\lambda}(\omega) \triangleq& \inf \left\{t \geq 0:  \lvert \theta (t, \omega) \rvert > \left(L - \frac{1}{\lambda} \right)^{\frac{1}{4}} \right\} \nonumber \\
& \wedge \inf\left\{ t \geq 0: \lVert \theta(\omega) \rVert_{C_{t}^{\frac{1}{2} - 2 \delta}} > \left(L - \frac{1}{\lambda} \right)^{\frac{1}{2}} \right\} \wedge L, \label{est 127a}\\
\tau_{L}(\omega) \triangleq& \lim_{\lambda \to \infty} \tau_{L}^{\lambda}(\omega). \label{est 127b}
\end{align}
\end{subequations} 
As a Brownian path is locally H$\ddot{\mathrm{o}}$lder continuous with exponent $\alpha \in (0, \frac{1}{2})$, we can also define for $L > 1$ and $\delta \in (0, \frac{1}{4})$, 
\begin{equation}\label{est 312}
T_{L} \triangleq \inf\{t> 0: \lvert B(t) \rvert \geq L^{\frac{1}{4}} \} \wedge \inf \{t > 0: \lVert B \rVert_{C_{t}^{\frac{1}{2} - 2 \delta}} \geq L^{\frac{1}{2}} \} \wedge L. 
\end{equation} 
It follows from \cite[Lemma 3.5]{HZZ19} that $\tau_{L}^{\lambda}$ and therefore $\tau_{L}$ is a $(\bar{\mathcal{B}}_{t})_{t\geq 0}$-stopping time. For the fixed $(\Omega, \mathcal{F}, \textbf{P})$, we assume Theorem \ref{Theorem 2.1}, denote by $v$ the solution constructed by Theorem \ref{Theorem 2.1} on $[0, \mathfrak{t}]$ where $\mathfrak{t} = T_{L}$ for $L > 1$ sufficiently large. Denoting the probability law of $(v, B)$ by $\mathcal{L}(v,B)$ and fixing $P = \mathcal{L} (v, B)$, we have the following two results from previous works, mainly \cite[Propositions 5.4-5.5]{HZZ19},  by making use of the fact that $\theta(t, (v,B)) = B(t)$ for all $t \in [0, T_{L}]$ $\textbf{P}$-a.s.

\begin{proposition}\label{Proposition 4.4} 
\rm{(\cite[Proposition 5.4]{HZZ19})} Let $\tau_{L}$ be defined by \eqref{est 127b}. Then $P= \mathcal{L} (v, B)$ is a probabilistically weak solution to \eqref{est 30} on $[0, \tau_{L}]$ that satisfies Definition \ref{Definition 4.2}.
\end{proposition}

\begin{proposition}\label{Proposition 4.5}
\rm{(\cite[Proposition 5.5]{HZZ19})} Let $\tau_{L}$ be defined by \eqref{est 127b}. Then $P\otimes_{\tau_{L}}R$ defined by \eqref{est 309} corresponding to $\tau_{L}$ is a probabilistically weak solution to \eqref{est 30} on $[0,\infty)$ that satisfies Definition \ref{Definition 4.1}. 
\end{proposition}

\begin{proof}[Proof of Theorem \ref{Theorem 2.2} assuming Theorem \ref{Theorem 2.1}]
This proof is similar to previous works such as \cite{Y23a} and thus left in the Appendix for completeness. 
\end{proof} 

\section{Proof of Theorem \ref{Theorem 2.1}}\label{Section 5}
Throughout this Section \ref{Section 5}, for generality we consider the generalized SQG equations with $u = \Lambda^{2-\gamma_{2}} v$ with $\gamma_{2} \in [1, 2)$ from \eqref{est 30b} and \eqref{est 315a}.  

\subsection{Setups}\label{Section 5.1}
We define for $k \in \mathbb{S}^{1}$, 
\begin{equation}\label{est 192} 
\mathbb{P}_{q+1, k} \triangleq \mathbb{P} P_{\approx k \lambda_{q+1}}
\end{equation} 
where $P_{\approx k\lambda_{q+1}}$ is a Fourier operator with a Fourier symbol $\hat{K}_{\approx k \lambda_{q+1}}(\xi) = \hat{K}_{\approx 1} \left( \frac{ \xi}{\lambda_{q+1}} - k \right)$ and $\hat{K}_{\approx 1}$ is a smooth bump function such that $\supp \hat{K}_{\approx 1} \subset \{\xi:  \lvert \xi \rvert \leq \frac{1}{8} \}$ and $\hat{K}_{\approx 1} \rvert_{\{\xi:  \lvert \xi \rvert \leq \frac{1}{16} \}} \equiv 1$. Within the support of $\hat{K}_{\approx 1} \left( \frac{\cdot}{\lambda_{q+1}} - k \right)$, we have $\frac{7}{8} \lambda_{q+1} \leq \lvert \xi \rvert \leq \frac{9}{8} \lambda_{q+1}$ so that $\supp \widehat{ \mathbb{P}_{q+1, k} f} \subset \{ \xi: \frac{7}{8} \lambda_{q+1} \leq \lvert \xi \rvert \leq \frac{9}{8} \lambda_{q+1} \}$. It follows that for any $a, b \geq 0$ and some suitable constant $C_{a,b}$ that is independent of $\lambda_{q+1}$, 
\begin{align*}
\sup_{\xi \in\mathbb{R}^{2}} \lvert \xi \rvert^{a} \left\lvert  (-\Delta)^{\frac{b}{2}} \hat{K}_{\approx k \lambda_{q+1} } (\xi) \right\rvert \leq C_{a,b} \lambda_{q+1}^{a-b};
\end{align*}
moreover, we have 
\begin{align*}
\left\lVert \lvert x \rvert^{b} (-\Delta)^{\frac{a}{2}} K_{\approx k \lambda_{q+1}} \right\rVert_{L^{1}(\mathbb{R}^{2})} \leq C_{a,b} \lambda_{q+1}^{a-b} \hspace{3mm} \forall \hspace{1mm} a, b \in [0,2]. 
\end{align*}
Thus, for all $f$ that is $\mathbb{T}^{2}$-periodic, we can write $\mathbb{P}_{q+1, k} f(x) = K_{q+1, k} \ast f(x)$ for $K_{q+1, k}$ that satisfies for all $a, b \geq 0$, 
\begin{equation*}  
\left\lVert \lvert x \rvert^{b}(-\Delta)^{\frac{a}{2}} K_{q+1, k} (x) \right\rVert_{L^{1}(\mathbb{R}^{2})}  \leq C_{a,b} \lambda_{q+1}^{a-b}.
\end{equation*}
By Young's inequality for convolution, we obtain a universal constant $C_{1} \geq 0$ such that 
\begin{equation}\label{est 202}
\lVert \mathbb{P}_{q+1, k} \rVert_{C_{x} \mapsto C_{x}} \leq C_{1}. 
\end{equation}  
At last, we define $\tilde{P}_{\approx \lambda_{q+1}}$ to be the Fourier operator with a symbol supported in frequency $\{\xi: \frac{\lambda_{q+1}}{4} \leq \lvert \xi \rvert \leq 4 \lambda_{q+1} \}$ and is identically one on $\{\xi: \frac{3 \lambda_{q+1}}{8} \leq \lvert \xi \rvert \leq 3 \lambda_{q+1}\}$. 

Next, we recall $l = \lambda_{q+1}^{-\alpha}$ from \eqref{est 32} and split $[t_{q+1}, T_{L}]$ into finitely many subintervals of size $\tau_{q+1}$ defined by 
\begin{equation}\label{est 210}
\tau_{q+1}^{-1} \triangleq l^{-\frac{1}{2}} \lambda_{q+1}^{\frac{ 3- \gamma_{2}}{2}} \delta_{q+1}^{\frac{1}{4}}. 
\end{equation} 
We let $\chi: \mathbb{R} \mapsto [0,1]$ be a smooth cut-off function with support in $(-1, 1)$ such that $\chi \equiv 1$ on $(-\frac{1}{4}, \frac{1}{4})$, and additionally define 
\begin{equation}\label{est 209} 
\chi_{j}(t) \triangleq \chi(\tau_{q+1}^{-1} t - j) \hspace{3mm} \text{ for } j \in \{ \lfloor t_{q+1} \tau_{q+1}^{-1} \rfloor, \lfloor t_{q+1} \tau_{q+1}^{-1} \rfloor + 1, \hdots,  \lceil T_{L} \tau_{q+1}^{-1} \rceil\}, 
\end{equation}
that satisfies 
\begin{equation*}
\sum_{j= \lfloor t_{q+1} \tau_{q+1}^{-1} \rfloor}^{\lceil T_{L} \tau_{q+1}^{-1}  \rceil} \chi_{j}^{2}(t)= 1 \hspace{3mm} \forall \hspace{1mm} t \in [t_{q+1}, T_{L}]. 
\end{equation*} 
For convenience, we denote 
\begin{align*}
&\sum_{j} \triangleq \sum_{j= \lfloor t_{q+1} \tau_{q+1}^{-1} \rfloor}^{\lceil T_{L} \tau_{q+1}^{-1} \rceil }, \\
& \sum_{k} \triangleq \sum_{k\in\Gamma_{j}} \text{ for fixed } j \in \{\lfloor t_{q+1} \tau_{q+1}^{-1} \rfloor, \lfloor t_{q+1} \tau_{q+1}^{-1} \rfloor + 1,  \hdots, \lceil T_{L} \tau_{q+1}^{-1} \rceil \}, \\
& \sum_{j,k} \triangleq \sum_{j= \lfloor t_{q+1} \tau_{q+1}^{-1} \rfloor}^{\lceil T_{L}  \tau_{q+1}^{-1} \rceil} \sum_{k\in\Gamma_{j}},  
\end{align*}
where we define 
\begin{align*}
\Gamma_{j} \triangleq 
\begin{cases}
\Gamma_{1} & \text{ when } j \text{ is odd}, \\
\Gamma_{2} & \text{ when } j \text{ is even},
\end{cases}
\end{align*}
with $\Gamma_{1}$ and $\Gamma_{2}$ from Lemma \ref{Geometric Lemma}.

The range of following parameters are crucial in our estimates. First, we consider 
\begin{equation}\label{est 315a}
\gamma_{2} \in [1, 2) 
\end{equation} 
as fixed. Then we fix 
\begin{equation}\label{est 315d}  
 \gamma_{1} \in \left(0, 2 - \frac{\gamma_{2}}{2}\right). 
\end{equation} 
For 
\begin{equation}\label{est 443}
L \geq \frac{44}{10 - 5 \gamma_{2}}, 
\end{equation} 
we choose $b \in \mathbb{N}$ such that 
\begin{equation}\label{est 446}
b >  \frac{11L^{2}}{2-\gamma_{2}} + 5. 
\end{equation} 
This choice of $b$ is made after determining all of its necessary lower bounds in \eqref{est 464}, \eqref{est 465}, \eqref{est 466}, \eqref{est 467}, \eqref{est 450}, \eqref{est 468}, \eqref{est 469}, \eqref{est 472}, \eqref{est 473}, \eqref{est 474},  \eqref{est 476}, \eqref{est 470}, and \eqref{est 471}. Furthermore, we consider $\beta$ in the range of 
\begin{equation}\label{est 315b}
1 - \frac{\gamma_{2}}{2} < \beta <  \frac{12 - 6 \gamma_{2}}{11},  
\end{equation} 
to be taken close to the lower bound as needed; finally, we choose 
\begin{equation}\label{est 315c}
\alpha \triangleq 1 + \frac{\beta}{2}. 
\end{equation} 
Next, we define 
\begin{subequations}\label{est 326}
\begin{align}
&M_{0} \in C^{\infty}(\mathbb{R}) \text{ such that } M_{0}(t) 
= 
\begin{cases}
e^{2L} & \text{ if } t \leq 0, \\
e^{4Lt + 2L} & \text{ if } t \geq T \wedge L, 
\end{cases}
\hspace{5mm} 0 \leq M_{0}'(t) \leq 8L M_{0}(t), \\
&m_{L} \triangleq \sqrt{3} L^{\frac{1}{4}} e^{\frac{1}{2} L^{\frac{1}{4}}}. 
\end{align}
\end{subequations}
We denote specific universal non-negative constants $C_{G}$ and $C_{S}$ respectively due to the Gagliardo-Nirenberg and the Sobolev's inequalities: for $f$ that is sufficiently smooth in $x \in \mathbb{T}^{2}$ and mean-zero,  
\begin{subequations}\label{est 155}
\begin{align}
&\lVert f \rVert_{C_{x}^{2-\gamma_{2}}} + \lVert \Lambda^{2-\gamma_{2}} f \rVert_{L_{x}^{\infty}} \leq C_{G} \lVert f\rVert_{L_{x}^{2}}^{\frac{\gamma_{2}}{3}} \lVert \Lambda^{3} f \rVert_{L_{x}^{2}}^{1- \frac{\gamma_{2}}{3}}, \label{est 155a} \\
&\lVert f \rVert_{L_{x}^{\infty}} \leq C_{S} \lVert f \rVert_{\dot{H}_{x}^{3-\gamma_{1}}}. \label{est 155b}
\end{align}
\end{subequations} 
Due to \eqref{est 312} and the fact that we extended $B(t)$ to $[-2, 0]$ by $B(t) = B(0)$ for all $t \in [-2, 0]$, we see that for all $\delta \in (0, \frac{1}{4})$ and $t \in [-2, T_{L}] = [t_{0}, T_{L}]$, 
\begin{equation}\label{est 324} 
\lvert B(t) \rvert \leq L^{\frac{1}{4}} \hspace{1mm} \text{ and } \hspace{1mm} \lVert B \rVert_{C_{t,0}^{\frac{1}{2} - 2 \delta}} \leq L^{\frac{1}{2}}
\end{equation} 
which implies that for $L$ that satisfies \eqref{est 443}, we have 
\begin{equation}\label{est 332}
\lVert \Upsilon \rVert_{C_{t,0}},  \lVert \Upsilon^{-1} \rVert_{C_{t,0}}, \leq e^{L^{\frac{1}{4}}}, \hspace{5mm}  \lVert \Upsilon_{l}^{-\frac{1}{2}} \rVert_{C_{t,0}} \leq e^{\frac{1}{2}L^{\frac{1}{4}}}, \hspace{5mm} \lVert \Upsilon \rVert_{C_{t,0}^{\frac{1}{2} - 2\delta}}, \lVert \Upsilon^{-1} \rVert_{C_{t,0}^{\frac{1}{2} - 2\delta}} \leq m_{L}^{2}.
\end{equation} 
For $C_{1}$ from \eqref{est 202}, we fix a universal constant 
\begin{equation}\label{est 481} 
C_{0} \geq \max\left\{\frac{\pi}{2},  2^{4} \sup_{k \in \Gamma_{1} \cup \Gamma_{2}} \lVert \gamma_{k} \rVert_{C(B(\Id, \epsilon_{\gamma}))}C_{1}\right\}
\end{equation} 
that is sufficiently large to satisfy the last inequalities of \eqref{est 482} and \eqref{est 483}. For another universal constant $\epsilon_{\gamma} > 0$ from Lemma \ref{Geometric Lemma}, we consider the following inductive hypothesis. 
\begin{hypothesis}\label{Hypothesis 5.1}
\noindent For all $t \in [t_{q}, T_{L}]$, 
\begin{enumerate}[(a)] 
\item\label{Hypothesis 5.1 (a)} $\supp \hat{y}_{q} \subset B(0, 2 \lambda_{q})$, 
\begin{equation}\label{est 316}
\lVert y_{q} \rVert_{C_{t,x,q}} \leq C_{0} \left(1+ \sum_{1 \leq j \leq q} \delta_{j}^{\frac{1}{2}} \right) m_{L} M_{0}(t)^{\frac{1}{2}}, 
\end{equation} 
\begin{equation}\label{est 317} 
\lVert y_{q} \rVert_{C_{t,q} C_{x}^{2-\gamma_{2}}} + \lVert \Lambda^{2-\gamma_{2}} y_{q} \rVert_{C_{t,x,q}} \leq C_{0} m_{L} M_{0}(t)^{\frac{1}{2}} \lambda_{q}^{2-\gamma_{2}} \delta_{q}^{\frac{1}{2}}, 
\end{equation} 
\item\label{Hypothesis 5.1 (b)}  $\supp \hat{\mathring{R}}_{q} \subset B(0, 4 \lambda_{q})$, 
\begin{equation}\label{est 318}
\lVert \mathring{R}_{q} \rVert_{C_{t,x,q}} \leq \epsilon_{\gamma} M_{0}(t) \lambda_{q+1}^{2-\gamma_{2}} \delta_{q+1}, 
\end{equation} 
\item\label{Hypothesis 5.1 (c)}  
\begin{equation}\label{est 319}
\lVert (\partial_{t} + \Upsilon_{l} \Lambda^{2-\gamma_{2}} y_{l} \cdot \nabla) y_{q} \rVert_{C_{t,x,q}} \leq C_{0} L^{\frac{1}{2}} e^{2L^{\frac{1}{4}}} M_{0}(t) \lambda_{q}^{3-\gamma_{2}} \delta_{q}.
\end{equation} 
\end{enumerate}  
\end{hypothesis}

\begin{proposition}\label{Proposition 5.1} 
Suppose that $\gamma_{2}, \gamma_{1}, \beta$, $\alpha$, and $b$ respectively satisfy \eqref{est 315a}, \eqref{est 315d}, \eqref{est 315b}, \eqref{est 315c}, and \eqref{est 446}. Let $L$ satisfy  \eqref{est 443} and define 
\begin{equation}\label{est 321}
y_{0} (t,x) \triangleq \frac{m_{L} M_{0}(t)^{\frac{1}{2}}}{2\pi} 
\begin{pmatrix}
\sin(x_{2}) \\
0
\end{pmatrix}.
\end{equation} 
Then, together with 
\begin{align}
\mathring{R}_{0}(t,x) \triangleq \frac{ m_{L}(\frac{d}{dt} + \frac{1}{2}) M_{0}(t)^{\frac{1}{2}} }{2\pi} 
\begin{pmatrix}
0 & - \cos(x_{2}) \\
-\cos(x_{2}) & 0 
\end{pmatrix}  + \mathcal{B}( \Lambda^{\gamma_{1}} y_{0})(t,x),  \label{est 322}
\end{align} 
the pair $(y_{0}, \mathring{R}_{0})$ solves \eqref{est 320} and satisfies the Hypothesis \ref{Hypothesis 5.1} at level $ q = 0$ on $[t_{0}, T_{L}]$ provided $a \in 5 \mathbb{N}$ is taken sufficiently large and that 
\begin{equation}\label{est 153b} 
\max \left\{2, \frac{2^{\frac{5}{2} - \frac{\gamma_{2}}{2}} C_{0}}{\pi} \right\} <  a^{b(-1 +  \frac{\gamma_{2}}{2} + \beta)} \leq \frac{ \sqrt{\epsilon_{\gamma}} e^{\frac{L}{2}}}{ \sqrt{m_{L}} \sqrt{\left(4L + \frac{1}{2} \right)\frac{1}{\pi} + C_{S} 2^{-\frac{1}{2}}} }
\end{equation}  
where the first inequality of \eqref{est 153b} guarantees \eqref{est 442}. Finally, $y_{0}(t,x)$ and $\mathring{R}_{0}(t,x)$ are both deterministic over $[t_{0}, 0]$. 
\end{proposition} 

\begin{proof}[Proof of Proposition \ref{Proposition 5.1}] 
We observe that $y_{0}$ is divergence-free, mean-zero while $\mathring{R}_{0}$ is symmetric and trace-free due to Lemma \ref{Inverse-divergence lemma}. Moreover, we recall from \cite[Equation (95)]{Y23a} that 
\begin{subequations}\label{est 171}
\begin{align}
& \int_{\mathbb{T}^{2}} \lvert \sin(x_{2}) \rvert^{2} dx = \int_{\mathbb{T}^{2}} \lvert \cos(x_{2}) \rvert^{2}dx = 2\pi^{2}, \label{est 171a}\\
& \supp \widehat{\sin(x_{2})} \subset \{(0,1), (0, -1)\}, \hspace{2mm} \widehat{\sin(x_{2}))}((0,1)) = -i2\pi^{2}, \hspace{2mm} \widehat{ \sin(x_{2})}((0, -1)) = i2\pi^{2}, \label{est 171b}\\
& \supp \widehat{\cos(x_{2})} \subset \{(0,1), (0,-1) \}, \hspace{2mm}\widehat{\cos(x_{2})}((0,1)) = \widehat{\cos(x_{2})}((0,-1)) = 2 \pi^{2}. \label{est 171c}
\end{align}
\end{subequations} 
Thus, because $(\Lambda^{2- \gamma_{2}} y_{0} \cdot \nabla) y_{0} = 0$ and 
\begin{align}\label{est 477}
\Upsilon (\nabla y_{0})^{T}\cdot \Lambda^{2-\gamma_{2}} y_{0} = \Upsilon \frac{m_{L}^{2} M_{0}(t)}{(2\pi)^{2}} 
\begin{pmatrix}
0\\
\cos(x_{2}) \Lambda^{2-\gamma_{2}} \sin(x_{2})
\end{pmatrix} \overset{\eqref{est 171}}{=} \Upsilon \frac{ m_{L}^{2} M_{0}(t)}{2 (2\pi)^{2}} \nabla \sin(x_{2})^{2}, 
\end{align}
we see that $(y_{0}, \mathring{R}_{0})$ solves \eqref{est 320} with 
\begin{align}\label{est 444}
p_{0} \triangleq \frac{ \Upsilon m_{L}^{2} M_{0}(t) \sin(x_{2})^{2}}{2 (2\pi)^{2}}. 
\end{align}
Moreover, due to \eqref{est 171} we see that $\supp \hat{y}_{0} \subset B(0, 2\lambda_{0})$ if we take $\lambda_{0} \overset{\eqref{est 35}}{=} a > \frac{1}{2}$. This also implies directly from \eqref{est 322} that $\supp \hat{\mathring{R}}_{0} \subset B(0, 2\lambda_{0})  \subset B(0, 4\lambda_{0})$.  Next, because $C_{0} \geq \frac{1}{2\pi}$ due to \eqref{est 481}, it follows immediately that $\lVert y_{0} \rVert_{C_{t,x,0}} \leq \frac{m_{L} M_{0}(t)^{\frac{1}{2}}}{2\pi} \leq C_{0} m_{L} M_{0}(t)^{\frac{1}{2}}$. Additionally, by \eqref{est 315b} that guarantees $\beta < 2 -\gamma_{2}$ and \eqref{est 171}, we can take $a \in 5 \mathbb{N}$ sufficiently large to deduce 
\begin{equation*}
 \lVert y_{0} \rVert_{C_{t,0}C_{x}^{2- \gamma_{2}}}  +\lVert \Lambda^{2- \gamma_{2}} y_{0} \rVert_{C_{t,x,0}} \overset{\eqref{est 321} \eqref{est 155a}}{\leq} C_{G} 2^{-\frac{1}{2}} m_{L}  M_{0}(t)^{\frac{1}{2}}  \leq C_{0}  m_{L} M_{0}(t)^{\frac{1}{2}} \lambda_{0}^{2-\gamma_{2}} \delta_{0}^{\frac{1}{2}}.
\end{equation*} 
To compute $\lVert \mathring{R}_{0} \rVert_{C_{t,x}}$, we rely on \eqref{est 321}, \eqref{est 155b}, and \eqref{est 171} to estimate for all $t \in [t_{0}, T_{L}]$,
\begin{equation}\label{est 323}
 \lVert \mathcal{B} \Lambda^{\gamma_{1}} y_{0}(t) \rVert_{C_{x}} \leq m_{L} M_{0}(t)^{\frac{1}{2}}C_{S} 2^{-\frac{1}{2}}. 
\end{equation} 
Therefore, we can estimate from \eqref{est 322}  
\begin{align*}
 \lVert \mathring{R}_{0} \rVert_{C_{t,x,0}} \overset{\eqref{est 323} \eqref{est 326}}{\leq} m_{L} \left[\left(4L + \frac{1}{2}\right) \frac{2}{\pi} + C_{S} 2^{-\frac{1}{2}}\right] M_{0}(t)^{\frac{1}{2}}   \overset{\eqref{est 153b} \eqref{est 35} \eqref{est 326}}{\leq}  \epsilon_{\gamma}M_{0}(t) \lambda_{1}^{2- \gamma_{2}}\delta_{1}.
\end{align*}
Finally, using the fact that $\beta < \frac{3-\gamma_{2}}{2}$ due to \eqref{est 315b}, for $a \in 5 \mathbb{N}$ sufficiently large, we estimate 
\begin{align*}
& \lVert ( \partial_{t} + \Upsilon_{\lambda_{1}^{-\alpha}}\Lambda^{2- \gamma_{2}} y_{0}\ast_{x} \phi_{\lambda_{1}^{-\alpha}} \ast_{t} \varphi_{\lambda_{1}^{-\alpha}} \cdot \nabla) y_{0} \rVert_{C_{t,x,0}}  \nonumber \\
\overset{\eqref{est 332}}{\leq}&  \lVert \partial_{t} y_{0} \rVert_{C_{t,x}} + e^{L^{\frac{1}{4}}} \lVert \Lambda^{2- \gamma_{2}} y_{0} \rVert_{C_{t,x,0}} \lVert \nabla y_{0} \rVert_{C_{t,x,0}}  \nonumber \\
\overset{\eqref{est 155a} \eqref{est 321}}{\leq}&  C_{0} M_{0}(t) 3L^{\frac{1}{2}} e^{2L^{\frac{1}{4}}} \left[ \frac{2}{C_{0} \pi} + \frac{ C_{G}  }{2 \sqrt{2} \pi C_{0}} \right] \leq  C_{0} L^{\frac{1}{2}}  e^{2L^{\frac{1}{4}}} M_{0}(t) \lambda_{0}^{3- \gamma_{2}} \delta_{0}. 
\end{align*} 
At last, it is clear that $y_{0}(t,x)$ and consequently $\mathring{R}_{0}(t,x)$ are both deterministic over $[t_{0}, 0]$.  
\end{proof} 

\begin{remark}
We saw in \eqref{est 477}-\eqref{est 444} that the shear flow of the form $f(x_{2}) = \begin{pmatrix}
\sin(x_{2})\\
0
\end{pmatrix}$ satisfies 
\begin{align*}
(\nabla f)^{T} \cdot \Lambda^{2- \gamma_{2}} f = \nabla \begin{pmatrix}
0\\
\frac{1}{2}\sin(x_{2})^{2}
\end{pmatrix}.
\end{align*}
This was clear to see  due to \eqref{est 171}. Nevertheless, this is true for general functions $f = f(x_{2})$, $(\nabla f)^{T} \cdot \Lambda^{2-\gamma_{2}} f = 
\begin{pmatrix}
0\\
f'(x_{2}) \Lambda^{2-\gamma_{2}} f(x_{2})
\end{pmatrix}= \nabla 
\begin{pmatrix}
0\\
g(x_{2})
\end{pmatrix}$ for $g(x_{2})$ that is an anti-derivative of $f'(x_{2}) \Lambda^{2-\gamma_{2}} f(x_{2})$. 
\end{remark}

Before we state the following key proposition, we recall that $y_{q+1}= y_{l} + w_{q+1}$ from \eqref{est 164}.  
\begin{proposition}\label{Proposition 5.2}
Let $L$ that satisfies \eqref{est 443} be larger if necessary to satisfy 
\begin{equation}\label{est 445}
\max \left\{2, \frac{2^{\frac{5}{2} - \frac{\gamma_{2}}{2}} C_{0}}{\pi} \right\} <   \frac{ \sqrt{\epsilon_{\gamma}} e^{\frac{L}{2}}}{ \sqrt{m_{L}} \sqrt{\left(4L + \frac{1}{2} \right)\frac{1}{\pi} + C_{S} 2^{-\frac{1}{2}}} }.
\end{equation} 
Let $\gamma_{2}, \gamma_{1}$, $\alpha$, and $b$ respectively satisfy \eqref{est 315a},\eqref{est 315d}, \eqref{est 315c}, and \eqref{est 446}. Suppose that $(y_{q}, \mathring{R}_{q})$  is a $(\mathcal{F}_{t})_{t\geq 0}$-adapted solution to \eqref{est 320} that satisfies Hypothesis \ref{Hypothesis 5.1} over $[t_{q}, T_{L}]$. Then there exist $a \in 5 \mathbb{N}$ sufficiently large and $\beta > 1 - \frac{\gamma_{2}}{2}$ in \eqref{est 315b} sufficiently close to $1- \frac{\gamma_{2}}{2}$ that satisfies \eqref{est 153b} and $(\mathcal{F}_{t})_{t\geq 0}$-adapted processes $(y_{q+1}, \mathring{R}_{q+1})$ that solves \eqref{est 320}, satisfies the Hypothesis \ref{Hypothesis 5.1} at level $q+1$, and for all $t \in [t_{q+1}, T_{L}]$
\begin{equation}\label{est 327}  
\lVert y_{q+1}(t) - y_{q}(t) \rVert_{C_{x}} \leq C_{0} e^{\frac{1}{2} L^{\frac{1}{4}}} M_{0}(t)^{\frac{1}{2}} \delta_{q+1}^{\frac{1}{2}}.
\end{equation} 
Moreover, $w_{q+1} = y_{q+1} - y_{l}$ from \eqref{est 164} satisfies 
\begin{equation}\label{est 163} 
\supp\hat{w}_{q+1} \subset \left\{ \xi: \frac{\lambda_{q+1}}{2} \leq \lvert \xi \rvert \leq 2 \lambda_{q+1} \right\}.
\end{equation} 
Finally, if $y_{q}(t,x)$ and $\mathring{R}_{q}(t,x)$ are deterministic over $[t_{q}, 0]$, then so are $y_{q+1}(t,x)$ and $\mathring{R}_{q+1}(t, x)$ over $[t_{q+1}, 0]$. 
\end{proposition}

\begin{proof}[Proof of Theorem \ref{Theorem 2.1} assuming Proposition \ref{Proposition 5.2}]
Assuming Proposition \ref{Proposition 5.2} holds, Theorem \ref{Theorem 2.1} can be readily proved similarly to previous works (e.g. \cite{Y23a}); thus, we leave this in the Appendix for completeness. 
\end{proof}

\subsection{Proof of Proposition \ref{Proposition 5.2}} 
Throughout this section, our estimates will be for $t \in [t_{q+1}, T_{L}]$. Due to \eqref{est 445} we can take $a \in 5 \mathbb{N}$ large such that 
\begin{equation}\label{est 475}
a \geq e^{8} 
\end{equation} 
and $\beta > 1 - \frac{\gamma_{2}}{2}$ sufficiently close to $1 - \frac{\gamma_{2}}{2}$ and achieve \eqref{est 153b}. First, we compute 
\begin{align}
\lVert y_{q} - y_{l} \rVert_{C_{t,x,q+1}} \lesssim& l \lVert y_{q} \rVert_{C_{t,q+1}C_{x}^{1}} + l \lVert y_{q} \ast_{x} \phi_{l} \rVert_{C_{t,q+1}^{1}C_{x}}  \nonumber  \\\overset{\eqref{est 330} \eqref{est 317}}{\lesssim}& l m_{L}^{2} M_{0}(t) e^{L^{\frac{1}{4}}}  \lambda_{q}^{3-\gamma_{2}} \delta_{q}  [ \lambda_{q}^{-2 + \gamma_{2} + \beta} + 1 ] \lesssim  l m_{L}^{2} M_{0}(t) e^{L^{\frac{1}{4}}}  \lambda_{q}^{3-\gamma_{2}} \delta_{q}, \label{est 340} 
\end{align} 
where the last inequality used $\beta \leq 2-\gamma_{2}$ due to \eqref{est 315b}. Next, we decompose the Reynolds stress $\mathring{R}_{q+1}$ carefully as follows: first, using \eqref{est 320}, \eqref{est 164}, and \eqref{est 341} gives us 
\begin{align}
& \divergence \mathring{R}_{q+1} - \nabla p_{q+1} = - \Upsilon_{l} [ ( \Lambda^{2-\gamma_{2}} y_{l} \cdot \nabla) y_{l} - (\nabla y_{l})^{T} \cdot \Lambda^{2- \gamma_{2}} y_{l} ] - \nabla p_{l} + \divergence \mathring{R}_{l} + \divergence R_{\text{Com1}}   \nonumber  \\
&+ \partial_{t} w_{q+1} + \frac{1}{2} w_{q+1} + \Lambda^{\gamma_{1}} w_{q+1} + \Upsilon [ ( \Lambda^{2-\gamma_{2}} y_{q+1} \cdot \nabla) y_{q+1} - (\nabla y_{q+1})^{T} \cdot \Lambda^{2- \gamma_{2}} y_{q+1} ]. \label{est 342} 
\end{align} 
We rewrite the nonlinear terms by using \eqref{est 164} as follows: 
\begin{align}
& - \Upsilon_{l} [ ( \Lambda^{2- \gamma_{2}} y_{l} \cdot \nabla) y_{l} - ( \nabla y_{l})^{T} \cdot \Lambda^{2- \gamma_{2}} y_{l} ] + \Upsilon [(\Lambda^{2- \gamma_{2}} y_{q+1} \cdot \nabla) y_{q+1} - (\nabla y_{q+1})^{T} \cdot \Lambda^{2- \gamma_{2}} y_{q+1} ] \nonumber \\
=& (\Upsilon - \Upsilon_{l}) [(\Lambda^{2-\gamma_{2}} y_{q+1} \cdot\nabla) y_{q+1} - (\nabla y_{q+1})^{T} \cdot \Lambda^{2-\gamma_{2}} y_{q+1} ] \nonumber\\
&+ \Upsilon_{l} [ (\Lambda^{2-\gamma_{2}} w_{q+1} \cdot \nabla )y_{l} + (\Lambda^{2-\gamma_{2}} w_{q+1} \cdot \nabla) w_{q+1} + \Lambda^{2-\gamma_{2}} y_{l} \cdot \nabla w_{q+1} \nonumber \\
& \hspace{5mm} - (\nabla w_{q+1})^{T} \cdot \Lambda^{2-\gamma_{2}} y_{l} - (\nabla w_{q+1})^{T} \cdot \Lambda^{2-\gamma_{2}} w_{q+1} - (\nabla y_{l})^{T} \cdot \Lambda^{2-\gamma_{2}} w_{q+1} ]. \label{est 343} 
\end{align}
We can improve one of the singular terms within \eqref{est 343} by rewriting 
\begin{equation}\label{est 344}
-\Upsilon_{l} (\nabla w_{q+1})^{T} \cdot\Lambda^{2- \gamma_{2}} y_{l} = \Upsilon_{l} ( \nabla \Lambda^{2- \gamma_{2}} y_{l})^{T} \cdot w_{q+1} - \Upsilon_{l} \nabla (w_{q+1} \cdot \Lambda^{2- \gamma_{2} } y_{l}), 
\end{equation} 
where the latter term can be included within the pressure term. Applying \eqref{est 343}-\eqref{est 344} to \eqref{est 342} gives 
\begin{equation}\label{est 190} 
\divergence \mathring{R}_{q+1} = \divergence (R_{T} + R_{N} + R_{L} + R_{O} + R_{\text{Com1}} + R_{\text{Com2}})
\end{equation} 
and 
\begin{equation*}
p_{q+1} \triangleq p_{l} + \Upsilon_{l} (w_{q+1} \cdot \Lambda^{2- \gamma_{2}} y_{l})
\end{equation*} 
if we define, in addition to $R_{\text{Com1}}$ and $p_{l}$ in \eqref{est 338},  
\begin{subequations}\label{est 345} 
\begin{align}
\divergence R_{T} \triangleq&  \partial_{t} w_{q+1} + \Upsilon_{l} (\Lambda^{2-\gamma_{2}} y_{l} \cdot \nabla) w_{q+1}, \label{est 345a}\\
\divergence R_{N} \triangleq& \Upsilon_{l} [ (\nabla \Lambda^{2-\gamma_{2}} y_{l})^{T} \cdot w_{q+1} + (\Lambda^{2-\gamma_{2}} w_{q+1} \cdot \nabla) y_{l} - (\nabla y_{l})^{T} \cdot \Lambda^{2-\gamma_{2}} w_{q+1} ], \label{est 345c} \\
\divergence R_{L} \triangleq&  \Lambda^{\gamma_{1}} w_{q+1} +  \frac{1}{2} w_{q+1}, \label{est 345d} \\
\divergence R_{O} \triangleq& \divergence \mathring{R}_{l} + \Upsilon_{l} [ (\Lambda^{2- \gamma_{2}} w_{q+1} \cdot \nabla) w_{q+1} - (\nabla w_{q+1})^{T} \cdot \Lambda^{2-\gamma_{2}} w_{q+1} ], \label{est 345b} \\
\divergence R_{\text{Com2}} \triangleq& (\Upsilon - \Upsilon_{l}) [(\Lambda^{2- \gamma_{2}} y_{q+1} \cdot \nabla) y_{q+1} - (\nabla y_{q+1})^{T} \cdot \Lambda^{2-\gamma_{2}} y_{q+1} ], \label{est 345e} 
\end{align}
\end{subequations} 
representing respectively the Reynolds stress errors of transport, Nash, linear, oscillation, and second commutator types. 

We recall $D_{t,q}$ from \eqref{est 349} and define $\Phi_{j}(t,x)$ for all $j \in \{ \lfloor t_{q+1} \tau_{q+1}^{-1} \rfloor, \hdots, \lceil T_{L} \tau_{q+1}^{-1} \rceil\}$ that solves 
\begin{equation}\label{est 354} 
D_{t,q} \Phi_{j} = 0,  \hspace{5mm} \Phi_{j} (\tau_{q+1} j, x) = x, 
\end{equation} 
$\mathring{R}_{q,j}$ that solves 
\begin{equation}\label{est 355}
D_{t,q} \mathring{R}_{q,j} = 0, \hspace{5mm}  \mathring{R}_{q,j} (\tau_{q+1} j, x) = \mathring{R}_{l} (\tau_{q+1} j, x),
\end{equation} 
as well as 
\begin{subequations}\label{est 484}
\begin{align}
\bar{a}_{k,j} (t,x) \triangleq& \Upsilon_{l}^{-\frac{1}{2}} a_{k,j}(t,x) \label{est 351}\\
\triangleq& \Upsilon_{l}^{-\frac{1}{2}}  \frac{ M_{0} (\tau_{q+1} j)^{\frac{1}{2}}}{\sqrt{\gamma_{2}}}\delta_{q+1}^{\frac{1}{2}} \gamma_{k} \left( \Id - \frac{ \mathring{R}_{q,j}(t,x)}{\lambda_{q+1}^{2- \gamma_{2}} \delta_{q+1}  M_{0} (\tau_{q+1} j)} \right); \label{est 447}
\end{align} 
\end{subequations}
we point out that $a_{k,j}$ defined in \eqref{est 447} is different from that defined in the case of an additive noise (see \cite[Equation (132)]{Y23a}). The fact that $\left( \Id - \frac{ \mathring{R}_{q,j}(t,x)}{\lambda_{q+1}^{2- \gamma_{2}} \delta_{q+1} M_{0} (\tau_{q+1} j)} \right)$ lies in the domain of $\gamma_{k}$ can be verified as 
\begin{align*}
\left\lvert \frac{ \mathring{R}_{q,j} (t,x)}{\lambda_{q+1}^{2- \gamma_{2}} \delta_{q+1}  M_{0} (\tau_{q+1} j)} \right\rvert \overset{\eqref{est 197a}  \eqref{est 355}}{\leq} \frac{ \lVert \mathring{R}_{l} (\tau_{q+1} j, x) \rVert_{C_{x}}}{ \lambda_{q+1}^{2- \gamma_{2}} \delta_{q+1}  M_{0} (\tau_{q+1} j)} \overset{\eqref{est 318}}{\leq} \epsilon_{\gamma}.
\end{align*}
Furthermore, with $\chi_{j}$ from \eqref{est 209} and $b_{k}$ from \eqref{est 178}, we define 
\begin{equation}\label{est 347}  
w_{q+1} (t,x) \triangleq  \sum_{j,k}  \chi_{j}(t) \mathbb{P}_{q+1, k} ( \bar{a}_{k,j} (t,x) b_{k} (\lambda_{q+1} \Phi_{j} (t,x))). 
\end{equation} 
Because $\supp \widehat{\mathbb{P}_{q+1,k},f} \subset \{ \xi: \frac{7}{8} \lambda_{q+1} \leq \lvert \xi \rvert \leq \frac{9}{8} \lambda_{q+1}\}$ from Section \ref{Section 5.1}, we see that $w_{q+1}$ in \eqref{est 347} satisfies \eqref{est 163}. This, together with Hypothesis \ref{Hypothesis 5.1} \ref{Hypothesis 5.1 (a)} at level $q$, implies that $y_{q+1} = y_{l} + w_{q+1}$ from \eqref{est 164}  has its frequency support contained in $B(0, 2 \lambda_{q+1})$. In turn, this leads \eqref{est 342}  to imply $\supp \hat{\mathring{R}}_{q+1} \subset B(0, 4 \lambda_{q+1})$ as desired. Additionally, we define 
\begin{subequations}\label{est 448}
\begin{align}
&\tilde{w}_{q+1, j, k} (t,x) \triangleq \chi_{j}(t) \bar{a}_{k,j}(t,x) b_{k} (\lambda_{q+1} \Phi_{j}(t,x)),\label{est 348}\\
& \psi_{q+1, j, k} (t,x) \triangleq \frac{c_{k} (\lambda_{q+1} \Phi_{j} (t,x))}{c_{k} (\lambda_{q+1} x)} \overset{\eqref{est 178} }{=} e^{i \lambda_{q+1} (\Phi_{j} (t,x) - x) \cdot k} \label{est 200b} 
\end{align}
\end{subequations}
where $c_{k}$ is defined in \eqref{est 178}. It follows that  
\begin{equation}\label{est 208}
w_{q+1} = \sum_{j, k} \mathbb{P}_{q+1, k} \tilde{w}_{q+1, j, k} \hspace{1mm} \text{ and } \hspace{1mm}  b_{k} (\lambda_{q+1} \Phi_{j} (x)) = b_{k} (\lambda_{q+1} x) \psi_{q+1, j, k} (x)
\end{equation} 
due to \eqref{est 347} and \eqref{est 178}. 

\begin{proposition}\label{Proposition 5.3}  
Define $D_{t,q}$ by \eqref{est 349}. Then $w_{q+1}$ defined in \eqref{est 347} satisfies the following estimates: for $C_{1}$ from \eqref{est 202} and all $t \in [t_{q+1}, T_{L}]$, 
\begin{subequations}\label{est 350} 
\begin{align}
&\lVert w_{q+1}(t) \rVert_{C_{x}} \leq 2 \sup_{k \in \Gamma_{1} \cup \Gamma_{2}} \lVert \gamma_{k} \rVert_{C(B(\Id, \epsilon_{\gamma} ))} \frac{ C_{1} e^{\frac{1}{2} L^{\frac{1}{4}}} M_{0} (t)^{\frac{1}{2}}}{\sqrt{\gamma_{2}}} \delta_{q+1}^{\frac{1}{2}},\label{est 350a} \\
&\lVert D_{t,q} w_{q+1} (t) \rVert_{C_{x}}  \lesssim e^{\frac{1}{2} L^{\frac{1}{4}}}  M_{0}(t)^{\frac{1}{2}} \delta_{q+1}^{\frac{1}{2}}\tau_{q+1}^{-1}. \label{est 350b}
\end{align}
\end{subequations} 
Consequently, \eqref{est 327}, as well as all of \eqref{est 316}, \eqref{est 317}, and \eqref{est 319} at level $q+1$ hold, completing the verification of the  Hypothesis \ref{Hypothesis 5.1} \ref{Hypothesis 5.1 (a)} and \ref{Hypothesis 5.1 (c)} at level $q+1$.  
\end{proposition}

\begin{proof}[Proof of Proposition \ref{Proposition 5.3}]
First, due to \eqref{est 347}, \eqref{est 351},  and the fact that for all $t$ there exist at most two non-trivial cutoffs, we can compute 
\begin{equation}\label{est 352} 
\lVert w_{q+1}(t) \rVert_{C_{x}} \overset{\eqref{est 202} \eqref{est 332} \eqref{est 178} }{\leq}  2\sup_{k \in \Gamma_{1} \cup \Gamma_{2}} \lVert \gamma_{k} \rVert_{C(B(\Id, \epsilon_{\gamma} ))}   e^{\frac{1}{2} L^{\frac{1}{4}}}  \frac{C_{1} M_{0}(t)^{\frac{1}{2}}}{\sqrt{\gamma_{2}}} \delta_{q+1}^{\frac{1}{2}}. 
\end{equation} 
Next, combining the facts 
\begin{equation}\label{est 464} 
b> \frac{22(L^{2}+1)}{5+ 3 \gamma_{2}}
\end{equation} 
due to \eqref{est 446} and   $1 - \frac{\beta}{2} > \frac{5+ 3 \gamma_{2}}{11}$ due to \eqref{est 315b} gives us 
\begin{align}
& L^{2} + b( \beta - \alpha)+ 3 - \gamma_{2} - 2 \beta \overset{\eqref{est 315c} \eqref{est 464}}{<} L^{2} + \left[ \frac{22(L^{2} + 1)}{5 + 3 \gamma_{2}} \right] \left(-1 + \frac{\beta}{2} \right) + 3 - \gamma_{2} - 2 \beta \nonumber \\
<& L^{2} + \left( \frac{ 2(L^{2} + 1)}{1- \frac{\beta}{2}} \right)\left(-1 + \frac{\beta}{2} \right) + 3 - \gamma_{2} - 2 \beta  \overset{\eqref{est 315b}}{<}  -L^{2} - 1. \label{est 399} 
\end{align}
As a consequence of \eqref{est 352}, \eqref{est 475}, and \eqref{est 399}, we see that for all $a \in 5 \mathbb{N}$ sufficiently large 
\begin{align}
\lVert y_{q+1}(t) - y_{q}(t) \rVert_{C_{x}} &\overset{\eqref{est 164}}{\leq} \lVert w_{q+1} \rVert_{C_{t,x,q+1}} + \lVert y_{l} - y_{q} \rVert_{C_{t,x,q+1}} \nonumber \\
\overset{\eqref{est 164}  \eqref{est 340}  }{\lesssim}& e^{\frac{1}{2} L^{\frac{1}{4}}} \delta_{q+1}^{\frac{1}{2}} M_{0}(t)^{\frac{1}{2}}  \left( 1+ a^{L^{2} + b^{q} [ b(\beta - \alpha) + 3 - \gamma_{2} - 2 \beta]} \right) \leq C_{0} e^{\frac{1}{2} L^{\frac{1}{4}}}  \delta_{q+1}^{\frac{1}{2}} M_{0}(t)^{\frac{1}{2}}  \label{est 482}
\end{align} 
which implies \eqref{est 327}; we recall that in \eqref{est 481} we chose the universal constant $C_{0}$ sufficiently large to satisfy the last inequality of \eqref{est 482}. Also, via Young's inequality for convolution, we can prove \eqref{est 316} at level $q+1$ as follows: 
\begin{equation}\label{est 353}
\lVert y_{q+1} \rVert_{C_{t,x,q+1}} \overset{\eqref{est 164} }{\leq} \lVert y_{l} \rVert_{C_{t,x,q+1}} + \lVert w_{q+1} \rVert_{C_{t,x,q+1}} \overset{\eqref{est 316} \eqref{est 350a}}{\leq}C_{0} \left( 1+ \sum_{1 \leq j \leq q+1} \delta_{j}^{\frac{1}{2}} \right) m_{L} M_{0}(t)^{\frac{1}{2}}
\end{equation} 
where we used that $C_{0}$ satisfies \eqref{est 481}. Similarly, using \eqref{est 164} and the fact that \eqref{est 315b} guarantees $\beta \leq 2- \gamma_{2}$ and $C_{0}$ satisfies \eqref{est 481}, we can prove \eqref{est 317} at level $q+1$ as
\begin{align*}
& \lVert y_{q+1} \rVert_{C_{t,q+1}C_{x}^{2- \gamma_{2}}} + \lVert \Lambda^{2-\gamma_{2}} y_{q+1} \rVert_{C_{t,x,q+1}}   \\
\overset{\eqref{est 163} \eqref{est 317} \eqref{est 350a}   }{\leq}&C_{0} m_{L} M_{0}(t)^{\frac{1}{2}} \lambda_{q+1}^{2- \gamma_{2}} \delta_{q+1}^{\frac{1}{2}} \left[ \frac{1}{2} + a^{b^{q} [ (2-\gamma_{2} - \beta)(1-b)]} \right] \overset{\eqref{est 350a}}{\leq} C_{0} m_{L} M_{0}(t)^{\frac{1}{2}} \lambda_{q+1}^{2- \gamma_{2}} \delta_{q+1}^{\frac{1}{2}}. \nonumber
\end{align*}
Thus, we verified the Hypothesis \ref{Hypothesis 5.1} \ref{Hypothesis 5.1 (a)} at level $q+1$. Next, we can compute from \eqref{est 349}, \eqref{est 351}, and \eqref{est 355}, 
\begin{equation}\label{est 356} 
D_{t,q} \bar{a}_{k,j} (t,x)  =  \partial_{t} \Upsilon_{l}^{-\frac{1}{2}} \frac{ M_{0} (\tau_{q+1} j)^{\frac{1}{2}}}{\sqrt{\gamma_{2}}} \delta_{q+1}^{\frac{1}{2}} \gamma_{k} \left( \Id - \frac{ \mathring{R}_{q,j} (t,x)}{\lambda_{q+1}^{2-\gamma_{2}} \delta_{q+1}  M_{0} (\tau_{q+1} j)} \right);
\end{equation} 
we note that the analogous term in the case of additive noise vanishes (see \cite[Equation (141)]{Y23a}). On the other hand, by relying on \eqref{est 349}, \eqref{est 178}, and \eqref{est 354}, we can verify that 
\begin{equation}\label{est 357} 
D_{t,q} b_{k} (\lambda_{q+1} \Phi_{j} (t,x)) = 0 \hspace{1mm} \text{ and } \hspace{1mm} D_{t,q} c_{k} (\lambda_{q+1} \Phi_{j} (t,x)) = 0.
\end{equation} 
Using \eqref{est 208}, \eqref{est 348}, \eqref{est 349}, and the identities \eqref{est 356}-\eqref{est 357} we can compute
\begin{align}
 D_{t,q} w_{q+1}& =  \sum_{j,k} [ D_{t,q}, \mathbb{P}_{q+1, k} ] \tilde{w}_{q+1, j, k} + \mathbb{P}_{q+1, k} [\partial_{t}\chi_{j}  \bar{a}_{k,j}  b_{k} (\lambda_{q+1} \Phi_{j}) \label{est 358}\\
& \hspace{10mm} + \chi_{j} \partial_{t} \Upsilon_{l}^{-\frac{1}{2}} \frac{ M_{0}(\tau_{q+1} j)^{\frac{1}{2}}}{\sqrt{\gamma_{2}}} \delta_{q+1}^{\frac{1}{2}} \gamma_{k} \left( \Id - \frac{ \mathring{R}_{q,j}}{\lambda_{q+1}^{2-\gamma_{2}} \delta_{q+1} M_{0} (\tau_{q+1} j)} \right) b_{k} (\lambda_{q+1} \Phi_{j}) ] \nonumber 
\end{align}
where $[a,b] \triangleq ab - ba$. Then we can deduce \eqref{est 350b} by relying on \cite[Equation (A.18)]{BSV19} for $a \in 5 \mathbb{N}$ sufficiently large as follows:
\begin{align*} 
& \lVert D_{t,q} w_{q+1} (t) \rVert_{C_{x}}  \overset{\eqref{est 358} \eqref{est 202}}{\lesssim}   \sum_{j,k} \lVert \nabla (\Upsilon_{l} \Lambda^{2- \gamma_{2}} y_{l})(t) \rVert_{C_{x}} \lVert \tilde{w}_{q+1, j, k} (t) \rVert_{C_{x}} \nonumber \\
& \hspace{12mm} + \lVert \partial_{t}\chi_{j}(t) \bar{a}_{k,j} (t,x) b_{k} (\lambda_{q+1} \Phi_{j} (t,x)) \rVert_{C_{x}} \nonumber \\
& \hspace{12mm} + \left\lVert \chi_{j} (t) \Upsilon_{l}^{-\frac{3}{2}} \partial_{t} \Upsilon_{l}  M_{0} (\tau_{q+1} j)^{\frac{1}{2}} \delta_{q+1}^{\frac{1}{2}} \gamma_{k} \left( \Id - \frac{ \mathring{R}_{q,j} (t,x)}{\lambda_{q+1}^{2-\gamma_{2}} \delta_{q+1}  M_{0} (\tau_{q+1} j)} \right)  \right\rVert_{C_{x}} \nonumber \\
& \overset{\eqref{est 332} \eqref{est 317}\eqref{est 32} \eqref{est 475} }{\lesssim} e^{\frac{1}{2}L^{\frac{1}{4}}} M_{0}(t)^{\frac{1}{2}} \delta_{q+1}^{\frac{1}{2}} \tau_{q+1}^{-1} [a^{L^{2}} \lambda_{q}^{3- \gamma_{2}}\delta_{q}^{\frac{1}{2}}  \lambda_{q+1}^{-\frac{\alpha}{2}} \lambda_{q+1}^{-\frac{3-\gamma_{2}}{2}} \lambda_{q+1}^{\frac{\beta}{2}}   + 1  + a^{L^{2}} \lambda_{q+1}^{\frac{\alpha}{2}} \lambda_{q+1}^{-\frac{3-\gamma_{2}}{2}} \lambda_{q+1}^{\frac{\beta}{2}} ] \\
& \hspace{35mm} \lesssim e^{\frac{1}{2} L^{\frac{1}{4}}} M_{0}(t)^{\frac{1}{2}} \delta_{q+1}^{\frac{1}{2}} \tau_{q+1}^{-1}, 
\end{align*}  
where the last inequality used the facts that $\alpha + 3 - \gamma_{2} - \beta > \frac{38 - 8 \gamma_{2}}{11}$ and $- \alpha + 3 - \gamma_{2} - \beta > \frac{4- 2 \gamma_{2}}{11}$ due to \eqref{est 315a} and \eqref{est 315c} which permit 
\begin{equation}\label{est 465}
b > \frac{ 2 [ L^{2} + 3 - \gamma_{2} - \beta]}{\alpha + 3 - \gamma_{2} - \beta}  \hspace{3mm} \text{ and } \hspace{3mm} b > \frac{2L^{2}}{-\alpha + 3 - \gamma_{2} - \beta},  
\end{equation}
both of which are implied by \eqref{est 446}. 

Next, we recall that the notation of $y_{l}$ and $\Upsilon_{l}$ from \eqref{est 337} and compute via \eqref{est 349} and \eqref{est 164} 
\begin{align}
&D_{t, q+1} y_{q+1} \label{est 359} \\
=& D_{t,q} w_{q+1} + [\Upsilon_{\lambda_{q+2}^{-\alpha}} \Lambda^{2-\gamma_{2}} y_{l} \ast_{x} \phi_{\lambda_{q+2}^{-\alpha}} \ast_{t} \varphi_{\lambda_{q+2}^{-\alpha}} - \Upsilon_{\lambda_{q+1}^{-\alpha}} \Lambda^{2-\gamma_{2}} y_{l} ] \cdot \nabla w_{q+1}  +  (D_{t,q} y_{q})\ast_{x} \phi_{\lambda_{q+1}^{-\alpha}} \ast_{t} \varphi_{\lambda_{q+1}^{-\alpha}} \nonumber \\
& \hspace{5mm} +\Upsilon_{\lambda_{q+2}^{-\alpha}} \Lambda^{2- \gamma_{2}} y_{l} \ast_{x} \phi_{\lambda_{q+2}^{-\alpha}} \ast_{t} \varphi_{\lambda_{q+2}^{-\alpha}} \cdot \nabla y_{l}  -  [\Upsilon_{\lambda_{q+1}^{-\alpha}} \Lambda^{2-\gamma_{2}} y_{l} \cdot \nabla y_{q}] \ast_{x} \phi_{\lambda_{q+1}^{-\alpha}} \ast_{v} \varphi_{\lambda_{q+1}^{-\alpha}}\nonumber \\
& \hspace{5mm} + \Upsilon_{\lambda_{q+2}^{-\alpha}} \Lambda^{2- \gamma_{2}} w_{q+1} \ast_{x} \phi_{\lambda_{q+2}^{-\alpha}} \ast_{t} \varphi_{\lambda_{q+2}^{-\alpha}} \cdot\nabla w_{q+1} + \Upsilon_{\lambda_{q+2}^{-\alpha}} \Lambda^{2-\gamma_{2}} w_{q+1}\ast_{x} \phi_{\lambda_{q+2}^{-\alpha}} \ast_{t} \varphi_{\lambda_{q+2}^{-\alpha}} \cdot \nabla y_{l}. \nonumber
\end{align}
We can apply Young's inequality for convolution and \eqref{est 332} to verify \eqref{est 319} at level $q+1$ as follows:
\begin{align} 
&\lVert D_{t,q+1} y_{q+1} \rVert_{C_{t,x,q+1}} \overset{\eqref{est 359} \eqref{est 332} \eqref{est 163} }{\lesssim} \lVert D_{t,q} w_{q+1} \rVert_{C_{t,x,q+1}} +  e^{L^{\frac{1}{4}}} \lVert \Lambda^{2-\gamma_{2}} y_{q} \rVert_{C_{t,x,q+1}} \lambda_{q+1} \lVert w_{q+1} \rVert_{C_{t,x,q+1}} \nonumber \\
&\hspace{32mm}+ \lVert D_{t,q} y_{q} \rVert_{C_{t,x,q+1}} +  e^{L^{\frac{1}{4}}}  \lVert \Lambda^{2-\gamma_{2}} y_{q} \rVert_{C_{t,x,q+1}} \lambda_{q}^{\gamma_{2} - 1} \lVert \Lambda^{2-\gamma_{2}} y_{q} \rVert_{C_{t,x,q+1}} \nonumber \\
&\hspace{33mm}+ e^{L^{\frac{1}{4}}}  \lambda_{q+1}^{3-\gamma_{2}} \lVert w_{q+1} \rVert_{C_{t,x,q+1}}^{2} +  e^{L^{\frac{1}{4}}}   \lambda_{q+1}^{2-\gamma_{2}} \lVert w_{q+1} \rVert_{C_{t,x,q+1}} \lambda_{q}^{\gamma_{2} -1} \lVert \Lambda^{2-\gamma_{2}} y_{q} \rVert_{C_{t,x,q+1}}  \nonumber \\
&\overset{\eqref{est 350}\eqref{est 317} \eqref{est 319}}{\lesssim}  e^{\frac{1}{2} L^{\frac{1}{4}}} \tau_{q+1}^{-1} M_{0}(t)^{\frac{1}{2}} \delta_{q+1}^{\frac{1}{2}}  +  C_{0} e^{\frac{3}{2} L^{\frac{1}{4}}} m_{L} M_{0}(t) \delta_{q}^{\frac{1}{2}} \lambda_{q}^{2-\gamma_{2}} \lambda_{q+1} \delta_{q+1}^{\frac{1}{2}} \nonumber \\
& \hspace{11mm} + C_{0} L^{\frac{1}{2}}  e^{2L^{\frac{1}{4}}} M_{0}(t) \lambda_{q}^{3- \gamma_{2}} \delta_{q} + C_{0}^{2} e^{L^{\frac{1}{4}}} m_{L}^{2} M_{0}(t) \delta_{q} \lambda_{q}^{3-\gamma_{2}}  \nonumber \\
& \hspace{11mm} + e^{2L^{\frac{1}{4}}}  M_{0}(t) \lambda_{q+1}^{3-\gamma_{2}} \delta_{q+1} +  C_{0} e^{\frac{3}{2}L^{\frac{1}{4}}}  m_{L} \lambda_{q+1}^{2-\gamma_{2}} M_{0}(t) \delta_{q+1}^{\frac{1}{2}} \delta_{q}^{\frac{1}{2}} \lambda_{q} \nonumber \\
& \hspace{5mm} \overset{\eqref{est 210} \eqref{est 32} \eqref{est 475}}{\lesssim} L^{\frac{1}{2}}  e^{2L^{\frac{1}{4}}} M_{0}(t) \lambda_{q+1}^{3-\gamma_{2}} \delta_{q+1} [ a^{b^{q+1} [ \frac{-3+ \gamma_{2}}{2} + \frac{\beta}{2} + \frac{\alpha}{2}]} + C_{0}a^{L + b^{q} [b(-2+ \gamma_{2} + \beta) + 2 - \gamma_{2} - \beta]} + 1]  \nonumber \\
& \hspace{9mm} \leq C_{0}   L^{\frac{1}{2}}e^{2L^{\frac{1}{4}}} M_{0}(t) \lambda_{q+1}^{3-\gamma_{2}} \delta_{q+1}, \label{est 483}
\end{align} 
where at the last inequality, we relied on the fact that  $-3 + \gamma_{2} + \beta + \alpha < \frac{ - 4 + 2 \gamma_{2}}{11}$ due to \eqref{est 315c} and 
\begin{equation}\label{est 466} 
b > \frac{L + 2 - \gamma_{2} - \beta}{2- \gamma_{2} - \beta} 
\end{equation} 
due to \eqref{est 446} where $2 - \gamma_{2} - \beta > 0$ due to \eqref{est 315b}. 
\end{proof} 

\begin{proposition}\label{Proposition 5.4} 
Let $a_{k,j}$ and $\bar{a}_{k,j}$ be defined by \eqref{est 484}; moreover, let $\psi_{q+1, j, k}$ be defined by \eqref{est 200b}. Then they satisfy the following estimates: for all $t \in \supp \chi_{j}$, 
\begin{subequations}\label{est 360}
\begin{align}
& \lVert D^{N} a_{k,j} (t) \rVert_{C_{x}} \lesssim M_{0}(t)^{\frac{1}{2}} \lambda_{q}^{N} \delta_{q+1}^{\frac{1}{2}} \hspace{1mm} \forall \hspace{1mm} N \in \mathbb{N}_{0}, \label{est 360c}\\
& \lVert D^{N} \bar{a}_{k,j} (t) \rVert_{C_{x}} \lesssim e^{\frac{1}{2} L^{\frac{1}{4}}} M_{0}(t)^{\frac{1}{2}} \lambda_{q}^{N} \delta_{q+1}^{\frac{1}{2}} \hspace{1mm} \forall \hspace{1mm} N \in \mathbb{N}_{0}, \label{est 360a}\\
&\lVert \psi_{q+1, j,k} \rVert_{C_{t,x,q+1}} \leq 1, \label{est 360d}\\
& \lVert D^{N} \psi_{q+1, j, k} (t) \rVert_{C_{x}} \lesssim m_{L} e^{L^{\frac{1}{4}}} M_{0}(t)^{\frac{1}{2}}  \lambda_{q+1} \tau_{q+1} \lambda_{q}^{N+ 2 - \gamma_{2}}  \delta_{q}^{\frac{1}{2}} \hspace{1mm} \forall \hspace{1mm} N \in \mathbb{N}. \label{est 360b}
\end{align} 
\end{subequations}
One of the useful consequences includes the following estimate: for all $t \in \supp \chi_{j}$, 
\begin{equation} \label{est 424}
\lVert \nabla (a_{k,j} \psi_{q+1, j, k})(t) \rVert_{C_{x}}  \lesssim  M_{0}(t)^{\frac{1}{2}} \lambda_{q} \delta_{q+1}^{\frac{1}{2}}.
\end{equation} 
\end{proposition} 

\begin{proof}[Proof of Proposition \ref{Proposition 5.4}]
It is immediate from \eqref{est 484} that 
\begin{equation*}
\lVert a_{k,j} (t) \rVert_{C_{x}} \lesssim M_{0}(t)^{\frac{1}{2}} \delta_{q+1}^{\frac{1}{2}} \hspace{1mm} \text{ and } \hspace{1mm} \lVert \bar{a}_{k,j} (t) \rVert_{C_{x}} \overset{\eqref{est 332}}{\lesssim} e^{\frac{1}{2} L^{\frac{1}{4}}} M_{0}(t)^{\frac{1}{2}} \delta_{q+1}^{\frac{1}{2}}.
\end{equation*} 
For $N \in \mathbb{N}$ we compute by the chain rule estimate (e.g. \cite[Equation (130)]{BDIS15}) and \eqref{est 332}
\begin{align} 
& \lVert D^{N} \bar{a}_{k,j}(t) \rVert_{C_{x}} \lesssim e^{\frac{1}{2} L^{\frac{1}{4}}}  M_{0}(\tau_{q+1} j)^{\frac{1}{2}}  \delta_{q+1}^{\frac{1}{2}} \label{est 361}\\
& \hspace{6mm}  \times \left[ \frac{1}{\lambda_{q+1}^{2-\gamma_{2}} \delta_{q+1} M_{0}(\tau_{q+1} j)} \lVert D^{N} \mathring{R}_{q,j} \rVert_{C_{t,x,q+1}} + \left( \frac{1}{\lambda_{q+1}^{2-\gamma_{2}} \delta_{q+1} M_{0}(\tau_{q+1} j)} \right)^{N} \lVert D\mathring{R}_{q,j} \rVert_{C_{t,x,q+1}}^{N} \right] \nonumber 
\end{align}
and the estimate for $\lVert D^{N} a_{k,j}(t) \rVert_{C_{x}}$ is identical except it has no factor of $e^{\frac{1}{2} L^{\frac{1}{4}}}$. The inequality of 
\begin{equation*}
\alpha + 3 - \gamma_{2} - \beta > \frac{38 - 8 \gamma_{2}}{11} 
\end{equation*}
that can be verified via \eqref{est 315a}, \eqref{est 315b}, and \eqref{est 315c} allows us to observe 
\begin{align}
b > \frac{ 2 [ L^{2} + 3 - \gamma_{2} - \beta]}{\alpha + 3 - \gamma_{2} - \beta} \label{est 467}
\end{align}
due to \eqref{est 446}; in turn, this allows us to estimate 
\begin{align}\label{est 362} 
\tau_{q+1} \lVert D \Upsilon_{l} \Lambda^{2- \gamma_{2}} y_{l} \rVert_{C_{t,x,q+1}} \overset{\eqref{est 332} \eqref{est 317}}{\lesssim}& \tau_{q+1} e^{L^{\frac{1}{4}}} \lambda_{q} ( m_{L}  M_{0}(t)^{\frac{1}{2}} \delta_{q}^{\frac{1}{2}} \lambda_{q}^{2- \gamma_{2}})  \nonumber \\
\overset{\eqref{est 210} \eqref{est 32}\eqref{est 475}}{\lesssim}&  a^{L^{2} + b^{q} [b( - \frac{\alpha}{2} - \frac{3}{2} + \frac{\gamma_{2}}{2} + \frac{\beta}{2}) + 3 - \gamma_{2} - \beta]} \lesssim 1
\end{align}  
which leads us to estimate by relying on \eqref{est 197b} and \eqref{est 318}  
\begin{equation}\label{est 363}
\lVert D \mathring{R}_{q,j} (t) \rVert_{C_{x}} \leq \lVert D \mathring{R}_{l} (\tau_{q+1} j) \rVert_{C_{x}} e^{(t - \tau_{q+1} j) \lVert D \Upsilon_{l} \Lambda^{2-\gamma_{2}} y_{l} \rVert_{C_{t,x,q+1}}}  \lesssim\lambda_{q} M_{0} (\tau_{q+1} j)  \lambda_{q+1}^{2- \gamma_{2}} \delta_{q+1} 
\end{equation} 
and for $N \in \mathbb{N}\setminus \{1\}$ 
\begin{align}
\lVert D^{N} \mathring{R}_{q,j}(t) \rVert_{C_{x}} & \overset{\eqref{est 198} \eqref{est 332}}{\lesssim}  \lambda_{q}^{N} \lVert \mathring{R}_{q}(\tau_{q+1} j) \rVert_{C_{x}} + \tau_{q+1} \lambda_{q}^{N} e^{L^{\frac{1}{4}}} \lVert \Lambda^{2-\gamma_{2}} y_{q} \rVert_{C_{t,x,q+1}} \lambda_{q}\lVert \mathring{R}_{q}(\tau_{q+1} j) \rVert_{C_{x}}  \nonumber \\
&  \overset{\eqref{est 317} \eqref{est 210} \eqref{est 475}}{\lesssim}  \lambda_{q}^{N} M_{0} (\tau_{q+1} j) \lambda_{q+1}^{2- \gamma_{2}} \delta_{q+1} [ 1 + a^{b^{q} [ b(-\frac{\alpha}{2} - \frac{3}{2} + \frac{\gamma_{2}}{2} + \frac{\beta}{2} ) + 3 - \gamma_{2} - \beta]} a^{L^{2}} ] \nonumber \\
& \hspace{40mm} \lesssim \lambda_{q}^{N} M_{0}(\tau_{q+1} j) \lambda_{q+1}^{2-\gamma_{2}} \delta_{q+1}. \label{est 364} 
\end{align}
Applying \eqref{est 363}-\eqref{est 364} to \eqref{est 361} allows us to conclude \eqref{est 360a}, and the estimate of \eqref{est 360c} follows too. 

Next, \eqref{est 360d} is clear from \eqref{est 200b} and hence we now consider \eqref{est 360b} in the case $N = 1$ and use \cite[Equation (130)]{BDIS15} to estimate 
\begin{align*}
\lVert D \psi_{q+1, j, k} (t) \rVert_{C_{x}} \overset{\eqref{est 200b}}{\lesssim}& \lambda_{q+1} \lVert \nabla \Phi_{j}(t) - \Id \rVert_{C_{x}}  \nonumber\\
&\overset{\eqref{est 199a} \eqref{est 332} \eqref{est 362}\eqref{est 317}}{\lesssim} \lambda_{q+1} \tau_{q+1} m_{L} e^{L^{\frac{1}{4}}} M_{0}(t)^{\frac{1}{2}}  \lambda_{q}^{3-\gamma_{2}} \delta_{q}^{\frac{1}{2}}.
\end{align*}
For $N \in \mathbb{N} \setminus \{1\}$, we estimate by \cite[Equation (130)]{BDIS15} 
\begin{align*}
&\lVert D^{N} \psi_{q+1, j, k}(t)  \rVert_{C_{x}}  \\
&\overset{\eqref{est 200b}}{\lesssim} \lambda_{q+1} \lVert \Phi_{j}(t) \rVert_{C_{x}^{N}}  + \lambda_{q+1}^{N} \lVert \nabla \Phi_{j} (t) - \Id \rVert_{C_{x}}^{N} \nonumber \\
&\overset{\eqref{est 199}\eqref{est 362} }{\lesssim}   \lambda_{q}^{N} \left[ \lambda_{q+1} \tau_{q+1} e^{L^{\frac{1}{4}}} m_{L}  M_{0}(t)^{\frac{1}{2}} \delta_{q}^{\frac{1}{2}} \lambda_{q}^{2- \gamma_{2}}  + [\lambda_{q+1} \tau_{q+1} e^{L^{\frac{1}{4}}}  m_{L} M_{0}(t)^{\frac{1}{2}} \delta_{q}^{\frac{1}{2}} \lambda_{q}^{2-\gamma_{2}}]^{N} \right] \nonumber \\
& \hspace{5mm} \overset{\eqref{est 317}}{\lesssim} m_{L} e^{L^{\frac{1}{4}}} M_{0}(t)^{\frac{1}{2}}  \lambda_{q+1} \tau_{q+1} \lambda_{q}^{N+ 2 - \gamma_{2}}  \delta_{q}^{\frac{1}{2}} 
\end{align*}
where the last inequality also used the fact that $1 - \gamma_{2} + \alpha - \beta > \frac{16- 8 \gamma_{2}}{11}$ due to \eqref{est 315b} and \eqref{est 315c} allowing us 
\begin{align}\label{est 450} 
b > \frac{ 2[ 2 - \gamma_{2} - \beta + L^{2}]}{1- \gamma_{2} + \alpha - \beta}
\end{align}
that is guaranteed by \eqref{est 446} so that 
\begin{equation}\label{est 374}
e^{L^{\frac{1}{4}}} m_{L} M_{0}(t)^{\frac{1}{2}} \lambda_{q+1} \tau_{q+1} \delta_{q}^{\frac{1}{2}} \lambda_{q}^{2-\gamma_{2}}\lesssim a^{b^{q} [ b (\frac{ - 1 + \gamma_{2}}{2} - \frac{\alpha}{2} + \frac{\beta}{2}) + 2 - \gamma_{2} - \beta] + L^{2}}  \ll 1.
\end{equation} 
At last, to prove \eqref{est 424}, we rely on \eqref{est 360c} to obtain for all $t \in \supp \chi_{j}$, 
\begin{align*}
\lVert   \nabla (a_{k,j} \psi_{q+1, j, k})(t) \rVert_{C_{x}} \lesssim M_{0}(t)^{\frac{1}{2}} \lambda_{q} \delta_{q+1}^{\frac{1}{2}}  [ 1 + m_{L} e^{L^{\frac{1}{4}}} M_{0}(t)^{\frac{1}{2}} \lambda_{q+1} \tau_{q+1} \lambda_{q}^{2- \gamma_{2}}  \delta_{q}^{\frac{1}{2}} ] \lesssim  M_{0}(t)^{\frac{1}{2}} \lambda_{q} \delta_{q+1}^{\frac{1}{2}}, 
\end{align*}
where the last inequality utilized \eqref{est 374}. 
\end{proof} 

\subsubsection{Bounds on $R_{T}$, the transport error}
From \eqref{est 345a} we write using \eqref{est 208} and \eqref{est 348}   
\begin{align*}
R_{T} =&   \sum_{j,k} 1_{\supp \chi_{j}} \mathcal{B} [ [\Upsilon_{l} \Lambda^{2- \gamma_{2}} y_{l} \cdot \nabla, \mathbb{P}_{q+1, k} ] \tilde{w}_{q+1, j, k}  \nonumber\\
& \hspace{5mm} + \mathbb{P}_{q+1, k} ([\partial_{t} + \Upsilon_{l} \Lambda^{2- \gamma_{2}} y_{l} \cdot \nabla] ( \chi_{j} \bar{a}_{k,j}  b_{k} (\lambda_{q+1} \Phi_{j})))]
\end{align*}
where we can use \eqref{est 356}-\eqref{est 357} to rewrite 
\begin{align*}
& [\partial_{t} + \Upsilon_{l} \Lambda^{2- \gamma_{2}} y_{l} \cdot \nabla ] ( \chi_{j} \bar{a}_{k,j}  b_{k} (\lambda_{q+1} \Phi_{j}) ) \nonumber \\
=& \partial_{t} \chi_{j} \bar{a}_{k,j}  b_{k} (\lambda_{q+1} \Phi_{j}) \nonumber \\
&+ \chi_{j}  \partial_{t} \Upsilon_{l}^{-\frac{1}{2}} \frac{M_{0} (\tau_{q+1} j)^{\frac{1}{2}}}{\sqrt{\gamma_{2}}} \delta_{q+1}^{\frac{1}{2}} \gamma_{k} \left( \Id - \frac{ \mathring{R}_{q,j}}{\lambda_{q+1}^{2-\gamma_{2}} \delta_{q+1} M_{0} (\tau_{q+1} j)} \right) b_{k} (\lambda_{q+1} \Phi_{j}) 
\end{align*}
so that we arrive at 
\begin{align*}
&R_{T} =  \sum_{j,k} 1_{\supp \chi_{j}} \mathcal{B} [[\Upsilon_{l} \Lambda^{2- \gamma_{2}} y_{l} \cdot \nabla, \mathbb{P}_{q+1, k} ] \tilde{w}_{q+1, j, k} 
 + \mathbb{P}_{q+1, k} (\partial_{t} \chi_{j} \bar{a}_{k,j}  b_{k} (\lambda_{q+1} \Phi_{j}))   \\
 &+ \mathbb{P}_{q+1,k} \left( \chi_{j} \partial_{t} \Upsilon_{l}^{-\frac{1}{2}} \frac{M_{0} (\tau_{q+1} j)^{\frac{1}{2}}}{\sqrt{\gamma_{2}}} \delta_{q+1}^{\frac{1}{2}} \gamma_{k} \left( \Id - \frac{\mathring{R}_{q,j} }{\lambda_{q+1}^{2-\gamma_{2}} \delta_{q+1}  M_{0}(\tau_{q+1} j)} \right) b_{k} (\lambda_{q+1} \Phi_{j}) \right) ]. \nonumber
\end{align*}
By definitions of $\mathbb{P}_{q+1, k}$ and $\tilde{P}_{\approx \lambda_{q+1}}$ from \eqref{est 192} and Section \ref{Section 5.1}, as long as $a \geq \frac{16}{5}$ and $b \geq 2$, we can write  
\begin{align}
R_{T} =& \sum_{j,k} 1_{\supp \chi_{j}} \mathcal{B} \tilde{P}_{\approx \lambda_{q+1}} [ [ \Upsilon_{l} \Lambda^{2- \gamma_{2}} y_{l} \cdot \nabla, \mathbb{P}_{q+1, k} ] \tilde{w}_{q+1, j, k} \label{est 366}\\
& + \mathbb{P}_{q+1, k} (\partial_{t} \chi_{j}  \bar{a}_{k,j}  b_{k} (\lambda_{q+1} \Phi_{j})) \nonumber\\
&+ \mathbb{P}_{q+1, k} \left( \chi_{j} \partial_{t} \Upsilon_{l}^{-\frac{1}{2}} \frac{M_{0} (\tau_{q+1} j)^{\frac{1}{2}}}{\sqrt{\gamma_{2}}} \delta_{q+1}^{\frac{1}{2}} \gamma_{k} \left( \Id - \frac{\mathring{R}_{q,j} }{\lambda_{q+1}^{2-\gamma_{2}} \delta_{q+1}  M_{0}(\tau_{q+1} j)} \right) b_{k} (\lambda_{q+1} \Phi_{j}) \right)]. \nonumber 
\end{align}
Thus, we now estimate by \cite[Lemma A.6]{BSV19} for $a \in 5 \mathbb{N}$ sufficiently large 
\begin{align}
& \lVert R_{T} \rVert_{C_{t,x,q+1}} \label{est 397}\\
&\overset{\eqref{est 366} \eqref{est 360a} \eqref{est 332}}{\lesssim}  \lambda_{q+1}^{-1} \sum_{j,k}  ( \lVert \Upsilon_{l} \nabla \Lambda^{2-\gamma_{2}} y_{l} \rVert_{C_{t,x,q+1}} \lVert 1_{\supp \chi_{j}} \tilde{w}_{q+1, j, k} \rVert_{C_{t,x,q+1}}   \nonumber \\
& \hspace{44mm} + \tau_{q+1}^{-1} e^{\frac{1}{2} L^{\frac{1}{4}}}  M_{0}(t)^{\frac{1}{2}} \delta_{q+1}^{\frac{1}{2}} + e^{\frac{5}{2} L^{\frac{1}{4}}} l^{-1} M_{0}(t)^{\frac{1}{2}} \delta_{q+1}^{\frac{1}{2}} ) \nonumber  \\
&\overset{ \eqref{est 317} \eqref{est 360a}}{\lesssim}  \lambda_{q+1}^{-1} (e^{\frac{3}{2} L^{\frac{1}{4}}} m_{L} M_{0}(t) \lambda_{q}^{3- \gamma_{2}} \delta_{q}^{\frac{1}{2}} \delta_{q+1}^{\frac{1}{2}} + \tau_{q+1}^{-1} e^{\frac{1}{2} L^{\frac{1}{4}}}  M_{0}(t)^{\frac{1}{2}} \delta_{q+1}^{\frac{1}{2}} + e^{ \frac{5}{2} L^{\frac{1}{4}}}  l^{-1} M_{0}(t)^{\frac{1}{2}} \delta_{q+1}^{\frac{1}{2}})  \nonumber \\
& \hspace{12mm}\overset{\eqref{est 210} \eqref{est 32} \eqref{est 475}}{\lesssim}  M_{0}(t) \lambda_{q+2}^{2-\gamma_{2}} \delta_{q+2} [a^{b^{q}[ b^{2} (-2 + \gamma_{2} + 2 \beta) + b(-1 - \beta) + 3 - \gamma_{2} - \beta]} a^{L} \nonumber \\
& \hspace{10mm} + a^{b^{q} [ b^{2} (-2 + \gamma_{2} + 2 \beta) + b( \frac{1-\gamma_{2}}{2} + \frac{\alpha}{2} - \frac{3\beta}{2} ) ]} + a^{b^{q} [ b^{2} (-2 + \gamma_{2} + 2 \beta)+  b(-1 + \alpha - \beta) ]}  ]\ll M_{0}(t) \lambda_{q+2}^{2-\gamma_{2}} \delta_{q+2},  \nonumber 
\end{align}
where in the last inequality we used that 
\begin{equation}\label{est 468} 
b > \frac{3 - \gamma_{2} - \beta + L}{1+ \beta}
\end{equation} 
due to \eqref{est 446}, and that we have a strictly negative upper bounds of $1 - \gamma_{2} + \alpha - 3 \beta < -\frac{1}{2} + \frac{\gamma_{2}}{4}$, $-1 + \alpha - \beta = - \frac{\beta}{2}$ by \eqref{est 315c}, and we took $\beta > 1 - \frac{\gamma_{2}}{2}$ sufficiently close to $1 - \frac{\gamma_{2}}{2}$. 

\subsubsection{Bounds on $R_{N}$, the Nash error}
We define 
\begin{equation}\label{est 370}
N_{1} \triangleq \Upsilon_{l} \mathcal{B} [(\nabla \Lambda^{2- \gamma_{2}} y_{l})^{T} \cdot w_{q+1}] \text{ and } N_{2} \triangleq \Upsilon_{l} \mathcal{B} [ \Lambda^{2- \gamma_{2}} w_{q+1}^{\bot} (\nabla^{\bot} \cdot y_{l} )]
\end{equation} 
so that we can split $R_{N}$ from \eqref{est 345c} using \eqref{est 38} as 
\begin{equation}\label{est 369} 
R_{N} = N_{1} + N_{2}. 
\end{equation} 
For $N_{1}$, considering the frequency support of $y_{q}$ and thus the frequency support of $y_{l}$ due to Hypothesis \ref{Hypothesis 5.1} \ref{Hypothesis 5.1 (a)} and that of $w_{q+1}$ due to \eqref{est 163}, we rewrite it by \eqref{est 347} as 
\begin{equation*}
N_{1}(x) = \Upsilon_{l}  \sum_{j,k} \mathcal{B} \tilde{P}_{\approx \lambda_{q+1}} [( \nabla \Lambda^{2- \gamma_{2}} y_{l})^{T} \cdot \mathbb{P}_{q+1, k} ( \chi_{j} \bar{a}_{k,j} (x) b_{k} (\lambda_{q+1} \Phi_{j} (x)))].
\end{equation*} 
Thus, for $a \in 5 \mathbb{N}$ sufficiently large and $\beta > 1 - \frac{\gamma_{2}}{2}$ sufficiently close to $1 - \frac{\gamma_{2}}{2}$, we attain 
\begin{align}
\lVert N_{1}\rVert_{C_{t,x,q+1}} \overset{\eqref{est 332} \eqref{est 202}}{\lesssim}& e^{L^{\frac{1}{4}}} \lambda_{q+1}^{-1} \sum_{j,k} \lambda_{q} \lVert \Lambda^{2-\gamma_{2}} y_{q} \rVert_{C_{t,x,q+1}} \lVert 1_{\supp \chi_{j}} \bar{a}_{k,j}  \rVert_{C_{t,x,q+1}}   \label{est 372}  \\
\overset{\eqref{est 317} \eqref{est 360a}}{\lesssim}& m_{L} e^{\frac{3}{2}L^{\frac{1}{4}}}  M_{0}(t) \lambda_{q+2}^{2- \gamma_{2}} \delta_{q+2} \lambda_{q+2}^{-2 + \gamma_{2} + 2 \beta} \lambda_{q+1}^{-1-\beta} \lambda_{q}^{3- \gamma_{2} - \beta}  \nonumber\\
\overset{\eqref{est 475} }{\lesssim}& M_{0}(t) \lambda_{q+2}^{2 - \gamma_{2}} \delta_{q+2} a^{b^{q} [ b^{2} (-2 + \gamma_{2} + 2 \beta) + b(-1-\beta) + 3 - \gamma_{2} - \beta]} a^{L} \ll M_{0}(t) \lambda_{q+2}^{2-\gamma_{2}} \delta_{q+2}, \nonumber 
\end{align}
where the last inequality used the fact that 
\begin{align}\label{est 469}
b> \frac{3- \gamma_{2} - \beta + L}{1+ \beta} 
\end{align}
due to \eqref{est 446}. For $N_{2}$ we can first rewrite $w_{q+1}^{\bot}$ as 
\begin{align*}
w_{q+1}^{\bot}(t,x) &\overset{\eqref{est 208}\eqref{est 348}\eqref{est 178} \eqref{est 200b} }{=}   \sum_{j,k}   \mathbb{P}_{q+1, k} (\chi_{j}(t) \bar{a}_{k,j}(t,x)  (-ik e^{ik\cdot \lambda_{q+1} x}) \psi_{q+1, j, k}(t,x) ) \nonumber\\
\hspace{10mm} =&   \sum_{j,k}   \mathbb{P}_{q+1, k} (\chi_{j}(t) \bar{a}_{k,j}(t,x)  (-\frac{1}{\lambda_{q+1}} \nabla e^{ik \cdot \lambda_{q+1} x}) \psi_{q+1, j, k} (t,x) ) \nonumber \\
\hspace{10mm}\overset{\eqref{est 178}}{=}& - \lambda_{q+1}^{-1} \sum_{j,k} \nabla \mathbb{P}_{q+1, k} ( \chi_{j}(t) \bar{a}_{k,j} (t,x) \psi_{q+1, j, k} (t,x) c_{k} (\lambda_{q+1} x) ) \nonumber \\
&\hspace{10mm}+ \lambda_{q+1}^{-1} \sum_{j,k}  \mathbb{P}_{q+1, k} ( \chi_{j}(t) \nabla (\bar{a}_{k,j} (t,x) \psi_{q+1, j, k} (t,x)) c_{k} (\lambda_{q+1} x)).  
\end{align*}
This allows us to deduce by using \eqref{est 370}, \eqref{est 200b}, and the facts that $\mathcal{B}$ consists of $\mathbb{P}$ and $\mathbb{P} \nabla f = 0$ for all $f$, 
\begin{align}
N_{2}(x) =&\Upsilon_{l} \lambda_{q+1}^{-1}  \mathcal{B} ( \nabla ( \nabla^{\bot} \cdot y_{l})(x)  \sum_{j,k} \Lambda^{2- \gamma_{2}} \mathbb{P}_{q+1, k} (\chi_{j} \bar{a}_{k,j} (x) c_{k} (\lambda_{q+1} \Phi_{j} (x))) \nonumber \\
&+  \Upsilon_{l} \lambda_{q+1}^{-1} \mathcal{B} (( \nabla^{\bot} \cdot y_{l})(x) \sum_{j,k} \Lambda^{2-\gamma_{2}} \mathbb{P}_{q+1, k} (\chi_{j} \nabla (\bar{a}_{k,j} (x) \psi_{q+1, j, k} (x)) c_{k} (\lambda_{q+1} x))). \label{est 373}
\end{align} 
Now we are ready to compute from \eqref{est 373} by \cite[Equation (A.11)]{BSV19}, for $a \in 5 \mathbb{N}$ sufficiently large and $\beta > 1 - \frac{\gamma_{2}}{2}$ sufficiently close to $1 - \frac{\gamma_{2}}{2}$, 
\begin{align}
& \lVert N_{2} \rVert_{C_{t,x,q+1}} \overset{ \eqref{est 332}}{\lesssim} e^{L^{\frac{1}{4}}} \lambda_{q+1}^{-2} \sum_{j,k} \lVert \nabla(\nabla^{\bot} \cdot y_{l} ) \rVert_{C_{t,x,q+1}} \lVert \Lambda^{2- \gamma_{2}} \mathbb{P}_{q+1, k} ( \chi_{j} \bar{a}_{k,j}  c_{k} (\lambda_{q+1} \Phi_{j} (x))) \rVert_{C_{t,x,q+1}} \nonumber \\
& \hspace{15mm} + \lVert \nabla^{\bot}y_{l} \rVert_{C_{t,x,q+1}} \lVert \Lambda^{2- \gamma_{2}} \mathbb{P}_{q+1, k} (\chi_{j} \nabla (\bar{a}_{k,j} (x) \psi_{q+1, j, k} (x)) c_{k} (\lambda_{q+1} x) )  \rVert_{C_{t,x,q+1}}  \nonumber\\ 
& \hspace{5mm} \overset{\eqref{est 374}}{\lesssim} e^{L^{\frac{1}{4}}} \lambda_{q+1}^{-2}  \left[ m_{L} e^{\frac{1}{2} L^{\frac{1}{4}}} M_{0}(t) \lambda_{q}^{2} \delta_{q}^{\frac{1}{2}} \lambda_{q+1}^{2-\gamma_{2}} \delta_{q+1}^{\frac{1}{2}} + \lambda_{q} m_{L} M_{0}(t)^{\frac{1}{2}} \delta_{q}^{\frac{1}{2}} \lambda_{q+1}^{2- \gamma_{2}} [ \lambda_{q} \delta_{q+1}^{\frac{1}{2}}  e^{\frac{1}{2} L^{\frac{1}{4}}} M_{0}(t)^{\frac{1}{2}} ] \right] \nonumber \\
&\overset{\eqref{est 360} \eqref{est 374}}{\lesssim} M_{0}(t) \lambda_{q+2}^{2-\gamma_{2}} \delta_{q+2} a^{b^{q} [ b^{2} (-2 + \gamma_{2} + 2 \beta) + b(-\gamma_{2} - \beta) + 2-\beta]  + L}    \ll M_{0}(t) \lambda_{q+2}^{2- \gamma_{2}} \delta_{q+2},  \label{est 375}
\end{align}
where the last inequality used the fact that
\begin{equation}\label{est 476}
b > \frac{2 - \beta + L}{\gamma_{2} + \beta} 
\end{equation} 
due to \eqref{est 446}. Applying \eqref{est 372} and \eqref{est 375} to \eqref{est 369}  allows us to conclude 
\begin{equation}\label{est 376}
\lVert R_{N} \rVert_{C_{t,x,q+1}} \ll M_{0}(t) \lambda_{q+2}^{2- \gamma_{2}}\delta_{q+2}. 
\end{equation} 

\subsubsection{Bounds on $R_{L}$, the linear error}
We split $R_{L}$ rom \eqref{est 345d} as follows: 
\begin{equation}\label{est 377} 
R_{L} = L_{1} + L_{2} \hspace{1mm} \text{ where } L_{1} \triangleq \mathcal{B} \Lambda^{\gamma_{1}} w_{q+1} \hspace{1mm} \text{ and } \hspace{1mm} L_{2} \triangleq \frac{1}{2} \mathcal{B}  w_{q+1}.
\end{equation} 
We can estimate from \eqref{est 377}  for $a \in 5 \mathbb{N}$ sufficiently large and $\beta > 1 - \frac{\gamma_{2}}{2}$ sufficiently close to $1- \frac{\gamma_{2}}{2}$, 
\begin{align}
\lVert L_{1} \rVert_{C_{t,x,q+1}} \overset{\eqref{est 163}}{\lesssim}& \lambda_{q+1}^{\gamma_{1} -1} \lVert w_{q+1} \rVert_{C_{t,x,q+1}} \overset{ \eqref{est 350a}}{\lesssim} \lambda_{q+1}^{\gamma_{1} -1} e^{\frac{1}{2} L^{\frac{1}{4}}}  M_{0}(t)^{\frac{1}{2}} \delta_{q+1}^{\frac{1}{2}} \nonumber\\  
\overset{\eqref{est 35} \eqref{est 315b}}{\lesssim}&  M_{0}(t) \lambda_{q+2}^{2-\gamma_{2}} \delta_{q+2} a^{b^{q+1} [b(-2 + \gamma_{2} +  2 \beta) + \gamma_{1} +  \frac{\gamma_{2}}{2} -2]} \ll M_{0}(t) \lambda_{q+2}^{2-\gamma_{2}} \delta_{q+2}, \label{est 379}  
\end{align} 
where we used in the last inequality the fact that $\gamma_{1} < 2 - \frac{\gamma_{2}}{2}$ due to \eqref{est 315d}. Similarly, starting from \eqref{est 377}, for $\beta > 1 - \frac{\gamma_{2}}{2}$ sufficiently close to $1 - \frac{\gamma_{2}}{2}$ so that 
\begin{align*}
b(-2 + \gamma_{2} + 2 \beta) - 1 - \beta < 0
\end{align*}
and $a \in 5 \mathbb{N}$ sufficiently large, we can estimate
\begin{align}
\lVert L_{2} \rVert_{C_{t,x,q+1}}\overset{\eqref{est 163}}{\lesssim}&  \lambda_{q+1}^{-1} \lVert w_{q+1} \rVert_{C_{t,x,q+1}} \nonumber \\
\overset{\eqref{est 350a}}{\lesssim}&  M_{0}(t) \lambda_{q+2}^{2-\gamma_{2}} \delta_{q+2} a^{b^{q+1} [ b(-2 + \gamma_{2} + 2 \beta) - 1 - \beta]} \ll   M_{0}(t) \lambda_{q+2}^{2-\gamma_{2}}\delta_{q+2}. \label{est 380}
\end{align} 
By applying \eqref{est 379} and \eqref{est 380} to \eqref{est 377}, we can conclude
\begin{equation}\label{est 381} 
\lVert R_{L} \rVert_{C_{t,x,q+1}} \ll  M_{0}(t) \lambda_{q+2}^{2-\gamma_{2}} \delta_{q+2}.
\end{equation} 

\subsubsection{Bounds on $R_{O}$, the oscillation error}
Because $\bar{a}_{k,j} = \Upsilon_{l}^{-\frac{1}{2}} a_{k,j}$ by \eqref{est 351}, we have
\begin{align*}
&\divergence R_{O} \overset{\eqref{est 345b} \eqref{est 347}}{=} \divergence \mathring{R}_{l} \\
&+ \Upsilon_{l} [ ( \Lambda^{2-\gamma_{2}}  \sum_{j,k} \chi_{j} \mathbb{P}_{q+1, k} ( \bar{a}_{k,j} b_{k} (\lambda_{q+1} \Phi_{j} )) \cdot \nabla) \sum_{j', k'} \chi_{j'} \mathbb{P}_{q+1, k'} (\bar{a}_{k', j'} b_{k'} (\lambda_{q+1} \Phi_{j'})) \nonumber \\
& \hspace{5mm} - (\nabla  \sum_{j,k} \chi_{j} \mathbb{P}_{q+1, k} (\bar{a}_{k,j} b_{k} (\lambda_{q+1} \Phi_{j} )))^{T} \cdot \Lambda^{2-\gamma_{2}} \sum_{j', k'} \chi_{j'} \mathbb{P}_{q+1, k'} (\bar{a}_{k', j'} b_{k'} (\lambda_{q+1} \Phi_{j'} ))] \nonumber \\
& \hspace{20mm} \overset{\eqref{est 351}}{=}  \divergence \mathring{R}_{l} \nonumber \\
&+ [ ( \Lambda^{2-\gamma_{2}}  \sum_{j,k} \chi_{j} \mathbb{P}_{q+1, k} ( a_{k,j} b_{k} (\lambda_{q+1} \Phi_{j} )) \cdot \nabla) \sum_{j',k'}\chi_{j'} \mathbb{P}_{q+1, k'} (a_{k', j'} b_{k'} (\lambda_{q+1} \Phi_{j'}))  \nonumber \\
& \hspace{5mm} - (\nabla  \sum_{j,k} \chi_{j} \mathbb{P}_{q+1, k} (a_{k,j} b_{k} (\lambda_{q+1} \Phi_{j} )))^{T} \cdot \Lambda^{2-\gamma_{2}} \sum_{j',k'}\chi_{j'} \mathbb{P}_{q+1, k'} (a_{k', j'} b_{k'} (\lambda_{q+1} \Phi_{j'} ))] \nonumber 
\end{align*}
so that the current oscillation term has an identical structure to that in the additive case; cf. \cite[Equations (125d) and (130)]{Y23a}: 
\begin{align*}
\text{``}  \divergence R_{O} \triangleq& \divergence \mathring{R}_{l} + (\Lambda^{2- \gamma_{2}} w_{q+1} \cdot \nabla) w_{q+1} - (\nabla w_{q+1})^{T} \cdot \Lambda^{2- \gamma_{2}} w_{q+1}\text{'''}
\end{align*}
and 
\begin{align*}
\text{``} w_{q+1}(t,x) \triangleq \sum_{j,k} \chi_{j}(t) \mathbb{P}_{q+1,k}(a_{k,j} (t,x) b_{k} (\lambda_{q+1} \Phi_{j}(t,x))).\text{''}
\end{align*}
To be more precise, $a_{k,j}$ in \cite[Equation (132)]{Y23a}  has a factor of $\frac{ \sqrt{\bar{C}} L^{\frac{1}{2}}}{\sqrt{\epsilon_{\gamma}}}$ and $\frac{\epsilon_{\gamma}}{\lambda_{q+1}^{2-\gamma_{2}} \bar{C}L}$ respectively inside and outside the argument of $\gamma_{k}$ while our  $\mathring{R}_{q,j}$ within the definition of $a_{k,j}$ is defined by \eqref{est 355} rather than \cite[Equation (131)]{Y23a}, but this difference will not be of great significance and we will be able to show 
\begin{equation}\label{est 307} 
e^{2L^{\frac{1}{4}}}\lVert R_{O} \rVert_{C_{t,x,q+1}} \ll M_{0}(t) \lambda_{q+2}^{2- \gamma_{2}} \delta_{q+2}
\end{equation} 
where we chose to include the exponential term $e^{2L^{\frac{1}{4}}}$ for the subsequent convenience in the estimate of $R_{\text{Com2}}$. We leave the details of the derivation of \eqref{est 307} in the Appendix for completeness as it is still more complex than the case of the Navier-Stokes equations and we changed the definition of $a_{k,j}$ in \eqref{est 447} and many more from those in the additive case (cf. \cite[Equation (132)]{Y23a}).  

\subsubsection{Bounds on $R_{\text{Com1}}$, the first commutator error}
We estimate from \eqref{est 338} using $\lVert \mathcal{B} \rVert_{C_{x} \mapsto C_{x}} \lesssim 1$ and the standard commutator estimate (e.g. \cite[Equation (5)]{CDS12b}),  
\begin{align}
&\lVert R_{\text{Com1}} \rVert_{C_{t,x,q+1}} \overset{\eqref{est 332}}{\lesssim}  l^{2} \lVert \Lambda^{2- \gamma_{2}} y_{q}^{\bot} \Upsilon^{\frac{1}{2}} \rVert_{C_{t,q+1 }C_{x}^{1}} \lVert \nabla^{\bot} \cdot y_{q} \Upsilon^{\frac{1}{2}} \rVert_{C_{t, q+1 }C_{x}^{1}} \label{est 339} \\
& \hspace{25mm} + l^{1-4\delta} \lVert \Upsilon \rVert_{C_{t, q+1}^{\frac{1}{2} - 2 \delta}} \lVert \Lambda^{2-\gamma_{2}} y_{q}^{\bot} \ast_{x} \phi_{l} (\nabla^{\bot} \cdot y_{q} \ast_{x} \phi_{l}) \rVert_{C_{t, q+1}^{\frac{1}{2} - 2 \delta}C_{x}}  \nonumber \\
& \hspace{25mm} + e^{L^{\frac{1}{4}}} l^{2} \lVert \Lambda^{2-\gamma_{2}} y_{q}^{\bot} \ast_{x} \phi_{l} \rVert_{C_{t, q+1}^{1}C_{x}} \lVert \nabla^{\bot}\cdot y_{q} \ast_{x} \phi_{l} \rVert_{C_{t, q+1}^{1}C_{x}}  \nonumber  \\
& \hspace{7mm} \overset{\eqref{est 332} \eqref{est 317} \eqref{est 330} }{\lesssim}  l^{2} e^{L^{\frac{1}{4}}}m_{L}^{2} \lambda_{q}^{5- \gamma_{2}} M_{0}(t) \delta_{q} \nonumber \\
&+ l^{1- 4\delta} e^{L^{\frac{1}{4}}} (\lambda_{q}^{2- \gamma_{2}} \lVert y_{q} \rVert_{C_{t, q+1}^{1}C_{x}} \lambda_{q}^{\gamma_{2} -1}   \lVert \Lambda^{2-\gamma_{2}} y_{q} \rVert_{C_{t,x, q+1}}  )^{\frac{1}{2} - 2 \delta} [\lambda_{q}^{\gamma_{2} -1} (m_{L}  M_{0}(t)^{\frac{1}{2}}\delta_{q}^{\frac{1}{2}} \lambda_{q}^{2-\gamma_{2}}  )^{2} ]^{\frac{1}{2}+ 2 \delta} \nonumber \\
& + e^{3L^{\frac{1}{4}}} m_{L}^{4} l^{2}  M_{0}(t)^{2} \lambda_{q}^{9- 3 \gamma_{2}} \delta_{q}^{2}  \nonumber \\
&\overset{\eqref{est 317} \eqref{est 330} }{\lesssim}  M_{0}(t) [ e^{L^{\frac{1}{4}}} m_{L}^{2}  l^{2} \lambda_{q}^{5- \gamma_{2}}  \delta_{q}  \nonumber \\
& \hspace{20mm}+  M_{0}(t)^{\frac{3}{4} - \delta} m_{L}^{\frac{9}{2} - 2 \delta} e^{(\frac{1}{2} - 2 \delta)L^{\frac{1}{4}}} l^{1-4\delta}  \lambda_{q}^{\frac{9-3 \gamma_{2}}{2} + (\gamma_{2} -3) 2 \delta}  \delta_{q}^{\frac{5}{4} - \delta} +   e^{3L^{\frac{1}{4}}} m_{L}^{4} M_{0}(t) l^{2} \lambda_{q}^{9-3 \gamma_{2}} \delta_{q}^{2} ]. \nonumber 
\end{align}
Within the right hand side of \eqref{est 339} we can bound the first term of $e^{L^{\frac{1}{4}}} m_{L}^{2} l^{2}  \lambda_{q}^{5- \gamma_{2}}  \delta_{q}$ by the third term of $e^{3L^{\frac{1}{4}}}  m_{L}^{4} M_{0}(t)l^{2} \lambda_{q}^{9-3 \gamma_{2}} \delta_{q}^{2}$ using the fact that $2 - \gamma_{2} > \beta$ due to \eqref{est 315b}. Concerning the second term of $M_{0}(t)^{\frac{3}{4} - \delta} m_{L}^{\frac{9}{2} - 2 \delta} e^{(\frac{1}{2} - 2 \delta)L^{\frac{1}{4}}} l^{1-4\delta} \lambda_{q}^{\frac{9-3 \gamma_{2}}{2} + (\gamma_{2} -3) 2 \delta} \delta_{q}^{\frac{5}{4} - \delta}$ and the third term of $e^{3L^{\frac{1}{4}}} m_{L}^{4} M_{0}(t)l^{2} \lambda_{q}^{9-3 \gamma_{2}} \delta_{q}^{2}$ in \eqref{est 339} we realize that 
\begin{align*}
&M_{0}(t)^{\frac{3}{4}} m_{L}^{\frac{9}{2}} e^{\frac{1}{2}L^{\frac{1}{4}}}  l \lambda_{q}^{\frac{9-3\gamma_{2}}{2}} \delta_{q}^{\frac{5}{4}} \overset{\eqref{est 32} \eqref{est 35} \eqref{est 475}}{\lesssim} a^{L^{2}} a^{b^{q} [ -b \alpha + \frac{9-3\gamma_{2} - 5 \beta}{2} + b^{2} (-2 + \gamma_{2} + 2 \beta) ]}  \lambda_{q+2}^{2- \gamma_{2}} \delta_{q+2} \ll \lambda_{q+2}^{2-\gamma_{2}} \delta_{q+2}, \\
& e^{3L^{\frac{1}{4}}} m_{L}^{4} M_{0}(t)  l^{2} \lambda_{q}^{9- 3\gamma_{2}} \delta_{q}^{2} \overset{\eqref{est 32} \eqref{est 35} \eqref{est 475}}{\lesssim}  a^{L^{2}} a^{b^{q} [-2 \alpha b + 9 - 3 \gamma_{2} - 4 \beta + b^{2} (-2 + \gamma_{2} + 2 \beta) ]} \lambda_{q+2}^{2-\gamma_{2}} \delta_{q+2}  \ll \lambda_{q+2}^{2-\gamma_{2}} \delta_{q+2}
\end{align*}  
where the last inequalities are by 
\begin{align}\label{est 472}
b > \frac{1}{\alpha} \left[ L^{2} + \frac{9 - 3 \gamma_{2} - 5 \beta}{2}\right] \hspace{1mm}  \text{ and } \hspace{1mm} b > \frac{L^{2} + 9 - 3 \gamma_{2} - 4 \beta}{2\alpha}
\end{align}
due to \eqref{est 446}. Therefore, taking $\delta \in (0, \frac{1}{4})$ sufficiently small enough, $a \in 5 \mathbb{N}$ sufficiently large, and $\beta > 1 - \frac{\gamma_{2}}{2}$ sufficiently close to $1 -\frac{\gamma_{2}}{2}$ allows us to conclude from \eqref{est 339} that 
\begin{align}
\lVert R_{\text{Com1}} \rVert_{C_{t,x, q+1}}  \ll  M_{0}(t) \lambda_{q+2}^{2-\gamma_{2}} \delta_{q+2}. \label{est 398}
\end{align} 

\subsubsection{Bounds on $R_{\text{Com2}}$, the stochastic commutator error}\label{Section 5.2.6}
We continue from the decomposition of the stochastic commutator error $R_{\text{Com2}}$ to $\sum_{k=1}^{3} R_{\text{Com2,k}}$ in \eqref{est 389}-\eqref{est 390}.  We estimate via Young's inequality for convolution, for $a \in 5 \mathbb{N}$ sufficiently large and $\beta > 1- \frac{\gamma_{2}}{2}$ sufficiently close to $1 - \frac{\gamma_{2}}{2}$,  
\begin{align}
&\lVert R_{\text{Com2,1}} \rVert_{C_{t,x,q+1}} \overset{\eqref{est 332}}{\lesssim}  l^{\frac{1}{2} - 2 \delta} \lVert \Upsilon \rVert_{C_{t,q+1}^{\frac{1}{2} - 2 \delta}}  e^{L^{\frac{1}{4}}} \lVert \mathring{R}_{q} \rVert_{C_{t,x,q+1}} \label{est 393} \\
&\overset{\eqref{est 391a} \eqref{est 332} \eqref{est 32} \eqref{est 318}}{\lesssim} \lambda_{q+1}^{-\alpha (\frac{1}{2} - 2 \delta)}  e^{2L^{\frac{1}{4}}}M_{0}(t) \lambda_{q+1}^{2-\gamma_{2}} \delta_{q+1}  \nonumber  \\
& \hspace{10mm}  \lesssim   M_{0}(t) \lambda_{q+2}^{2-\gamma_{2}} \delta_{q+2} a^{b^{q} [b^{2} (-2 + \gamma_{2} + 2 \beta) + b( - \alpha (\frac{1}{2} - 2 \delta) + 2 - \gamma_{2} - 2 \beta ) ]} a^{L}  \ll M_{0}(t) \lambda_{q+2}^{2-\gamma_{2}} \delta_{q+2} \nonumber
\end{align}
where the last inequality relied on the fact that $- \frac{\alpha}{2} + 2 - \gamma_{2} - 2 \beta < -\frac{3}{4} + \frac{\gamma_{2}}{8} < 0$ due to \eqref{est 315b}-\eqref{est 315c} and 
\begin{equation}\label{est 473}
b> \frac{L}{\alpha ( \frac{1}{2} - 2 \delta) - 2 +\gamma_{2} + 2\beta} 
\end{equation} 
for all $\delta \in (0, \frac{1}{8})$ due to \eqref{est 446}. 

Next, for $a \in 5 \mathbb{N}$ sufficiently large 
\begin{align}\label{est 394}   
\lVert R_{\text{Com2,2}} \rVert_{C_{t,x,q+1}} \overset{\eqref{est 391b}}{=}& \lVert (\Upsilon - \Upsilon_{l}) \Upsilon_{l}^{-1} R_{0} \rVert_{C_{t,x,q+1}}   \nonumber\\
\overset{\eqref{est 332}}{\leq}& 2e^{2L^{\frac{1}{4}}} \lVert R_{O} \rVert_{C_{t,x,q+1}} \overset{\eqref{est 307}}{\ll}   M_{0}(t) \lambda_{q+2}^{2-\gamma_{2}} \delta_{q+2}
\end{align} 
where we used $e^{2L^{\frac{1}{4}}} \lVert R_{O} \rVert_{C_{t,x, q+1}} \ll  M_{0}(t)\lambda_{q+2}^{2-\gamma_{2}} \delta_{q+2}$ in \eqref{est 307}. 

Finally, concerning $R_{\text{Com2,3}}$ of \eqref{est 391c}, similarly to how we handled the estimates of $R_{T}$, we notice that we may apply $\tilde{P}_{\approx \lambda_{q+1}}$ to the products of $w_{q+1}$ and $y_{l}$ considering their frequency support from \eqref{est 163} and Hypothesis \ref{Hypothesis 5.1}(a), and estimate for $a \in 5 \mathbb{N}$ sufficiently large and $\beta > 1 - \frac{\gamma_{2}}{2}$ sufficiently close to $1 -\frac{\gamma_{2}}{2}$, 
\begin{align}  
& \lVert R_{\text{Com2,3}} \rVert_{C_{t,x,q+1}} \label{est 395} \\
&\leq \lVert  \Upsilon - \Upsilon_{l} \rVert_{C_{t,q+1}}  \nonumber\\
& \hspace{5mm} \times \lVert  \mathcal{B} \tilde{P}_{\approx \lambda_{q+1}} [\Lambda^{2-\gamma_{2}} w_{q+1}^{\bot} \nabla^{\bot} \cdot y_{l} + \Lambda^{2-\gamma_{2}} y_{l}^{\bot} \nabla^{\bot} \cdot w_{q+1}]  + \mathcal{B} (\Lambda^{2-\gamma_{2}} y_{l}^{\bot} \nabla^{\bot} \cdot y_{l})  \rVert_{C_{t,x,q+1}} \nonumber\\ 
&\overset{\eqref{est 332} \eqref{est 163} }{\lesssim} l^{\frac{1}{2} - 2\delta} m_{L}^{2} [ \lambda_{q+1}^{-1} \lambda_{q+1}^{2-\gamma_{2}} \lVert w_{q+1} \rVert_{C_{t,x,q+1}} \lambda_{q}^{\gamma_{2} -1} \lVert \Lambda^{1-\gamma_{2}} \nabla^{\bot} y_{q} \rVert_{C_{t,x,q+1}} \nonumber \\
& \hspace{10mm} + \lambda_{q+1}^{-1} \lVert \Lambda^{2- \gamma_{2}} y_{q} \rVert_{C_{t,x,q+1}} \lambda_{q+1} \lVert w_{q+1} \rVert_{C_{t,x,q+1}} + \lVert \Lambda^{2-\gamma_{2}} y_{q} \rVert_{C_{t,x,q+1}}^{2} \lambda_{q}^{\gamma_{2} -1} ]  \nonumber \\ 
&\overset{\eqref{est 350a}  \eqref{est 317}}{\lesssim}  \lambda_{q+1}^{-\alpha ( \frac{1}{2} - 2 \delta)} m_{L}^{2} M_{0}(t) [ \delta_{q+1}^{\frac{1}{2}} \lambda_{q}^{2-\gamma_{2}} \delta_{q}^{\frac{1}{2}} m_{L} e^{\frac{1}{2} L^{\frac{1}{4}}} + \lambda_{q}^{3-\gamma_{2}} \delta_{q} m_{L}^{2}]  \nonumber \\
&\lesssim   M_{0}(t)  \lambda_{q+2}^{2-\gamma_{2}} \delta_{q+2}  [a^{b^{q} [ b^{2} (-2 + \gamma_{2} + 2 \beta) - b [\alpha (\frac{1}{2} - 2 \delta) + \beta] + 2 - \gamma_{2} - \beta]} + a^{b^{q} [b^{2} (-2 + \gamma_{2} + 2 \beta) - b \alpha( \frac{1}{2} - 2 \delta) + 3 - \gamma_{2} - 2 \beta]} ]]a^{L} \nonumber \\
& \ll  M_{0}(t) \lambda_{q+2}^{2-\gamma_{2}} \delta_{q+2} \nonumber 
\end{align}
where the last inequality relied on the fact that 
\begin{equation}\label{est 474}
b > \frac{2- \gamma_{2} - \beta + L}{\alpha(\frac{1}{2} - 2 \delta) + \beta} \hspace{1mm} \text{ and } \hspace{1mm} b > \frac{3 - \gamma_{2} - 2 \beta + L}{\alpha (\frac{1}{2} - 2 \delta)}
\end{equation} 
for all $\delta \in (0, \frac{1}{8})$ due to \eqref{est 446}.  Applying \eqref{est 393}, \eqref{est 394}, and \eqref{est 395} to \eqref{est 390} finally allows us to conclude
\begin{equation}\label{est 396}
\lVert R_{\text{Com2}} \rVert_{C_{t,x,q+1}} \ll  M_{0}(t) \lambda_{q+2}^{2-\gamma_{2}} \delta_{q+2}.
\end{equation} 

In conclusion, applying \eqref{est 397}, \eqref{est 376}, \eqref{est 381}, \eqref{est 307}, \eqref{est 398}, and \eqref{est 396} to \eqref{est 190} gives us 
\begin{equation*}
\lVert \mathring{R}_{q+1} \rVert_{C_{t,x,q+1}} =\lVert R_{T} + R_{N} + R_{L} + R_{O} + R_{\text{Com1}} + R_{\text{Com2}} \rVert_{C_{t,x,q+1}} \ll M_{0} \lambda_{q+2}^{2-\gamma_{2}} \delta_{q+2}. 
\end{equation*} 
Finally, the proof that $(y_{q+1}, \mathring{R}_{q+1})$ are $(\mathcal{F}_{t})_{t\geq 0}$-adapted and that $y_{q+1}(t,x)$ and $\mathring{R}_{q+1}(t,x)$ are deterministic over $[t_{q+1}, 0]$ can be verified very similarly to \cite{Y23a} and thus omitted. 

\section{Appendix}
\subsection{Further Preliminaries}
The following inverse divergence operator $\mathcal{B}$ was used frequently:
\begin{lemma}\label{Inverse-divergence lemma} 
\rm{(Inverse divergence from \cite[Definition 4.1]{BSV19})} Let $f$ be divergence-free and mean-zero. Then we define the $(i,j)$-th entry of $\mathcal{B}f$ for $i, j \in \{1,2\}$ by 
\begin{equation*}
(\mathcal{B}f)_{ij} \triangleq - \partial_{j} \Lambda^{-2} f_{i} - \partial_{i}\Lambda^{-2} f_{j}.
\end{equation*} 
For $f$ that is not divergence-free, we define $\mathcal{B} f \triangleq \mathcal{B} \mathbb{P} f$; for $f$ that is not mean-zero, we define $\mathcal{B} f \triangleq \mathcal{B} \left( f - \fint_{\mathbb{T}^{2}} f dx \right)$. It follows that $\divergence (\mathcal{B} f) = \mathbb{P} f$, and $\mathcal{B} f$ is a symmetric and trace-free matrix. One property that we will frequently rely on in estimates is $\lVert \mathcal{B} \rVert_{C_{x} \mapsto C_{x}} \lesssim 1$ (e.g. \cite[p. 127]{BV19a} and \cite[Lemma 7.3]{LQ20}). 
\end{lemma}
For $k \in \mathbb{S}^{1}$ we define
\begin{equation}\label{est 178} 
b_{k}(\xi) \triangleq i k^{\bot} e^{ik\cdot \xi} \hspace{1mm} \text{ and } \hspace{1mm} c_{k}(\xi) \triangleq e^{ik\cdot \xi}
\end{equation} 
which satisfies 
\begin{equation*}
b_{k} = \nabla_{\xi}^{\bot} c_{k}, \hspace{1mm} c_{k} = - \nabla_{\xi}^{\bot} \cdot b_{k}, \hspace{1mm} \Lambda_{\xi}^{\alpha} b_{k}(\xi) = b_{k}(\xi) \hspace{1mm} \forall \hspace{1mm} \alpha \geq 0, \hspace{1mm} b_{k}(\xi)^{\bot} = - \nabla_{\xi} c_{k}(\xi).
\end{equation*} 
For any finite family of vectors $\Gamma \subset \mathbb{S}^{1}$ and $\{ a_{k}  \} \subset \mathbb{C}$ such that $a_{-k} = \bar{a}_{k}$, we set $W(\xi) \triangleq \sum_{k \in \Gamma} a_{k} b_{k} (\xi)$ and $W_{k}(\xi) = a_{k} b_{k}(\xi)$ and we have 
\begin{equation}\label{est 0}
\divergence (W \otimes W) = \frac{1}{2} \nabla_{\xi} \lvert W \rvert^{2} + ( \nabla_{\xi}^{\bot} \cdot W) W^{\bot}, \hspace{1mm} \sum_{k \in \Gamma} W_{k} \otimes W_{-k} = \sum_{k\in \Gamma} \lvert a_{k} \rvert^{2} k^{\bot} \otimes k^{\bot}. 
\end{equation}

\begin{lemma}\label{Geometric Lemma} 
\rm{(Geometric lemma from \cite[Lemma 4.2]{BSV19})} Let $B(\Id, \epsilon)$ denote the ball of symmetric $2\times 2$ matrices, centered at $\Id$ of radius $\epsilon > 0$. Then there exists $\epsilon_{\gamma} > 0$ with which there exist disjoint finite subsets $\Gamma_{j} \subset \mathbb{S}^{1}$ for $j \in \{1,2\}$ and smooth positive functions $\gamma_{k} \in C^{\infty} (B(\Id, \epsilon_{\gamma})), j \in \{1,2\}, k \in \Gamma_{j}$ such that 
\begin{enumerate}
\item $5 \Gamma_{j} \subset \mathbb{Z}^{2}$ for each $j$, 
\item if $k \in \Gamma_{j}$, then $- k \in \Gamma_{j}$ and $\gamma_{k} =  \gamma_{-k}$, 
\item 
\begin{equation}\label{est 265}
R = \frac{1}{2} \sum_{k \in \Gamma_{j}} (\gamma_{k}(R))^{2} (k^{\bot} \otimes k^{\bot})  \hspace{3mm} \forall \hspace{1mm} R \in B(\Id, \epsilon_{\gamma}) 
\end{equation} 
\item $\lvert k + k' \rvert \geq \frac{1}{2}$ for all $k, k' \in \Gamma_{j}$ such that $k + k' \neq 0$. 
\end{enumerate} 
\end{lemma}

\begin{lemma}
\rm{(\cite[Lemma A.2]{BSV19}, \cite[Proposition D.1]{BDIS15})} 
Within this lemma, we denote $D_{t}\triangleq \partial_{t} + u \cdot \nabla$. Given a smooth vector field $u(t,x)$, as well as $f_{0}$ and $g$, consider a smooth solution $f$ to
\begin{align*}
\partial_{t} f + (u\cdot\nabla) f = g, \hspace{5mm}  f(t_{0}, x) = f_{0}(x), 
\end{align*}
where $u(t,x)$ is a given smooth vector field. Let $\Phi$ be the inverse of the flux $X$ of $u$ starting at time $t_{0}$ as the identity
\begin{align*}
\partial_{t} X = u(X, t),  \hspace{5mm} X(t_{0}, x) = x, 
\end{align*}
so that $f(t,x) = f_{0} (\Phi(t,x))$ in case $g \equiv 0$. It follows that for $t > t_{0}$,
\begin{subequations}\label{est 197}
\begin{align}
& \lVert f (t) \rVert_{C_{x}} \leq \lVert f_{0} \rVert_{C_{x}} + \int_{t_{0}}^{t} \lVert g(\tau) \rVert_{C_{x}} d\tau, \label{est 197a} \\
& \lVert Df (t) \rVert_{C_{x}} \leq \lVert D f_{0} \rVert_{C_{x}} e^{( t - t_{0}) \lVert Du \rVert_{C_{t, x}}} + \int_{t_{0}}^{t} e^{(t- \tau)  \lVert D u \rVert_{C_{t, x}}} \lVert Dg(\tau) \rVert_{C_{x}} d\tau; \label{est 197b} 
\end{align}
\end{subequations}  
more generally, for any $N \in \mathbb{N} \setminus \{1\}$, there exists a constant $C = C(N)$ such that 
\begin{align}
\lVert D^{N} f(t) \rVert_{C_{x}}  \leq& ( \lVert D^{N} f_{0} \rVert_{C_{x}} + C(t-t_{0})  \lVert D^{N} u \rVert_{C_{t,x}} \lVert Df_{0} \rVert_{C_{x}} ) e^{C(t-t_{0})  \lVert D u \rVert_{C_{t,x}}} \nonumber\\
&+ \int_{t_{0}}^{t} e^{C(t-\tau)  \lVert Du \rVert_{C_{t,x}}} (  \lVert D^{N} g(\tau) \rVert_{C_{x}} + C(t-\tau)  \lVert D^{N} u \rVert_{C_{t,x}} \lVert Dg(\tau) \rVert_{C_{x}} )d\tau.  \label{est 198} 
\end{align} 
Additionally, 
\begin{subequations}\label{est 199}
\begin{align}
&\lVert D\Phi(t) - \Id \rVert_{C_{x}} \leq e^{(t-t_{0}) \lVert Du \rVert_{C_{t,x}}} - 1 \leq (t- t_{0}) \lVert Du \rVert_{C_{t,x}} e^{(t- t_{0}) \lVert Du \rVert_{C_{t,x}}}, \label{est 199a}\\
& \lVert D^{N} \Phi(t) \rVert_{C_{x}} \leq C(t-t_{0})  \lVert D^{N} u \rVert_{C_{t,x}} e^{C(t-t_{0})  \lVert Du \rVert_{C_{t,x}}} \hspace{3mm} \forall \hspace{1mm} N \in \mathbb{N}\setminus \{1\}. \label{est 199b}
\end{align}
\end{subequations}
\end{lemma}

\begin{proposition}\label{Calderon commutator} 
\rm{(Calder$\acute{\mathrm{o}}$n commutator from \cite[Lemma A.5]{BSV19}; cf. \cite[Proposition A.3]{Y23a}, \cite[Lemma 2.2]{M08}, and \cite[Theorem 10.3]{L02})} Let $p \in (1,\infty)$. If $\phi \in W^{1,\infty}(\mathbb{T}^{2})$, then for any $f \in L^{p}(\mathbb{T}^{2})$ that is mean-zero, 
\begin{equation*}
\lVert [\Lambda, \phi] f \rVert_{L_{x}^{p}} \lesssim_{p} \lVert \phi \rVert_{W_{x}^{1,\infty}} \lVert f \rVert_{L_{x}^{p}}.
\end{equation*} 
Moreover, if $s \in [0,1], \phi \in W^{2,\infty}(\mathbb{T}^{2})$, then for any $f \in H^{s}(\mathbb{T}^{2})$ that is mean-zero, 
\begin{equation}\label{est 50}
\lVert [\Lambda, \phi] f \rVert_{H_{x}^{s}} \lesssim_{s} \lVert \phi \rVert_{W_{x}^{2,\infty}} \lVert f \rVert_{H_{x}^{s}}. 
\end{equation} 
\end{proposition} 

\subsection{Proof of Proposition \ref{Proposition 4.1}}
The proof of Proposition \ref{Proposition 4.1} (1) can be attained by applying the martingale representation theorem (e.g. \cite[Theorem 8.2]{DZ14}) to \cite[Proposition 4.1]{Y23a}. To prove Proposition \ref{Proposition 4.1} (2) concerning stability, we define 
\begin{equation}
\bar{\mathbb{S}} \triangleq C_{\text{loc}} ( [0,\infty); (H_{\sigma}^{4})^{\ast} \times U_{1}) \cap L_{\text{loc}}^{2} ([0,\infty); \dot{H}_{\sigma}^{\frac{1}{2}} \times U_{1}) 
\end{equation} 
and $F: \dot{H}_{\sigma}^{\frac{1}{2}} \mapsto (H_{\sigma}^{4})^{\ast}$ by 
\begin{equation}
F(\xi) \triangleq - \mathbb{P} [ ( \Lambda \xi \cdot \nabla) \xi - (\nabla \xi)^{T} \cdot \Lambda \xi ] - \Lambda^{\gamma_{1}} \xi
\end{equation} 
(see \cite[Equation (260)]{Y23a}). By hypothesis, for all $l \in \mathbb{N}, P_{l} \in \mathcal{W} ( s_{l}, \xi_{l}, \theta_{l}, \{ C_{t,q}\}_{q\in \mathbb{N}, t \geq s_{l}} )$ so that by Definition \ref{Definition 4.1} (M2), for any $t \geq s_{l}$, 
\begin{equation}
\xi(t)= \xi_{l} + \int_{s_{l}}^{t} F(\xi(r)) dr + \int_{s_{l}}^{t} G( \xi(r)) d\theta(r) \hspace{3mm} P_{l}\text{-a.s.}
\end{equation} 
where $\theta$ under $P_{l}$ is a cylindrical $(\bar{\mathcal{B}}_{t})_{t\geq s_{l}}$-Wiener process on $U$ starting from initial data $\theta_{l}$ at initial time $s_{l}$. I.e., the process $t \mapsto \theta (t + s_{l}) - \theta_{l}$ is a cylindrical Wiener process on $U$ with initial data 0 at initial time 0 as $\theta(s_{l}) = \theta_{l}$ $P_{l}$-a.s. by Definition \ref{Definition 4.1} (M1).  The law of this Wiener process is unique and tight in $C([0,\infty); U_{1}) \cap L_{\text{loc}}^{2} ([0,\infty); U_{1})$ so that for any $\epsilon > 0$, there exists a compact set 
\begin{align*}
K_{1} \subset C([0,\infty); U_{1}) \cap L_{\text{loc}}^{2} ([0,\infty); U_{1}) 
\end{align*}
such that  
\begin{equation}\label{est 7} 
\sup_{l\in\mathbb{N}} P_{l} (\theta(\cdot + s_{l})- \theta_{l} \in [C([0,\infty); U_{1}) \cap L_{\text{loc}}^{2} ([0,\infty); U_{1})]  \setminus K_{1}) \leq \epsilon. 
\end{equation} 
We now define 
\begin{equation}\label{est 8}
K_{2} \triangleq \cup_{l\in\mathbb{N}} \{ \theta \in C([0,\infty) ;  U_{1}): \hspace{0.5mm}  \theta(t+ s_{l}) - \theta_{l} \in K_{1} \hspace{1mm} \forall \hspace{0.5mm} t \in [0,\infty), \theta(t) = \theta_{l} \hspace{0.5mm} \forall \hspace{0.5mm} t \in [0, s_{l}] \}). 
\end{equation} 
We can compute 
\begin{equation}
\sup_{l\in\mathbb{N}} P_{l} (\overline{C([0,\infty); U_{1})} \setminus K_{2}) \overset{\eqref{est 7}}{\leq} \epsilon, \label{est 10}
\end{equation} 
and similar argument to previous works (e.g. \cite[Proof of Theorem 5.1]{HZZ19}) shows that $K_{2}$ is relatively compact in $C([0,\infty); U_{1}) \cap L_{\text{loc}}^{2} ([0,\infty); U_{1})$. On the other hand, from the proof of \cite[Proposition 5.2]{Y23a}, specifically \cite[Equation (261)]{Y23a}, we know that 
\begin{equation*}
K \triangleq \cup_{q\in\mathbb{N}} \cap_{k \in \mathbb{N}: k \geq k_{0}}\{ \xi \in \Omega_{q}: \sup_{t \in [0,k]} \lVert \xi(t) \rVert_{\dot{H}_{x}^{\frac{1}{2}}} + \sup_{r, t \in [0,k]: r \neq t} \frac{ \lVert \xi(t) - \xi(r) \rVert_{H_{x}^{-4}}}{\lvert t-r \rvert^{\kappa}} + \int_{s_{q}}^{k} \lVert \xi(r) \rVert_{\dot{H}_{x}^{\frac{1}{2} + \iota}}^{2} dr \leq R_{k} \}, 
\end{equation*} 
where $k_{0} \triangleq \sup_{l \in \mathbb{N}} s_{l}$,  is relatively compact in 
\begin{equation}
\mathbb{S} \triangleq C_{\text{loc}} ([0,\infty); (H_{\sigma}^{4})^{\ast}) \cap L_{\text{loc}}^{2} ([0,\infty); \dot{H}_{\sigma}^{\frac{1}{2}}).
\end{equation}  
(defined in \cite[Equation (258)]{Y23a}). Similarly to the proof of \cite[Proposition 5.2]{Y23a}, one can show 
\begin{equation}\label{est 15}
\sup_{l\in\mathbb{N}} P_{l} (\Omega_{0} \setminus \bar{K}) \leq \epsilon. 
\end{equation} 
It follows that $K \times K_{2}$ is relatively compact in $\bar{\mathbb{S}}$ and the tightness of $(P_{l})_{l\in\mathbb{N}}$ in $\bar{\mathbb{S}}$ follows from \eqref{est 10} and \eqref{est 15}. By Prokhorov's theorem (e.g., \cite[Theorem 2.3]{DZ14}) we may assume w.l.o.g. that $P_{l}$ converges weakly to some probability measure $P$. By Skorokhod's theorem (e.g., \cite[Theorem 2.4]{DZ14}), we deduce the existence of a probability space $(\tilde{\Omega}, \tilde{\mathcal{F}}, \tilde{P})$ and $\bar{\mathbb{S}}$-valued random variables $(\tilde{\xi}_{l}, \tilde{\theta}_{l})$ and $(\tilde{\xi}, \tilde{\theta})$ such that 
\begin{subequations}\label{est 16}
\begin{align}
& (\tilde{\xi}_{l}, \tilde{\theta}_{l}) \text{ has the law } P_{l} \hspace{1mm} \forall \hspace{1mm} l \in \mathbb{N}, \\
&(\tilde{\xi}_{l}, \tilde{\theta}_{l}) \to (\tilde{\xi}, \tilde{\theta}) \text{ in } \bar{\mathbb{S}} \hspace{1mm} \tilde{P}\text{-a.s.}\\
&  (\tilde{\xi}, \tilde{\theta}) \text{ has the law } P. 
\end{align}
\end{subequations} 
We denote by $(\tilde{\mathcal{F}}_{t})_{t\geq 0}$ the $\tilde{P}$-augmented canonical filtration of $(\tilde{\xi},\tilde{\theta})$. It follows that  $\tilde{\theta}$ is a cylindrical Wiener process on $U$ w.r.t. $(\tilde{\mathcal{F}}_{t})_{t\geq 0}$. 

The remaining task is to verify (M1)-(M3), of which (M1) and (M3) follow from the proof of \cite[Proposition 4.1]{Y23a} and hence we focus on the verification of (M2). As shown in \cite[Equations (273)-(274)]{Y23a}, we can show that for all  $\psi^{k} \in C^{\infty} (\mathbb{T}^{3}) \cap \dot{H}_{\sigma}^{\frac{1}{2}}$, $\tilde{P}$-a.s. 
\begin{equation*}
\langle \tilde{\xi}_{l}(t), \psi^{k} \rangle \to \langle \tilde{\xi}(t), \psi^{k} \rangle, \hspace{3mm} \int_{s_{l}}^{t} \langle F(\tilde{\xi}_{l} (r)), \psi^{k} \rangle dr  \to \int_{s}^{t} \langle F(\tilde{\xi}(r)), \psi^{k} \rangle dr.
\end{equation*} 
We define 
\begin{equation}\label{est 22} 
M_{t,s}^{\xi, k} \triangleq \langle \xi(t) - \xi(s) - \int_{s}^{t} F(\xi(r)) dr, \psi^{k} \rangle. 
\end{equation} 
Then similarly to \cite[Equations (279) and (281)]{Y23a}, we can show that for every $t \in [s, \infty)$ and $p \in (1,\infty)$, 
\begin{subequations}\label{est 17}
\begin{align} 
&\sup_{l\in\mathbb{N}} \mathbb{E}^{\tilde{P}}[ \lvert M_{t, s_{l}}^{\tilde{\xi}_{l}, k} \rvert^{2p}] \leq C \sup_{l\in\mathbb{N}} \mathbb{E}^{P_{l}} \left[ \left(\int_{s_{l}}^{t} \lVert G(\xi(r)) \rVert_{L_{2}(U; \dot{H}_{\sigma}^{\frac{1}{2}})}^{2} ds\right)^{p} \right] < \infty, \\
& \lim_{l\to\infty} \mathbb{E}^{\tilde{P}}[ \lvert M_{t, s_{l}}^{\tilde{\xi}_{l}, k} - M_{t,s}^{\tilde{\xi}, k} \rvert^{2}] = 0. 
\end{align}
\end{subequations}
Next, we let $t > r \geq s$ and $g$ be any bounded, $\bar{\mathcal{B}}_{r}$-measurable, continuous function on $\bar{\mathbb{S}}$. Using \eqref{est 17}, similarly to \cite[Equation (280)]{Y23a} we can show 
\begin{align*}
 \mathbb{E}^{P} [ (M_{t,s}^{\xi, k} - M_{r,s}^{\xi, k} ) g( (\xi, \theta)\lvert_{[0,r]})]  = 0 
\end{align*}
and consequently, $t \mapsto M_{t,s}^{k}$ is a $(\bar{\mathcal{B}}_{t})_{t \geq s}$-martingale under $P$. Similarly to \cite[Equation (282)]{Y23a} we can also show 
\begin{equation*}
\mathbb{E}^{P} [ (M_{t,s}^{\xi, k})^{2} - \int_{s}^{t} \lVert G(\xi (\lambda) )^{\ast} \psi^{k} \rVert_{U}^{2}  d\lambda  \lvert \bar{\mathcal{B}}_{r}] = (M_{r,s}^{\xi, k})^{2} - \int_{s}^{r} \lVert G(\xi (\lambda) )^{\ast} \psi^{k} \rVert_{U}^{2} d\lambda 
\end{equation*}
which implies $\langle \langle M_{t,s}^{\xi, k} \rangle \rangle = \int_{s}^{t} \lVert G(\xi (\lambda) )^{\ast} \psi^{k} \rVert_{U}^{2} d\lambda $ under $P$. Finally, in order to compute the cross variation process $\langle \langle M_{t,s}^{\xi, k}, \int_{s}^{t} \langle \psi^{k}, G(\xi) d \theta \rangle \rangle \rangle$, we let $\{\gamma_{j} \}_{j\in\mathbb{N}}$ be an orthonormal basis of $U$,  define $\theta_{j} \triangleq \langle \theta, \gamma_{j}  \rangle_{U}$, and compute  
\begin{equation}\label{est 18}
\langle \langle M_{t,s}^{\xi, k}, \theta_{j} (t) - \theta_{j} (s) \rangle \rangle = \int_{s}^{t} \langle G^{\ast} (\xi) \psi^{k}, \gamma_{j}  \rangle_{U} d\lambda. 
\end{equation} 
Consequently, $\langle \langle M_{t,s}^{\xi, k} - \int_{s}^{t} \langle \psi^{k}, G(\xi) d\theta \rangle \rangle \rangle  = 0$; hence, $M_{t,s}^{\xi, k} = \int_{s}^{t} \langle \psi^{k}, G(\xi) d\theta \rangle$. By definition of $M_{t,s}^{\xi, k}$, this implies \eqref{est 2} and verifies (M2).

\subsection{Proof of Theorem \ref{Theorem 2.2} assuming Theorem \ref{Theorem 2.1}} 
We fix $T > 0$ arbitrarily, $K > 1$, and $\kappa \in (0,1)$ such that $\kappa K^{2} \geq 1$ to apply Theorem \ref{Theorem 2.1}. The probability measure $P \otimes_{\tau_{L}}R$ from Proposition \ref{Proposition 4.5} satisfies $P \otimes_{\tau_{L}} R ( \{ \tau_{L} \geq T \}) > \kappa$ due to Lemma \ref{Lemma 4.3} and \eqref{est 148} which implies $\mathbb{E}^{P \otimes_{\tau_{L}}R}[ \lVert \xi(T) \rVert_{\dot{H}_{x}^{\frac{1}{2}}}^{2}] > \kappa K^{2} e^{T} \lVert v^{\text{in}} \rVert_{\dot{H}_{x}^{\frac{1}{2}}}^{2}$ where $v^{\text{in}}$ is the deterministic initial data in Theorem \ref{Theorem 2.1}. This, along with \eqref{est 37}, implies a lack of joint uniqueness in law and consequently non-uniqueness in law for \eqref{est 30} due to Cherny's theorem (\cite[Lemma C.1]{HZZ19}, \cite[Theorem 3.1]{C03}). 

\subsection{Proof of Theorem \ref{Theorem 2.1} assuming Proposition \ref{Proposition 5.2}}
Let us fix $T > 0, K > 1$, and $\kappa \in (0,1)$. We take $L$ that satisfies \eqref{est 443} and \eqref{est 445} larger if necessary to satisfy 
\begin{equation}\label{est 335}
\sqrt{2} e^{2LT} > (\sqrt{8} + \sqrt{2}) e^{2L^{\frac{1}{2}}}\hspace{1mm}  \text{ and } \hspace{1mm} L > [\ln (K e^{\frac{T}{2}})]^{2}. 
\end{equation}  
Because $\lim_{L\to\infty} T_{L} = \infty$ $\textbf{P}$-a.s., for the fixed $T> 0$ and $\kappa > 0$, increasing $L$ larger if necessary gives us \eqref{est 148}. We fix such $L > 1$. Then we turn to Proposition \ref{Proposition 5.1} to find $a \in 5 \mathbb{N}$ and $\beta$ in \eqref{est 315b} to obtain a collection $(y_{q}, \mathring{R}_{q})_{q \in \mathbb{N}_{0}}$, each of which solves \eqref{est 320} and satisfies the Hypothesis \ref{Hypothesis 5.1} over $[t_{q}, T_{L}]$, as well as \eqref{est 327} over $[t_{q+1}, T_{L}]$. It follows that for all $\beta' \in (\frac{1}{2}, \beta)$, because $b^{q+1} \geq b(q+1)$ for all $q \geq 0$ and $b \geq 2$,  
\begin{align*}
\sum_{q\geq 0} \lVert y_{q+1}(t) - y_{q}(t) \rVert_{C_{x}^{\beta'}} \overset{\eqref{est 327} \eqref{est 317}}{\lesssim} &  \sum_{q\geq 0} ( e^{\frac{1}{2} L^{\frac{1}{4}}}M_{0}(t)^{\frac{1}{2}} \delta_{q+1}^{\frac{1}{2}})^{1- \beta' } ( C_{0} m_{L} M_{0}(t)^{\frac{1}{2}} \lambda_{q+1}^{2-\gamma_{2}} \delta_{q+1}^{\frac{1}{2}} )^{\beta'} \\
& \hspace{15mm}  \lesssim m_{L} M_{0}(t)^{\frac{1}{2}}\sum_{q\geq 0} a^{b(q+1) (-\beta + \beta')} < \infty. 
\end{align*}  
This implies that $\{ y_{q}\}_{q \in \mathbb{N}_{0}}$ is Cauchy in $C_{t}C_{x}^{\beta'}$ for all $\beta' \in (0, \beta)$. Additionally, we can compute the following temporal regularity estimate that we present for general $\gamma_{2} \in [1,2)$ for the convenience in the proof of Proposition \ref{Proposition 5.2}: due to \eqref{est 319}, \eqref{est 332}, and \eqref{est 317} 
\begin{align}
\lVert \partial_{t} y_{q} \rVert_{C_{t,x,q}}  \leq& \lVert (\partial_{t} + \Upsilon_{l} \Lambda^{2-\gamma_{2}} y_{l} \cdot \nabla) y_{q}  \rVert_{C_{t,x,q}} + \lVert \Upsilon_{l} (\Lambda^{2- \gamma_{2}} y_{l} \cdot \nabla) y_{q} \rVert_{C_{t,x,q}} \label{est 330}   \\ 
\lesssim& L^{\frac{1}{2}} e^{2L^{\frac{1}{4}}} M_{0}(t) \lambda_{q}^{3-\gamma_{2}} \delta_{q} + e^{L^{\frac{1}{4}}} \lVert \Lambda^{2- \gamma_{2}} y_{q} \rVert_{C_{t,x,q}} \lVert \nabla y_{q} \rVert_{C_{t,x,q}} \lesssim m_{L}^{2} e^{L^{\frac{1}{4}}} M_{0}(t) \lambda_{q}^{3-\gamma_{2}} \delta_{q}.\nonumber 
\end{align} 
Thus, we can estimate for any $\eta \in [0, \frac{\beta}{2-\beta})$ where $\beta < 2$ due to \eqref{est 315b}, for all $q \in \mathbb{N}$, by \eqref{est 327} and \eqref{est 330} 
\begin{equation*}
\lVert y_{q}- y_{q-1} \rVert_{C_{t,q}^{\eta} C_{x}} \lesssim \lVert y_{q} - y_{q-1} \rVert_{C_{t,x,q}}^{1- \eta} \lVert y_{q} - y_{q-1} \rVert_{C_{t,q}^{1} C_{x}}^{\eta} 
\lesssim m_{L}^{2\eta} e^{(\frac{1+ \eta}{2}) L^{\frac{1}{4}}}  M_{0}(t)^{\frac{1+\eta}{2}}  \lambda_{q}^{(2-\beta) \eta - \beta} \to 0 
\end{equation*} 
as $q \to \infty$. Hence, there exists a limiting solution $\lim_{q\to\infty} y_{q} \triangleq y \in C([0,T_{L}]; C^{\beta'} (\mathbb{T}^{2})) \cap C^{\eta} ( [0, T_{L}]; C(\mathbb{T}^{2}))$, and there exists deterministic constant $C_{L} > 0$ with which $y$ satisfies the following inequality: 
\begin{equation}\label{est 137}
\lVert y \rVert_{C_{T_{L}}C_{x}^{\beta'}} + \lVert y \rVert_{C_{T_{L}}^{\eta} C_{x}} \leq C_{L} \hspace{3mm}\forall \hspace{1mm} \beta' \in \left(\frac{1}{2}, \beta \right), \eta \in \left[0, \frac{\beta}{2-\beta} \right), 
\end{equation}  
where $\frac{1}{3}< \frac{\beta}{2-\beta}$. Let us now formally denote its initial data as $y^{\text{in}} \triangleq y\rvert_{t=0}$. We can define $v^{\text{in}} \triangleq y^{\text{in}}$. We see that $y$ is $(\mathcal{F}_{t})_{t\geq 0}$-adapted because each $y_{q}$ is $(\mathcal{F}_{t})_{t\geq 0}$-adapted according to Propositions \ref{Proposition 5.1}-\ref{Proposition 5.2}. Moreover, we deduce from \eqref{est 318} and \eqref{est 315b} that $\lVert \mathring{R}_{q} \rVert_{C_{T_{L}} C_{x}} \leq  \epsilon_{\gamma} M_{0}(L) \lambda_{q+1}^{1- 2\beta}  \to 0$ as $q\to\infty$. It follows from Calder$\acute{\mathrm{o}}$n commutator estimate, specifically \eqref{est 50}, that $v \triangleq \Upsilon y$ is $(\mathcal{F}_{t})_{t\geq 0}$-adapted and solves \eqref{est 30} weakly. 

Next, we can directly compute 
\begin{equation}\label{est 334}
\lVert \Lambda^{\frac{1}{2}} y_{0} (t) \rVert_{L_{x}^{2}}^{2} 
\overset{\eqref{est 321} \eqref{est 171}}{=} m_{L}^{2}  M_{0}(t) 8 \pi^{4}.
\end{equation} 
Moreover, \eqref{est 137} which holds for $\eta \leq \frac{1}{3}$ and \eqref{est 332} give us for a.e. $\omega \in \Omega$ and all $\beta' \in (\frac{1}{2}, \beta)$ 
\begin{equation*}
\lVert v(\omega) \rVert_{C_{T_{L}} C_{x}^{\beta'}} +  \lVert v(\omega) \rVert_{C_{T_{L}}^{\eta} C_{x}}  \overset{\eqref{est 325} \eqref{est 332}}{\lesssim} e^{L^{\frac{1}{4}}}  \lVert y (\omega) \rVert_{C_{T_{L}} C_{x}^{\beta'}} + m_{L}^{2} \lVert y(\omega) \rVert_{C_{T_{L}}^{\eta} C_{x}} \overset{\eqref{est 137}}{<} \infty,
\end{equation*} 
which implies \eqref{est 9}. Next, we compute for all $t \in [0, T_{L}]$ using Hypothesis \ref{Hypothesis 5.1} \ref{Hypothesis 5.1 (a)}, the fact that $b^{q+1} \geq b(q+1)$ for all $q \geq 0$ and $b \geq 2$, that \eqref{est 315b} guarantees $\frac{1}{2} - \beta < 0$, and that $2C_{0} \geq \pi$ due to \eqref{est 481}, 
\begin{equation}\label{est 333}
\lVert y(t)  - y_{0}(t) \rVert_{\dot{H}_{x}^{\frac{1}{2}}} \overset{\eqref{est 327}}{\leq}  2^{\frac{3}{2}} \pi C_{0} m_{L} M_{0}(t)^{\frac{1}{2}} \frac{ a^{b(\frac{1}{2} - \beta)}}{1- a^{b( \frac{1}{2} - \beta)}} \overset{\eqref{est 153b}}{<} m_{L} M_{0}(t)^{\frac{1}{2}} \sqrt{2} \pi^{2}.
\end{equation}
It follows that 
\begin{align}
\lVert y(T)\rVert_{\dot{H}_{x}^{\frac{1}{2}}} \overset{\eqref{est 334}}{\geq} & m_{L} M_{0}(T)^{\frac{1}{2}} \sqrt{8} \pi^{2} - \lVert y(T) - y_{0}(T) \rVert_{\dot{H}_{x}^{\frac{1}{2}}} \overset{\eqref{est 333} \eqref{est 326}}{>}  m_{L} e^{2LT + L} \sqrt{2} \pi^{2}  \nonumber \\
\overset{\eqref{est 335} \eqref{est 333}}{>}& e^{2L^{\frac{1}{2}}} ( \lVert y_{0}(0)  \rVert_{\dot{H}_{x}^{\frac{1}{2}}} + \lVert y(0) - y_{0} (0) \rVert_{\dot{H}_{x}^{\frac{1}{2}}}) \geq e^{2L^{\frac{1}{2}}} \lVert y(0) \rVert_{\dot{H}_{x}^{\frac{1}{2}} }.  \label{est 336} 
\end{align}
Therefore, we can conclude \eqref{est 313} as follows: on $\{T_{L} \geq T \}$ 
\begin{equation*}
\lVert v(T) \rVert_{\dot{H}_{x}^{\frac{1}{2}}} \overset{\eqref{est 325} \eqref{est 336} \eqref{est 324} }{>} e^{- L^{\frac{1}{4}}} e^{2L^{\frac{1}{2}}} \lVert y(0) \rVert_{\dot{H}_{x}^{\frac{1}{2}}} \overset{\eqref{est 335}}{>}  K e^{\frac{T}{2}} \lVert v^{\text{in}} \rVert_{\dot{H}_{x}^{\frac{1}{2}}}. 
\end{equation*} 
At last, we note that $v^{\text{in}} (x) = \Upsilon(0) y(0,x) = y(0,x)$ is deterministic because $y_{q}(0,x)$ is deterministic for all $q \in \mathbb{N}_{0}$ due to Propositions \ref{Proposition 5.1}-\ref{Proposition 5.2}. 

\subsection{Proof of \eqref{est 307}} 
Here, we explain the proof of \eqref{est 307} for completeness, we only sketch it and refer to \cite{Y23a} for details due to their similarities. First, the computations in the additive case, specifically \cite[Equation (170)]{Y23a}, inform us that we may define $R_{O, \text{high}}$ in the case of linear multiplicative noise as 
\begin{align}
&R_{O, \text{high}} \triangleq  \nonumber \\
&  \mathcal{B} \tilde{P}_{\approx \lambda_{q+1}} ( \sum_{j,j',k,k': k + k' \neq 0} (\Lambda^{2-\gamma_{2}} \mathbb{P}_{q+1, k} [\chi_{j} a_{k,j} b_{k} (\lambda_{q+1}\Phi_{j} )]) \cdot \nabla \mathbb{P}_{q+1, k'} [ \chi_{j'} a_{k', j'} b_{k'} (\lambda_{q+1} \Phi_{j'})] \nonumber \\
& \hspace{20mm}  - (\nabla \mathbb{P}_{q+1,k} [\chi_{j} a_{k,j} b_{k} (\lambda_{q+1} \Phi_{j} ) ])^{T} \cdot \Lambda^{2-\gamma_{2}} \mathbb{P}_{q+1, k'} [ \chi_{j'} a_{k', j'} b_{k'} (\lambda_{q+1} \Phi_{j'})]). \label{est 386}  
\end{align}
Next, we can define $\mathcal{J}_{j,k}$ identically to \cite[Equation (171)]{Y23a} but with $ \tilde{w}_{q+1, j, k}$ therein replaced by $\chi_{j} a_{k,j} b_{k} (\lambda_{q+1} \Phi_{j})$ and $\vartheta_{j,k}$ similarly to\cite[Equation (179)]{Y23a}, specifically 
\begin{align*}
\mathcal{J}_{j,k} &\triangleq \frac{1}{2} [  \Lambda^{2-\gamma_{2}} \mathbb{P}_{q+1,k} [\chi_{j} a_{k,j} b_{k} (\lambda_{q+1} \Phi_{j})] \cdot \nabla \mathbb{P}_{q+1, -k} [\chi_{j} a_{-k,j} b_{-k} (\lambda_{q+1} \Phi_{j})]\\
&\hspace{30mm} + \Lambda^{2- \gamma_{2}} \mathbb{P}_{q+1, -k} [\chi_{j} a_{-k,j} b_{-k} (\lambda_{q+1} \Phi_{j})] \cdot \nabla \mathbb{P}_{q+1, k} [\chi_{j} a_{k,j} b_{k} (\lambda_{q+1} \Phi_{j})] \nonumber\\
& \hspace{10mm}  - (\nabla \mathbb{P}_{q+1, k} [\chi_{j} a_{k,j} b_{k} (\lambda_{q+1} \Phi_{j})])^{T} \cdot \Lambda^{2- \gamma_{2}} \mathbb{P}_{q+1, -k} [\chi_{j} a_{-k,j} b_{-k} (\lambda_{q+1} \Phi_{j})] \nonumber \\
& \hspace{10mm} - (\nabla \mathbb{P}_{q+1, -k} \chi_{j} a_{-k,j} b_{-k} (\lambda_{q+1} \Phi_{j}))^{T} \cdot \Lambda^{2- \gamma_{2}} \mathbb{P}_{q+1, k} [\chi_{j} a_{k,j} b_{k} (\lambda_{q+1} \Phi_{j})] ], \nonumber 
\end{align*}
and 
\begin{equation*}
\vartheta_{j,k} \triangleq \nabla^{\bot} \cdot \mathbb{P}_{q+1, k} \chi_{j} a_{k,j} b_{k} (\lambda_{q+1} \Phi_{j}). 
\end{equation*} 
It follows that the decomposition of $\mathcal{J}_{j,k}$ that is akin to \cite[Equation (172)]{Y23a} holds with $\mathcal{T}^{m}$, $\mathcal{L}_{j,k}$, and $ \mathcal{P}_{j,k}$ defined similarly to \cite[Equations (188), (190), and (192)]{Y23a}: 
\begin{equation*}
\mathcal{J}_{j,k} = \nabla \mathcal{P}_{j,k} + \divergence (\mathcal{L}_{j,k})  = \nabla \left( \mathcal{P}_{j,k} + \frac{ \mathcal{L}_{j,k}^{11} + \mathcal{L}_{j,k}^{22}}{2} \right) + \divergence (\mathring{\mathcal{L}}_{j,k}), 
\end{equation*} 
where 
\begin{align*}
&\mathcal{P}_{j,k} \triangleq \frac{1}{2} ( \Lambda^{-\gamma_{2}} \vartheta_{j,k}) \vartheta_{j, -k}, \\
& \mathcal{L}_{j,k}^{ml} \triangleq \frac{\gamma_{2}}{2} \mathcal{T}^{m} ( \Lambda^{ -\gamma_{2}} \vartheta_{j,k}, \Lambda^{1- \gamma_{2}} \mathcal{R}^{l} \vartheta_{j, -k} ) \\
& \hspace{10mm} \text{ in which  } \mathcal{T}^{m} (f,g) (x) \triangleq \int\int_{\mathbb{R}^{2} \times \mathbb{R}^{2}} K_{s^{m}} (x-y, x-z) f(y) g(z) dy dz \nonumber \\
& \hspace{20mm} ( \mathcal{T}^{m} (f,g))\hspace{1mm}\hat{}\hspace{1mm}  (\xi) \triangleq \int_{\mathbb{R}^{2}} s^{m} (\xi - \eta, \eta) \hat{f} (\xi - \eta) \hat{g} (\eta) d\eta.  \nonumber 
\end{align*}
This allows us to define 
\begin{equation}\label{est 259}
R_{O, \text{approx}} \triangleq \sum_{j} \chi_{j}^{2} (\mathring{R}_{l} -\mathring{R}_{q,j}) \hspace{2mm} \text{ and }  \hspace{2mm} R_{O, \text{low}} \triangleq \sum_{j} \chi_{j}^{2} \mathring{R}_{q,j} + \sum_{j,k}  \mathring{\mathcal{L}}_{j,k} 
\end{equation} 
and write 
\begin{equation}\label{est 291}
R_{O} = R_{O, \text{approx}} + R_{O, \text{low}} + R_{O, \text{high}}.
\end{equation}  
The identities 
\begin{align*} 
&\mathbb{P}_{q+1, k} [\chi_{j} a_{k,j} b_{k} (\lambda_{q+1} \Phi_{j})](x) \nonumber\\
& \hspace{10mm} =[\chi_{j} a_{k,j} b_{k} (\lambda_{q+1} \Phi_{j})] (x) +\chi_{j} [\mathbb{P}_{q+1, k}, a_{k,j} (x) \psi_{q+1, j, k} (x) ] b_{k} (\lambda_{q+1} x),   \\
&\Lambda^{2- \gamma_{2}} \mathbb{P}_{q+1, k} [\chi_{j} a_{k,j} b_{k} (\lambda_{q+1} \Phi_{j})](x) \nonumber\\
&\hspace{7mm} = \lambda_{q+1}^{2- \gamma_{2}} [\chi_{j} a_{k,j} b_{k} (\lambda_{q+1} \Phi_{j})](x) + \chi_{j} [\mathbb{P}_{q+1, k} \Lambda^{2- \gamma_{2}}, a_{k,j}(x) \psi_{q+1, j, k}(x) ] b_{k} (\lambda_{q+1} x), \\
& \Lambda^{- \gamma_{2}} \vartheta_{j,k}(x) =  \chi_{j} ik^{\bot} \cdot \Lambda^{1- \gamma_{2}} \mathcal{R}^{\bot} P_{\approx \lambda_{q+1} k} (a_{k,j} (x) \psi_{q+1, j, k} (x) c_{k} (\lambda_{q+1} x)),  \\
& \Lambda^{1- \gamma_{2}} \mathcal{R}_{l} \vartheta_{j, -k}(x) = - i \chi_{j} k^{\bot}  \cdot \nabla^{\bot}  \Lambda^{1- \gamma_{2}} \mathcal{R}_{l} P_{\approx -k \lambda_{q+1}} ( a_{-k, j} (x) \psi_{q+1, j, -k} (x) c_{-k} (\lambda_{q+1} x) ), \\
& ( \Lambda^{- \gamma_{2}} \vartheta_{j,k} )\hspace{1mm}\hat{}\hspace{1mm} (\xi) = - \chi_{j} \frac{ k \cdot \xi}{\lvert \xi \rvert^{\gamma_{2}}} \hat{K}_{\approx 1} \left( \frac{ \xi}{\lambda_{q+1}} - k \right) (a_{k,j} \psi_{q+1, j, k})\hspace{1mm}\hat{}\hspace{1mm}  (\xi - k \lambda_{q+1}), \\
&( \Lambda^{1- \gamma_{2}} \mathcal{R}_{l} \vartheta_{j, -k})\hspace{1mm}\hat{}\hspace{1mm}  (\xi)= \chi_{j} i \xi_{l} \frac{ k \cdot \xi}{\lvert \xi \rvert^{\gamma_{2}}} \hat{K}_{\approx 1} \left( \frac{\xi}{\lambda_{q+1}} + k \right) (a_{k,j} \psi_{q+1, j, -k})\hspace{1mm}\hat{}\hspace{1mm} (\xi + k \lambda_{q+1}),
\end{align*}
from \cite[Equations (175), (178), (194)--(195)]{Y23a} continue to hold. They allow us to define 
\begin{align*}
M_{k}^{ml} (\xi, \eta) \triangleq& - i s^{m} (\xi + k \lambda_{q+1}, \eta - k \lambda_{q+1}) \frac{ k\cdot ( \xi + k \lambda_{q+1} )}{\lvert \xi + k \lambda_{q+1} \rvert^{\gamma_{2}}} \nonumber\\
& \times \hat{K}_{\approx 1} \left( \frac{\xi}{\lambda_{q+1}} \right) (\eta_{l} - k_{l} \lambda_{q+1}) \frac{ k \cdot (\eta - k \lambda_{q+1})}{\lvert \eta - k \lambda_{q+1} \rvert^{\gamma_{2}}} \hat{K}_{\approx 1} \left( \frac{\eta}{\lambda_{q+1}} \right),  
\end{align*}
\begin{align*}
M_{k,r}^{ml} (\xi, \eta) \triangleq& \frac{ ((1-r) \eta - r \xi - k \lambda_{q+1})_{m}}{\lvert (1-r) \eta - r \xi - k \lambda_{q+1} \rvert^{2- \gamma_{2}}} \nonumber\\
& \times \frac{ k \cdot (\xi + k \lambda_{q+1} )}{\lvert \xi + k \lambda_{q+1} \rvert^{\gamma_{2}}} \hat{K}_{\approx 1} \left( \frac{\xi}{\lambda_{q+1}} \right) (\eta_{l} - k_{l} \lambda_{q+1}) \frac{ k \cdot (\eta - k \lambda_{q+1} )}{\lvert \eta - k \lambda_{q+1} \rvert^{\gamma_{2}}} \hat{K}_{\approx 1} \left( \frac{\eta}{\lambda_{q+1}} \right),   
\end{align*}
\begin{equation*}
(M_{k,r}^{\ast})^{ml} (\xi, \tilde{\xi}) \triangleq \frac{ (( 1-r) \tilde{\xi} - r \xi - k)_{m}}{\lvert (1-r) \tilde{\xi} - r \xi - k \rvert^{2 -\gamma_{2}}} (\tilde{\xi}_{l} - k_{l}) \frac{ k \cdot (\xi + k) k \cdot (\tilde{\xi} - k)}{\lvert \xi + k \rvert^{\gamma_{2}} \lvert \tilde{\xi} - k \rvert^{\gamma_{2}}} \hat{K}_{\approx 1} (\xi) \hat{K}_{\approx 1} (\tilde{\xi}),
\end{equation*}
similarly to \cite[Equations (196)--(197), (200)]{Y23a} and retain the identities 
\begin{align*}
\mathcal{L}_{j,k}^{ml} (x) =& \frac{\gamma_{2}}{2} \frac{\chi_{j}^{2}(t)}{(2\pi)^{2}} \int \int_{\mathbb{R}^{2} \times \mathbb{R}^{2}} M_{k}^{ml} (\xi, \eta) \nonumber\\
& \times (a_{k,j} \psi_{q+1, j, k})\hspace{1mm}\hat{}\hspace{1mm} (\xi) (a_{k,j} \psi_{q+1, j, -k})\hspace{1mm}\hat{}\hspace{1mm} (\eta) e^{ix\cdot (\xi + \eta)} d\xi d\eta, 
\end{align*}
\begin{align*}
& M_{k,r}^{ml} (\xi, \eta)=\lambda_{q+1}^{2- \gamma_{2}} (M_{k,r}^{\ast})^{ml} \left(\frac{ \xi}{\lambda_{q+1}}, \frac{\eta}{\lambda_{q+1}} \right), \\
& M_{k,r}^{ml} (0,0) = - k_{m} k_{l} \lambda_{q+1}^{2-\gamma_{2}}  \hspace{5mm} \forall \hspace{1mm} k \in \mathbb{S}^{1}, \\
&  \frac{1}{(2\pi)^{2}} \int\int_{\mathbb{R}^{2} \times \mathbb{R}^{2}} (a_{k,j} \psi_{q+1, j, k}) \hspace{1mm} \hat{}\hspace{1mm} (\xi) (a_{k,j} \psi_{q+1, j, -k}) \hspace{1mm} \hat{}\hspace{1mm} (\eta) e^{ix\cdot (\xi + \eta)} d\xi d \eta \overset{\eqref{est 200b}}{=} a_{k,j}^{2} (x),
\end{align*}
from \cite[Equations (199), (201)--(202)]{Y23a}. Additionally, we can define 
\begin{align*}
\tilde{\mathcal{L}}_{j,k}^{ml} (t,x) \triangleq& \frac{\gamma_{2}}{2}\frac{\chi_{j}^{2}(t)}{(2\pi)^{2}} \int\int_{\mathbb{R}^{2} \times \mathbb{R}^{2}} \int_{0}^{1} [ M_{k,r}^{ml} (\xi,\eta) - M_{k,r}^{ml} (0,0) ] dr \nonumber  \\
& \hspace{20mm} \times  (a_{k,j} \psi_{q+1, j, k})(\xi)\hspace{1mm}\hat{}\hspace{1mm} (a_{k,j} \psi_{q+1, j, -k})\hspace{1mm}\hat{}\hspace{1mm} (\eta) e^{ix\cdot (\xi + \eta)} d\xi d \eta, 
\end{align*}
\begin{align*}
( \tilde{\mathcal{L}}_{j,k}^{(1)})^{ml} (t,x) \triangleq& -i\frac{\gamma_{2}}{2}  \frac{ \chi_{j}^{2}(t)}{(2\pi)^{2}} \int_{0}^{1}\int_{0}^{1}\int \int_{\mathbb{R}^{2} \times \mathbb{R}^{2}} \lambda_{q+1}^{1- \gamma_{2}}  \nabla_{\xi} (M_{k,r}^{\ast})^{ml} \left( \frac{ \bar{r}\xi}{\lambda_{q+1}}, \frac{\bar{r} \eta}{\lambda_{q+1}} \right) \nonumber \\
& \cdot ( \nabla (a_{k,j} \psi_{q+1, j, k} ))\hspace{1mm}\hat{}\hspace{1mm} (\xi) (a_{k,j} \psi_{q+1, j, -k})\hspace{1mm}\hat{}\hspace{1mm}  (\eta) e^{ix\cdot (\xi + \eta)} d\xi d\eta d r d \bar{r}, \\ 
 ( \tilde{\mathcal{L}}_{j,k}^{(2)})^{ml} (t,x) \triangleq& -i \frac{\gamma_{2}}{2} \frac{\chi_{j}^{2} (t)}{(2\pi)^{2}} \int_{0}^{1}\int_{0}^{1} \int\int_{\mathbb{R}^{2} \times \mathbb{R}^{2}} \lambda_{q+1}^{1- \gamma_{2}}  \nabla_{\eta} (M_{k,r}^{\ast})^{ml} \left( \frac{\bar{r} \xi}{\lambda_{q+1}}, \frac{\bar{r} \eta}{\lambda_{q+1}} \right) \nonumber \\
& \cdot ( a_{k,j} \psi_{q+1, j, k} )\hspace{1mm}\hat{}\hspace{1mm} (\xi) (\nabla a_{k,j} \psi_{q+1, j, -k})\hspace{1mm}\hat{}\hspace{1mm}  (\eta) e^{ix \cdot (\xi + \eta)} d\xi d \eta d r d \bar{r} 
\end{align*}
similarly to \cite[Equation (203) and (205)]{Y23a} and retain the identities 
\begin{equation}\label{est 260} 
\mathcal{L}_{j,k}^{ml} (t, x)  =  \frac{\gamma_{2}}{2} \chi_{j}^{2}(t) \lambda_{q+1}^{2- \gamma_{2}} (k^{\bot} \otimes k^{\bot} -\Id)^{ml} a_{k,j}^{2}(t,x) + \tilde{\mathcal{L}}_{j,k}^{ml} (t,x), 
\end{equation} 
\begin{equation}\label{est 261}
\tilde{\mathcal{L}}_{j,k}^{ml} (t,x)  = ( \tilde{\mathcal{L}}_{j,k}^{(1)})^{ml} (t,x)  + ( \tilde{\mathcal{L}}_{j,k}^{(2)})^{ml} (t,x) 
\end{equation}
from \cite[Equations (204) and (206)]{Y23a}. Finally, we can define for $z, \tilde{z} \in \mathbb{R}^{2}$
\begin{align*}
& ( \mathcal{K}_{k, r, \bar{r}}^{(1)})^{ml} (z, \tilde{z}) \triangleq -i \left( \frac{ \lambda_{q+1}}{\bar{r}} \right)^{4}  ( \nabla_{\xi} (M_{k,r}^{\ast} )^{ml})\hspace{1mm}\check{}\hspace{1mm}  \left( \frac{\lambda_{q+1} z}{\bar{r}}, \frac{\lambda_{q+1} \tilde{z}}{\bar{r}}\right), \\
& ( \mathcal{K}_{k, r, \bar{r}}^{(2)})^{ml} (z, \tilde{z}) \triangleq -i\left( \frac{ \lambda_{q+1}}{\bar{r}} \right)^{4}  ( \nabla_{\eta} (M_{k,r}^{\ast} )^{ml})\hspace{1mm}\check{}\hspace{1mm}  \left( \frac{\lambda_{q+1} z}{\bar{r}}, \frac{\lambda_{q+1} \tilde{z}}{\bar{r}}\right), 
\end{align*}
and 
\begin{subequations}\label{est 258} 
\begin{align}
& \tilde{\mathcal{T}}_{k}^{(1), ml} (\nabla (a_{k,j} \psi_{q+1, j, k}), a_{k,j} \psi_{q+1, j, -k}) (t,x) \triangleq \frac{\gamma_{2}}{2} \lambda_{q+1}^{1-\gamma_{2}} \int_{0}^{1}\int_{0}^{1} \int\int_{\mathbb{R}^{2} \times \mathbb{R}^{2}} ( \mathcal{K}_{k,r,\bar{r}}^{(1)})^{ml} ( x- z, x - \tilde{z}) \nonumber\\
& \hspace{45mm} \cdot \nabla (a_{k,j} \psi_{q+1, j, k})(z) (a_{k,j} \psi_{q+1, j, -k}) (\tilde{z}) dz d \tilde{z} dr d \bar{r},\\
& \tilde{\mathcal{T}}_{k}^{(2), ml} ( a_{k,j} \psi_{q+1, j, k}, \nabla (a_{k,j} \psi_{q+1, j, -k})) (t,x) \triangleq \frac{\gamma_{2}}{2} \lambda_{q+1}^{1-\gamma_{2}} \int_{0}^{1}\int_{0}^{1} \int\int_{\mathbb{R}^{2} \times \mathbb{R}^{2}} ( \mathcal{K}_{k,r,\bar{r}}^{(2)})^{ml} ( x- z, x - \tilde{z}) \nonumber\\
& \hspace{45mm}\cdot (a_{k,j} \psi_{q+1, j, k})(z) \nabla (a_{k,j} \psi_{q+1, j, -k}) (\tilde{z}) dz d \tilde{z} dr d \bar{r},
\end{align}
\end{subequations} 
similarly to \cite[Equations (207)--(208)]{Y23a} and retain the identities  
\begin{subequations}\label{est 418}
\begin{align}
& (\tilde{\mathcal{L}}_{j,k}^{(1)})^{ml} (t,x)  = \chi_{j}^{2}(t) \tilde{\mathcal{T}}_{k}^{(1), ml} (\nabla (a_{k,j} \psi_{q+1, j, k}), a_{k,j} \psi_{q+1, j, -k})(t,x), \\
& ( \tilde{\mathcal{L}}_{j,k}^{(2)})^{ml} (t,x)  =  \chi_{j}^{2}(t) \tilde{\mathcal{T}}_{k}^{(2), ml} (a_{k,j} \psi_{q+1, j, k}, \nabla (a_{k,j} \psi_{q+1, j, -k}))(t,x), 
\end{align}
\end{subequations}
and an estimate 
\begin{equation}\label{est 267}
\left\lVert (z, \tilde{z})^{a} \nabla_{(z, \tilde{z})}^{b} ( \mathcal{K}_{k, r, \bar{r}}^{(i)} )^{ml} \right\rVert_{L_{z, \tilde{z}}^{1}(\mathbb{R}^{2} \times \mathbb{R}^{2} )} \leq C_{a,b} \left( \frac{\lambda_{q+1}}{\bar{r}}\right)^{\lvert b \rvert - \lvert a \rvert}, \hspace{3mm} i \in \{1,2\}, \hspace{1mm} 0 \leq \lvert a \rvert, \lvert b \rvert \leq 1,
\end{equation} 
which holds for uniformly in $r \in (0,1)$, from \cite[Equations (209)--(210)]{Y23a}. At last, \eqref{est 259}, \eqref{est 260}, and \eqref{est 261} lead to similar decomposition as \cite[Equations (211)--(212)]{Y23a}, namely 
\begin{equation}\label{est 262}
R_{O,\text{low}} = \mathring{O}_{1} + \mathring{O}_{2} 
\end{equation} 
where $\mathring{O}_{1}$ and $\mathring{O}_{2}$ respectively are the trace-free parts of 
\begin{subequations}\label{est 263}
\begin{align}
& O_{1}(t,x) \triangleq \sum_{j} \chi_{j}^{2}(t) \mathring{R}_{q,j}(t,x) + \frac{ \gamma_{2} \lambda_{q+1}^{2-\gamma_{2}}}{2}  \sum_{j,k}  (k^{\bot} \otimes k^{\bot} - \Id) \chi_{j}^{2}(t) a_{k,j}^{2}(t,x), \label{est 263a}\\
& O_{2} \triangleq O_{21} + O_{22} \hspace{2mm} \text{ where } \hspace{2mm} O_{21} \triangleq  \sum_{j,k}  \tilde{\mathcal{L}}_{j,k}^{(1)}  \hspace{1mm} \text{ and } O_{22} \triangleq  \sum_{j,k}  \tilde{\mathcal{L}}_{j,k}^{(2)}. \label{est 263b}
\end{align}
\end{subequations} 
Importantly, analogous computations to \cite[Equation (213)]{Y23a} go through similarly to show 
\begin{align}
O_{1} \overset{\eqref{est 263a}}{=} & \sum_{j} \chi_{j}^{2} \mathring{R}_{q,j} + \frac{ \gamma_{2} \lambda_{q+1}^{2-\gamma_{2}}}{2}  \sum_{j,k}  k^{\bot} \otimes k^{\bot} \chi_{j}^{2} a_{k,j}^{2} - \frac{\gamma_{2} \lambda_{q+1}^{2-\gamma_{2}}}{2}  \sum_{j,k}  \chi_{j}^{2} \Id a_{k,j}^{2} \nonumber \\
\overset{\eqref{est 351}}{=}& \sum_{j} \chi_{j}^{2} \lambda_{q+1}^{2- \gamma_{2}} [ \frac{ \mathring{R}_{q,j}}{\lambda_{q+1}^{2- \gamma_{2}}}  \nonumber \\
&+ \frac{\gamma_{2}}{2} \sum_{k}  k^{\bot} \otimes k^{\bot}  \frac{M_{0} (\tau_{q+1} j)}{\gamma_{2}} \delta_{q+1} \gamma_{k}^{2} \left( \Id - \frac{ \mathring{R}_{q,j}}{\lambda_{q+1}^{2-\gamma_{2}} \delta_{q+1}  M_{0} (\tau_{q+1} j)} \right) - \frac{\gamma_{2}}{2} \sum_{k} \Id a_{k,j}^{2} ]  \nonumber \\
&\overset{\eqref{est 265}}{=} \sum_{j} \chi_{j}^{2}(t) \lambda_{q+1}^{2- \gamma_{2}} \Id \left[ M_{0}(\tau_{q+1} j) \delta_{q+1} - \frac{\gamma_{2}}{2} \sum_{k} a_{k,j}^{2} \right] \label{est 459} 
\end{align}
so that $\mathring{O}_{1}$ vanishes as desired, leading us to simplify \eqref{est 262} as 
\begin{equation}\label{est 269} 
R_{O, \text{low}} = \mathring{O}_{21} + \mathring{O}_{22}
\end{equation} 
We now start our estimates. First, we observe that 
\begin{equation*}
D_{t,q} (\mathring{R}_{l} - \mathring{R}_{q,j}) = D_{t,q} \mathring{R}_{l} \hspace{1mm} \text{ and } \hspace{1mm} (\mathring{R}_{l} - \mathring{R}_{q,j} )(\tau_{q+1} j, x) = 0
\end{equation*}
due to \eqref{est 355} so that we may apply \eqref{est 197a} to estimate for all $t \in \supp \chi_{j}$
\begin{align}
 \lVert (\mathring{R}_{l} - \mathring{R}_{q,j})(t) & \rVert_{C_{x}} \overset{\eqref{est 349}}{\leq}  \tau_{q+1} \sup_{s \in [\tau_{q+1} j, t]} (\lVert \partial_{s} \mathring{R}_{l} (s) \rVert_{C_{x}} + \lvert \Upsilon_{l}(s) \rvert \lVert \Lambda^{2-\gamma_{2}} y_{l} (s) \rVert_{C_{x}} \lVert \nabla \mathring{R}_{l} (s) \rVert_{C_{x}}) \nonumber \\
\overset{\eqref{est 318} \eqref{est 317} \eqref{est 32}}{\lesssim}& \tau_{q+1} l^{-1} M_{0}(t) \lambda_{q+1}^{2-\gamma_{2}} \delta_{q+1} (1+ a^{L^{2} + b^{q}(3 - \gamma_{2} - \beta - b \alpha) } )  \lesssim \tau_{q+1} l^{-1} M_{0}(t) \lambda_{q+1}^{2-\gamma_{2}} \delta_{q+1} \label{est 383}
\end{align}
where the last inequality relied on 
\begin{align}
b>  \frac{L^{2} + 3 - \gamma_{2} - \beta}{\alpha} \label{est 470} 
\end{align}
due to \eqref{est 446}. We emphasize that this is one of the crucial parts where our appropriate modifications from the deterministic convex integration scheme of \cite{BSV19} plays a crucial role (recall Remark \ref{Remark 2.1} (2)). We can attain \eqref{est 383} only because we have $\lVert \partial_{s} \mathring{R}_{l}(s) \rVert_{C_{x}}$; if that were $\lVert \partial_{s} \mathring{R}_{q}(s) \rVert_{C_{x}}$ as in the deterministic case, it would be impossible because our Reynolds stress is not differentiable in time. It follows that 
\begin{equation}\label{est 387} 
\lVert R_{O, \text{approx}} \rVert_{C_{t,x,q+1}} 
\overset{\eqref{est 259}}{\leq} \sum_{j} \lVert 1_{\supp \chi_{j}} (\mathring{R}_{l} - \mathring{R}_{q,j}) \rVert_{C_{t,x,q+1}} 
\overset{\eqref{est 383}}{\lesssim}   \tau_{q+1} l^{-1} M_{0}(t) \lambda_{q+1}^{2-\gamma_{2}} \delta_{q+1}.
\end{equation} 
We can estimate 
\begin{align*}
& \lVert O_{21} \rVert_{C_{t,x,q+1}} \overset{\eqref{est 263b}}{=} \lVert   \sum_{j,k} \tilde{\mathcal{L}}_{j,k}^{(1)}   \rVert_{C_{t,x,q+1}} \\
\overset{\eqref{est 258}\eqref{est 418}}{\lesssim}&  \lVert \sum_{j,k}  \frac{\gamma_{2}}{2} \chi_{j}^{2} \lambda_{q+1}^{1-\gamma_{2}} \int_{0}^{1}\int_{0}^{1}\int\int_{\mathbb{R}^{2}\times \mathbb{R}^{2}}  (\mathcal{K}_{k,r,\bar{r}}^{(1)})^{ml} (x- z, x - \tilde{z})\cdot \nabla (a_{k,j} \psi_{q+1, j, k}) (z) \\
& \hspace{10mm}  \times (a_{k,j} \psi_{q+1,j,-k}) (\tilde{z}) dz d \tilde{z} dr d \bar{r} \rVert_{C_{t,x,q+1}} \overset{\eqref{est 267} \eqref{est 424} \eqref{est 360c}}{\lesssim}  \lambda_{q+1}^{1-\gamma_{2}} \lambda_{q} \delta_{q+1} M_{0}(t)
\end{align*} 
so that due to similarity of $O_{21}$ and $O_{22}$, we obtain 
\begin{equation}\label{est 385}  
\lVert R_{O, \text{low}} \rVert_{C_{t,x,q+1}} \overset{\eqref{est 269}}{\leq} \lVert O_{21} \rVert_{C_{t,x,q+1}} + \lVert O_{22} \rVert_{C_{t,x,q+1}} \lesssim \lambda_{q+1}^{1-\gamma_{2}} \lambda_{q} \delta_{q+1} M_{0}(t). 
\end{equation} 
Next, concerning $R_{O,\text{high}}$ in \eqref{est 386} we can define 
\begin{subequations}\label{est 272} 
\begin{align}
 O_{3} \triangleq& \mathcal{B} \tilde{P}_{\approx \lambda_{q+1}}\label{est 272a} \\
 & \times  \left( \sum_{j, j', k, k': k + k' \neq 0} (\Lambda^{2- \gamma_{2}} \mathbb{P}_{q+1, k} [\chi_{j} a_{k,j} b_{k} (\lambda_{q+1} \Phi_{j})]) \cdot \nabla \mathbb{P}_{q+1, k'} [\chi_{j'} a_{k',j'} b_{k'} (\lambda_{q+1} \Phi_{j'})] \right), \nonumber \\
 O_{4} \triangleq& \mathcal{B} \tilde{P}_{\approx \lambda_{q+1}}  \label{est 272b}  \\
 & \times  \left( \sum_{j, j', k, k': k + k' \neq 0}  (\nabla \mathbb{P}_{q+1, k} [\chi_{j} a_{k,j} b_{k} (\lambda_{q+1} \Phi_{j})])^{T} \cdot \Lambda^{2-\gamma_{2}} \mathbb{P}_{q+1, k'} [\chi_{j'} a_{k',j'} b_{k'} (\lambda_{q+1} \Phi_{j'})] \right),\nonumber 
\end{align}
\end{subequations} 
so that \eqref{est 386} gives us  
\begin{equation}\label{est 271} 
R_{O, \text{high}} = O_{3} - O_{4}. 
\end{equation} 
Similarly to \cite[Equations (221)--(224) and (227)]{Y23a}, we can split 
\begin{equation}\label{est 286} 
O_{3} = \sum_{k=1}^{3} O_{3k} 
\end{equation} 
where 
\begin{subequations}\label{est 275} 
\begin{align}
O_{31}(x) \triangleq& \mathcal{B} \tilde{P}_{\approx \lambda_{q+1}} \lambda_{q+1}^{2 - \gamma_{2}} \sum_{j, j',k,k': k + k'\neq 0}  \label{est 275a}  \\
&  \times \divergence ( [\chi_{j} a_{k,j} b_{k} (\lambda_{q+1} \Phi_{j})] \otimes [\chi_{j'} a_{k',j'} b_{k'} (\lambda_{q+1} \Phi_{j'})])(x),\nonumber \\
O_{32}(x) \triangleq& \mathcal{B} \tilde{P}_{\approx \lambda_{q+1}} \lambda_{q+1}^{2 - \gamma_{2}} \sum_{j, j',k, k': k + k'\neq 0}  \label{est 275b} \\
& \times  \divergence ( [\chi_{j} a_{k,j} b_{k} (\lambda_{q+1} \Phi_{j})](x) \otimes \chi_{j'} [\mathbb{P}_{q+1, k'}, a_{k', j'}(x) \psi_{q+1, j', k'}(x)] b_{k'} (\lambda_{q+1} x) ), \nonumber \\
O_{33}(x) \triangleq& \mathcal{B} \tilde{P}_{\approx \lambda_{q+1}}  \sum_{j, j', k, k': k + k'\neq 0}  \label{est 275c}\\
&  \times \divergence (\chi_{j} [\mathbb{P}_{q+1,k} \Lambda^{2-\gamma_{2}}, a_{k, j}(x) \psi_{q+1, j, k}(x)] b_{k} (\lambda_{q+1} x) \otimes \mathbb{P}_{q+1, k'} [\chi_{j'} a_{k',j'} b_{k'} (\lambda_{q+1} \Phi_{j'})](x)), \nonumber 
\end{align}
\end{subequations} 
and furthermore 
 \begin{equation}\label{est 280}
O_{31} = O_{311} + O_{312}
\end{equation} 
where 
\begin{subequations}
\begin{align}
&O_{311}(x)  \triangleq \frac{1}{2} \mathcal{B} \tilde{P}_{\approx \lambda_{q+1}} \lambda_{q+1}^{2- \gamma_{2}} \sum_{j, j', k, k': k + k' \neq 0} \chi_{j} \chi_{j'} \nabla (a_{k,j} (x) a_{k', j'} (x) \psi_{q+1, j', k'}(x) \psi_{q+1, j, k}(x))  \nonumber \\
& \hspace{35mm} \times (b_{k}(\lambda_{q+1} x) \cdot b_{k'} (\lambda_{q+1} x) - e^{i(k+k') \cdot \lambda_{q+1} x}),  \label{est 279} \\
& O_{312}(x) \triangleq  \mathcal{B} \tilde{P}_{\approx \lambda_{q+1}} \lambda_{q+1}^{2-\gamma_{2}} \sum_{j,j',k,k': k + k' \neq 0} \chi_{j} \chi_{j'}  \nonumber\\
& \hspace{10mm} \times b_{k'}(\lambda_{q+1} x) \otimes b_{k} (\lambda_{q+1} x) \nabla (a_{k,j} (x) a_{k', j'} (x) \psi_{q+1, j', k'} (x) \psi_{q+1, j, k} (x)).\label{est 278b}
\end{align}
\end{subequations}
Therefore, similarly to \cite[Equations (228)--(230)]{Y23a} we can bound using \cite[Equation (A.17)]{BSV19} for $O_{33}$, 
\begin{equation}\label{est 451}
\lVert O_{3} \rVert_{C_{t, x, q+1}}  \overset{\eqref{est 286} \eqref{est 280}}{\leq} \sum_{k=1}^{2} \lVert O_{31k} \rVert_{C_{t,x,q+1}} +  \sum_{k=2}^{3} \lVert O_{3k} \rVert_{C_{t,x,q+1}} \lesssim  \lambda_{q+1}^{1-\gamma_{2}} \lambda_{q} \delta_{q+1}M_{0}(t).  
\end{equation} 
Next, we can define 
\begin{subequations}\label{est 282} 
\begin{align}
&O_{41}(x) \triangleq \sum_{j, j', k, k': k + k' \neq 0}  \mathcal{B} \tilde{P}_{\approx \lambda_{q+1}} \label{est 282a} \\
&\hspace{5mm}  \times  (\chi_{j} \nabla ( [ \mathbb{P}_{q+1, k}, a_{k,j} (x) \psi_{q+1, j, k} (x) ] b_{k} (\lambda_{q+1} x) )^{T} \cdot \lambda_{q+1}^{2-\gamma_{2}} [\chi_{j'} a_{k',j'} b_{k'} (\lambda_{q+1} \Phi_{j'})](x)), \nonumber  \\
&O_{42}(x) \triangleq \sum_{j, j', k, k': k + k' \neq 0}  \mathcal{B} \tilde{P}_{\approx \lambda_{q+1}} \label{est 282b}   \\
& \hspace{5mm} \times ((\nabla \mathbb{P}_{q+1, k} [\chi_{j} a_{k,j} b_{k} (\lambda_{q+1} \Phi_{j})] (x))^{T} \cdot \chi_{j'} [\mathbb{P}_{q+1, k'} \Lambda^{2- \gamma_{2}}, a_{k', j'}(x) \psi_{q+1, j', k'} (x)] b_{k'} (\lambda_{q+1} x) ) \nonumber 
\end{align}
\end{subequations} 
so that $O_{4}$ from \eqref{est 272b} satisfies  
\begin{equation}\label{est 283} 
O_{4} = O_{41} + O_{42}
\end{equation} 
similarly to \cite[Equations (231)--(232)]{Y23a}. Applying \cite[Equation (A.17)]{BSV19} to $O_{41}$ gives us 
\begin{align}
& \lVert O_{41} \rVert_{C_{t,x,q+1}} \overset{\eqref{est 282a}}{\lesssim} \lambda_{q+1}^{-1} \sum_{j, j', k, k': k + k' \neq 0} [ \lambda_{q+1}^{-1} \lVert 1_{\supp \chi_{j}} \nabla^{2} (a_{k,j} \psi_{q+1, j, k}) \rVert_{C_{t,x,q+1}} \lVert b_{k} (\lambda_{q+1} x) \rVert_{C_{t,x,q+1}}\nonumber  \\
& \hspace{25mm} + \lambda_{q+1}^{-1} \lVert 1_{\supp \chi_{j}} \nabla (a_{k,j} \psi_{q+1, j, k} ) \rVert_{C_{t,x,q+1}} \lVert \nabla b_{k} (\lambda_{q+1} x) \rVert_{C_{t,x,q+1}} ] \lambda_{q+1}^{2-\gamma_{2}} \lVert a_{k', j'} \rVert_{C_{t,x,q+1}} \nonumber \\
&\hspace{5mm} \overset{\eqref{est 424} \eqref{est 360}}{\lesssim} \lambda_{q+1}^{1- \gamma_{2}} \lambda_{q} \delta_{q+1} M_{0}(t).\label{est 284} 
\end{align}
Similarly, we can estimate from \eqref{est 282b} by \cite[Equation (A.17)]{BSV19}, \eqref{est 424}, and \eqref{est 360}
\begin{equation}\label{est 285}
\lVert  O_{42} \rVert_{C_{t,x,q+1}} \lesssim \lambda_{q+1}^{1-\gamma_{2}} \lambda_{q} \delta_{q+1} M_{0}(t). 
\end{equation}
In sum, we conclude 
\begin{equation}\label{est 290}
\lVert O_{4} \rVert_{C_{t,x,q+1}} \overset{\eqref{est 283}}{\leq} \sum_{k=1}^{2} \lVert O_{4k} \rVert_{C_{t,x,q+1}} \overset{\eqref{est 284} \eqref{est 285}}{\lesssim} \lambda_{q+1}^{1-\gamma_{2}} \lambda_{q} \delta_{q+1} M_{0}(t).
\end{equation} 
Therefore, we showed that 
\begin{equation}\label{est 294}
\lVert R_{O,\text{high}} \rVert_{C_{t,x,q+1}} \overset{\eqref{est 271}}{\leq} \lVert O_{3} \rVert_{C_{t,x,q+1}} + \lVert O_{4} \rVert_{C_{t,x,q+1}} \overset{\eqref{est 451} \eqref{est 290}}{\lesssim} \lambda_{q+1}^{1-\gamma_{2}} \lambda_{q} \delta_{q+1} M_{0}(t).
\end{equation} 
At last, we are able to verify \eqref{est 307}: for $a \in 5 \mathbb{N}$ sufficiently large and $\beta > 1 - \frac{\gamma_{2}}{2}$ sufficiently close to $1- \frac{\gamma_{2}}{2}$, 
\begin{align*}
&e^{2L^{\frac{1}{4}}}\lVert R_{O} \rVert_{C_{t,x,q+1}} \overset{\eqref{est 291}}{\leq}  e^{2L^{\frac{1}{4}}} \left(\lVert R_{O, \text{approx}} \rVert_{C_{t,x,q+1}} + \lVert R_{O, \text{low}} \rVert_{C_{t,x,q+1}} + \lVert R_{O, \text{high}} \rVert_{C_{t,x,q+1}} \right) \nonumber \\
& \hspace{5mm} \overset{\eqref{est 387} \eqref{est 385}  \eqref{est 294}}{\lesssim} e^{2L^{\frac{1}{4}}} \left( \tau_{q+1} l^{-1} M_{0}(t) \lambda_{q+1}^{2- \gamma_{2}} \delta_{q+1} + \lambda_{q+1}^{1- \gamma_{2}} \lambda_{q} \delta_{q+1} M_{0}(t)\right) \nonumber \\
& \hspace{5mm} \overset{\eqref{est 210}\eqref{est 32}}{\lesssim}     M_{0}(t) \lambda_{q+2}^{2- \gamma_{2}} \delta_{q+2} [ a^{b^{q} [ b^{2} (-2 + \gamma_{2} + 2 \beta) + b( \frac{\alpha}{2} + \frac{1-\gamma_{2}}{2} - \frac{3\beta}{2} )]} + a^{b^{q} [ b^{2} (-2 + \gamma_{2} + 2 \beta) + b(1- \gamma_{2} - 2 \beta) + 1]} ] a^{L} \nonumber \\ 
& \hspace{15mm} \ll M_{0}(t) \lambda_{q+2}^{2- \gamma_{2}} \delta_{q+2}, 
\end{align*}
where the last inequality used the fact that 
\begin{equation}\label{est 471} 
b> \frac{2L}{-\alpha - 1 + \gamma_{2} + 3 \beta} \hspace{1mm} \text{ and } \hspace{1mm}  b > \frac{1+ L}{-1 + \gamma_{2} + 2 \beta}
\end{equation} 
due to \eqref{est 446}. 

\section*{Acknowledgments}
The author expresses deep gratitude to Prof. Adam Larios, Mr. Elliott Walker, and Prof. Yao Yao for valuable discussions.\\

\noindent \textbf{Conflict of interest} The author declares that he has no known competing financial interests or personal relationships that could have appeared to influence the work reported in this work.\\

\noindent \textbf{Data availability}  This work has no associated data. \\

\noindent \textbf{Funding} This work was supported by the Simons Foundation (962572, KY).


\begin{thebibliography}{100}  % 100 is a random guess of the total number of 
%references 
 %%
\addtolength{\leftmargin}{0.2in} % sets up alignment with the following line. 
\setlength{\itemindent}{-0.2in} 

\bibitem{B23} S. E. Berkemeier, \emph{On the 3D Navier-Stokes equations with a linear multiplicative noise and prescribed energy}, J. Evol. Equ., \textbf{23} (2023), pp. 1--50. 

\bibitem{BLW23} F. Bechtold, T. Lange, and J. Wichmann, \emph{On convex integration solutions to the surface quasi-geostrophic equation driven by generic additive noise}, arXiv:2311.00670 [math.PR], 2023. 

\bibitem{BFH20} D. Breit, E. Feireisl, and M. Hofmanov$\acute{\mathrm{a}}$, \emph{On solvability and ill-posedness of the compressible Euler system subject to stochastic forces}, Anal. PDE, \textbf{13} (2020), pp. 371--402. 

\bibitem{BJLZ23} E. Bru$\acute{\mathrm{e}}$, R. Jin, Y. Li, and D. Zhang, \emph{Non-uniqueness in law of Leray solutions to 3D forced stochastic Navier-Stokes equations}, arXiv:2309.09753 [math.PR], 2023. 

\bibitem{BDIS15} T. Buckmaster, C. De Lellis, P. Isett, and L. Sz$\acute{\mathrm{e}}$kelyhidi Jr., \emph{Anomalous dissipation for $1/5$-H$\ddot{o}$lder Euler flows}, Ann. of Math., \textbf{182} (2015), pp. 127--172. 

\bibitem{BNSW18} T. Buckmaster, A. Nahmod, G. Staffilani, and K. Widmayer, \emph{The surface quasi-geostrophic equation with random diffusion}, Int. Math. Res. Notices, \textbf{2020} (2020), pp. 9370--9385. 

\bibitem{BSV19} T. Buckmaster, S. Shkoller, and V. Vicol, \emph{Nonuniqueness of weak solutions to the SQG equations}, Comm. Pure Appl. Math., \textbf{LXXII} (2019), pp. 1809--1874. 

\bibitem{BV19a} T. Buckmaster and V. Vicol, \emph{Nonuniqueness of weak solutions to the Navier-Stokes equation}, Ann. of Math., \textbf{189} (2019), pp. 101--144. 

\bibitem{BV19b} T. Buckmaster and V. Vicol, \emph{Convex integration and phenomenologies in turbulence}, EMS Surveys in Mathematical Sciences, \textbf{6} (2019), pp. 173--263. 

\bibitem{BMS21} J. Burczak, S. Modena, and L. Sz$\acute{\mathrm{e}}$kelyhidi Jr., \emph{Non uniqueness of power-law flows}, Comm. Math. Phys., \textbf{388} (2021), pp. 199--243, \url{https://doi.org/10.1007/s00220-021-04231-7}. 

\bibitem{CV10} L. Caffarelli and A. Vasseur, \emph{Drift diffusion  equations with fractional diffusion and the quasi-geostrophic equation}, Ann. of Math. \textbf{171} (2010), pp. 1903--1930.

\bibitem{C48} J. Charney, \emph{On the scale of atmospheric motions}, Geofys. Pub. Oslo, \textbf{17} (1948), pp. 251--265

\bibitem{CDZ22} W. Chen, Z. Dong, and X. Zhu, \emph{Sharp non-uniqueness of solutions to stochastic Navier-Stokes equations}, arXiv:2208.08321 [math.PR], 2022. 

\bibitem{CKL20} X. Cheng, H. Kwon, and D. Li, \emph{Non-uniqueness of steady-state weak solutions to the surface quasi-geostrophic equations}, Comm. Math. Phys., \textbf{388} (2021), pp. 1281--1995. 

\bibitem{C03} A. S. Cherny, \emph{On the uniqueness in law and the pathwise uniqueness for stochastic differential equations}, Theory Probab. Appl., \textbf{46} (2003), pp. 406--419. 

\bibitem{CFF19} E. Chiodaroli, E. Feireisl, and F. Flandoli, \emph{Ill posedness for the full Euler system driven by multiplicative white noise}, Indiana Univ. Math. J., \textbf{70} (2021), pp. 1267--1282. DOI: 10.1512/iumj.2021.70.8591

\bibitem{CDS12a} A. Choffrut, C. De Lellis, and L. Sz$\acute{\mathrm{e}}$kelyhidi Jr., \emph{Dissipative continuous Euler flows in two and three dimensions}, arXiv:1205.1226 [math.AP], 2012.  

\bibitem{CDD18} M. Colombo, C. De Lellis, and L. De Rosa, \emph{Ill-posedness of Leray solutions for the hypodissipative Navier-Stokes equations}, Comm. Math. Phys., \textbf{362} (2018), pp. 659--688. 

\bibitem{CET94} P. Constantin, W. E, and E. S. Titi, \emph{Onsager's conjecture on the energy conservation for solutions of Euler's equation}, Comm. Math. Phys., \textbf{165} (1994), pp. 207--209.

\bibitem{CIW08} P. Constantin, G. Iyer, and J. Wu, \emph{Global regularity for a modified critical dissipative quasi-geostrophic equation}, Indiana U. Math. J., \textbf{57} (2008), pp. 2681--2692. 

\bibitem{CMT94} P. Constantin, A. J. Majda, and E. Tabak, \emph{Formation of strong fronts in the 2D quasi-geostrophic thermal active scalar}, Nonlinearity, \textbf{7} (1994), pp. 1495--1533. 

\bibitem{CV12} P. Constantin and V. Vicol, \emph{Nonlinear maximum principles for dissipative linear nonlocal operators and applications}, Geom. Funct. Anal., \textbf{22} (2012), pp. 1289--1321. 

\bibitem{CW99} P. Constantin and J. Wu, \emph{Behavior of solutions of 2D quasi-geostrophic equations}, SIAM J. Math. Anal., \textbf{30} (1999), pp. 937--948.

\bibitem{CFG11} D. C$\acute{\mathrm{o}}$rdoba, D. Faraco, and F. Gancedo, \emph{Lack of uniqueness for weak solutions of the incompressible porous media equation}, Arch. Ration. Mech. Anal., \textbf{200} (2011), pp. 725--746. 

\bibitem{CDS12b} S. Conti, C. De Lellis, and L. Sz$\acute{\mathrm{e}}$kelyhidi Jr., $h$-\emph{principle and rigidity for }$C^{1,\alpha}$\emph{ isometric embeddings}, In: Holden H., Karlsen K. (eds) Nonlinear Partial Differential Equations. Abel Symposia, vol 7. Springer, Berlin, Heidelberg, 2012.  

\bibitem{DZ14} G. Da Prato and J. Zabczyk, \emph{Stochastic Equations in Infinite Dimensions}, Cambridge University Press, Cambridge, 2014. 

\bibitem{D72} D. A. Dawson, \emph{Stochastic evolution equations}, Math. Biosci., \textbf{15} (1972),  pp. 287--316. 

\bibitem{DS09} C. De Lellis and L. Sz$\acute{\mathrm{e}}$kelyhidi Jr., \emph{The Euler equations as a differential inclusion}, Ann. of Math., \textbf{170} (2009), pp. 1417--1436. 

\bibitem{DS10} C. De Lellis and L. Sz$\acute{\mathrm{e}}$kelyhidi Jr., \emph{On admissibility criteria for weak solutions of the Euler equations}, Arch. Ration. Mech. Anal., \textbf{195} (2010), pp. 225--260. 

\bibitem{DS13} C. De Lellis and L. Sz$\acute{\mathrm{e}}$kelyhidi Jr., \emph{Dissipative continuous Euler flows}, Invent. Math., \textbf{193} (2013), pp. 377--407. 

\bibitem{D19} L. De Rosa, \emph{Infinitely many Leray-Hopf solutions for the fractional Navier-Stokes equations}, Comm. Partial Differential Equations, \textbf{44} (2019), pp. 335--365.   

\bibitem{E94} G. L. Eyink, \emph{Energy dissipation without viscosity in ideal hydrodynamics, I. Fourier analysis and local energy transfer}, Phys. D, \textbf{78} (1994), pp. 222--240. 

\bibitem{F08} F. Flandoli, \emph{An Introduction to 3D Stochastic Fluid Dynamics}: In: G. Da Prato, M. R$\ddot{\mathrm{u}}$ckner (Eds.) SPDE in Hydrodynamic: Recent Progress and Prospects. Lecture Notes in Mathematics, \textbf{1942}, Springer, Berlin, Heidelberg, (2008), pp. 51--150. 

\bibitem{F75} W. H. Fleming, \emph{Distributed parameter stochastic systems in population biology}, Control Theory, Numerical Methods and Computer Systems Modeling, \textbf{107} (1975),  pp. 179--191. 

\bibitem{HPGS95} I. Held, R. T. Pierrehumbert, S. T. Garner, and K. L. Swanson, \emph{Surface quasi-geostrophic dynamics}, J. Fluid Mech., \textbf{282} (1995), pp. 1--20. 

\bibitem{GR23} V. Giri, and R.-Oc. Radu, \emph{The 2D Onsager conjecture: a Newton-Nash iteration}, arXiv:2305.18105 [math.AP], 2023. 

\bibitem{GRZ09} B. Goldys, M. R$\ddot{\mathrm{o}}$ckner, and X. Zhang, \emph{Martingale solutions and Markov selections for stochastic partial differential equations}, Stochastic Process. Appl., \textbf{119} (2009), pp. 1725--1764. 

\bibitem{HLP23} M. Hofmanov$\acute{\mathrm{a}}$, T. Lange, and U. Pappalettera, \emph{Global existence and non-uniqueness of 3D Euler equations perturbed by transport noise}, Probab. Theory Related Fields (2023), \url{https://doi.org/10.1007/s00440-023-01233-5}. 

\bibitem{HLZZ23} M. Hofmanov$\acute{\mathrm{a}}$, X. Luo, R. Zhu, and X. Zhu, \emph{Surface quasi-geostrophic equation perturbed by derivatives of space-time white noise}, arXiv:2308.14358 [math.PR], 2023. 

\bibitem{HPZZ23a} M. Hofmanov$\acute{\mathrm{a}}$, U. Pappalettera, R. Zhu, and X. Zhu, \emph{Kolmogorov 4/5 law for the forced 3D Navier-Stokes equations}, arXiv:2304.14470v2 [math.AP], 2023. 

\bibitem{HPZZ23b} M. Hofmanov$\acute{\mathrm{a}}$, U. Pappalettera, R. Zhu, and X. Zhu, \emph{Anomalous and total dissipation due to advection by solutions of randomly forced Navier-Stokes equations}, arXiv:2305.08090v2 [math.AP], 2023. 

\bibitem{HZZ19} M. Hofmanov$\acute{\mathrm{a}}$, R. Zhu, and X. Zhu, \emph{Non-uniqueness in law of stochastic 3D Navier-Stokes equations}, J. Eur. Math. Soc. (2023), DOI 10.4171/JEMS/1360.

\bibitem{HZZ20} M. Hofmanov$\acute{\mathrm{a}}$, R. Zhu, and X. Zhu, \emph{On ill- and well-posedness of dissipative martingale solutions to stochastic 3D Euler equations}, Comm. Pure Appl. Math., \textbf{75} (2022), pp. 2446--2510. 

\bibitem{HZZ21a} M. Hofmanov$\acute{\mathrm{a}}$, R. Zhu, and X. Zhu, \emph{Global-in-time probabilistically strong and Markov solutions to stochastic 3D Navier-Stokes equations: existence and non-uniqueness}, Ann. Probab., \textbf{51} (2023), pp. 524--579. 

\bibitem{HZZ21b} M. Hofmanov$\acute{\mathrm{a}}$, R. Zhu, and X. Zhu, \emph{Global existence and non-uniqueness for 3D Navier--Stokes equations with space-time white noise}, Arch. Ration. Mech. Anal., \textbf{46} (2023), \url{https://doi.org/10.1007/s00205-023-01872-x}. 

\bibitem{HZZ22a} M. Hofmanov$\acute{\mathrm{a}}$, R. Zhu, and X. Zhu, \emph{A class of supercritical/critical singular stochastic PDEs: existence, non-uniqueness, non-Gaussianity, non-unique ergodicity}, J. Funct. Anal., \textbf{285} (2023), pp. 1--43. 

\bibitem{HZZ22b} M. Hofmanov$\acute{\mathrm{a}}$, R. Zhu, and X. Zhu, \emph{Non-unique ergodicity for deterministic and stochastic 3D Navier-Stokes and Euler equations}, arXiv:2208.08290 [math.PR], 2022. 

\bibitem{HZZ23a} M. Hofmanov$\acute{\mathrm{a}}$, R. Zhu, and X. Zhu, \emph{Non-uniqueness of Leray-Hopf solutions for stochastic forced Navier-Stokes equations}, arXiv:2309.03668 [math.PR], 2023. 

\bibitem{I18} P. Isett, \emph{A proof of Onsager's conjecture}, Ann. of Math., \textbf{188} (2018), pp. 871--963. 

\bibitem{IM21} P. Isett and A. Ma, \emph{A direct approach to nonuniqueness and failure of compactness for the SQG equation}, Nonlinearity, \textbf{34}, 3122 (2021). 

\bibitem{IV15} P. Isett and V. Vicol, \emph{H$\ddot{o}$lder continuous solutions of active scalar equations}, Ann. PDE, \textbf{1} (2015), pp. 1--77. 

\bibitem{K11} J. U. Kim, \emph{On the stochastic quasi-linear symmetric hyperbolic system}, J. Differential Equations, \textbf{250} (2011), pp. 1650--1684. 

\bibitem{K10} A. Kiselev, \emph{Regularity and blow-up for active scalars}, Math. Model. Nat. Phenom., \textbf{5} (2010), pp. 225--255.

\bibitem{KN10} A. Kiselev and F. Nazarov, \emph{Variation on a theme of Caffarelli and Vasseur}, J. Math. Sci., \textbf{166} (2010), pp. 31--39. 

\bibitem{KNV07} A. Kiselev, F. Nazarov, and A. Volberg, \emph{Global well-posedness for the critical 2D dissipative quasi-geostrophic equation}, Invent. Math., \textbf{167} (2007), pp. 445--453.

\bibitem{KY22} U. Koley and K. Yamazaki, \emph{Non-uniqueness in law of transport-diffusion equation forced by random noise}, arXiv:2203.13456 [math.AP], 2022.  

\bibitem{LL56} L. D. Landau and E. M. Lifshitz, \emph{Hydrodynamic fluctuations}, J. Exptl. Theoret. Phys. (U.S.S.R.) \textbf{32} (1957), pp. 618--619.

\bibitem{L02} P. G. Lemari$\acute{\mathrm{e}}$-Rieusset, \emph{Recent Developments in the Navier-Stokes Problem}, Chapman $\&$ Hall/CRC, 2002. 

\bibitem{LZZ22} Y. Li, Z. Zeng, and D. Zhang, \emph{Non-uniqueness of weak solutions to 3D magnetohydrodynamic equations}, J. Math. Pures Appl., \textbf{165} (2022), pp. 232--285. 

\bibitem{LZ22} H. L$\ddot{\mathrm{u}}$ and X. Zhu, \emph{Global-in-time probabilistically strong solutions to stochastic power-law equations: existence and non-uniqueness}, Stochastic Process. Appl., \textbf{164} (2023), pp. 62--98. 

\bibitem{LZ23} H. L$\ddot{\mathrm{u}}$ and X. Zhu, \emph{Sharp non-uniqueness of solutions to 2D Navier-Stokes equations with space-time white noise}, arXiv:2304.06526v2 [math.PR], 2023. 

\bibitem{LQ20} T. Luo and P. Qu, \emph{Non-uniqueness of weak solutions to 2D hypoviscous Navier-Stokes equations}, J. Differential Equations, \textbf{269} (2020), pp. 2896--2919. 

\bibitem{LTZ20} T. Luo, T. Tao, and L. Zhang, \emph{Finite energy weak solutions of 2D Boussinesq equations with diffusive temperature}, Discrete Contin. Dyn. Syst., \textbf{40} (2020), pp. 3737--3765. 

\bibitem{MW06} A. J. Majda and X. Wang, \emph{Non-linear Dynamics and Statistical Theories for Basic Geophysical Flows}, Cambridge University Press, 2006. 

\bibitem{M08} F. Marchand, \emph{Existence and regularity of weak solutions to the quasi-geostrophic equations in the spaces $L^{p}$ or $\dot{H}^{-\frac{1}{2}}$}, Comm. Math. Phys., \textbf{277} (2008), pp. 45--67. 

\bibitem{MS20} S. Modena and G. Sattig, \emph{Convex integration solutions to the transport equation with full dimensional concentration}, Ann. Inst. H. Poincar$\acute{\mathrm{e}}$ Anal. Non Lin$\acute{\mathrm{e}}$aire, \textbf{37} (2020), pp. 1075--1108.

\bibitem{MS23} S. Modena and A. Schenke, \emph{Local nonuniqueness for stochastic transport equations with deterministic drift}, arXiv:2306.08758v2 [math.PR], 2023.  

\bibitem{P23} U. Pappalettera, \emph{Global existence and non-uniqueness for the Cauchy problem associated to 3D Navier-Stokes equations perturbed by transport noise}, Stoch. PDE: Anal. Comp. (2023), \url{https://doi.org/10.1007/s40072-023-00318-5}. 

\bibitem{RS21} M. Rehmeier and A. Schenke, \emph{Nonuniqueness in law for stochastic hypodissipative Navier-Stokes equations}, Nonlinear Anal., \textbf{227} (2023), 113179. 

\bibitem{R95} S. Resnick, \emph{Dynamical problems in nonlinear advective partial differential equations}, Ph.D. Thesis, University of Chicago, 1995.

\bibitem{RZZ14} M. R$\ddot{\mathrm{o}}$ckner, R. Zhu, and X. Zhu, \emph{Local existence and non-explosion of solutions for stochastic fractional partial differential equations driven by multiplicative noise}, Stochastic Process. Appl., \textbf{124} (2014), pp. 1974--2002. 

\bibitem{RZZ15} M. R$\ddot{\mathrm{o}}$ckner, R. Zhu, and X. Zhu, \emph{Sub and supercritical stochastic quasi-geostrophic equation}, Ann. Probab., \textbf{43} (2015), pp. 1202--1273. 

\bibitem{VK18} H.-H. Vo and S. Kim, \emph{Convex integration for scalar conservation laws in one space dimension}, SIAM J. Math. Anal., \textbf{50} (2018), pp. 3122--3146. 

\bibitem{Y20a} K. Yamazaki, \emph{Remarks on the non-uniqueness in law of the Navier-Stokes equations up to the J.-L. Lions' exponent}, Stochastic Process. Appl., \textbf{147} (2022), pp. 226-269. 

\bibitem{Y20c} K. Yamazaki, \emph{Non-uniqueness in law for two-dimensional Navier-Stokes equations with diffusion weaker than a full Laplacian}, SIAM J. Math. Anal., \textbf{54} (2022), pp. 3997--4042. 

\bibitem{Y21a} K. Yamazaki, \emph{Non-uniqueness in law for Boussinesq system forced by random noise}, Calc. Var. Partial Differential Equations, \textbf{61} (2022), pp. 1--65, \url{https://doi.org/10.1007/s00526-022-02285-6}. 

\bibitem{Y21c} K. Yamazaki, \emph{Non-uniqueness in law of three-dimensional Navier-Stokes equations diffused via a fractional Laplacian with power less than one half}, Stoch. PDE: Anal. Comp. (2023), \url{https://doi.org/10.1007/s40072-023-00293-x}. 

\bibitem{Y21d} K. Yamazaki, \emph{Non-uniqueness in law of three-dimensional magnetohydrodynamics system forced by random noise},   Potential Anal. (2024), \url{https://doi.org/10.1007/s11118-024-10128-6}. 

\bibitem{Y22a} K. Yamazaki, \emph{Ergodicity of Galerkin approximations of surface quasi-geostrophic equations and Hall-magnetohydrodynamics system forced by degenerate noise}, Nonlinear Differential Equations and Applications NoDEA, \textbf{29} (2022), \url{https://doi.org/10.1007/s00030-022-00753-8}. 

\bibitem{Y23a} K. Yamazaki, \emph{Non-uniqueness in law of the two-dimensional surface quasi-geostrophic equations forced by random noise}, Ann. Inst. Henri Poincare Probab. Stat., to appear. 

\bibitem{YM19} K. Yamazaki and M. T. Mohan, \emph{Well-posedness of Hall-magnetohydrodynamics system forced by Levy noise}, Stoch. PDE: Anal. Comp., \textbf{7} (2019), pp. 331--378. 

\bibitem{Z12a} R. Zhu, \emph{SDE and BSDE on Hilbert spaces: applications to quasi-linear evolution equations and the asymptotic properties of the stochastic quasi-geostrophic equation}, Ph.D. Thesis, Bielefeld University, 2012. 

\bibitem{Z12b} X. Zhu, \emph{SDE in infinite dimensions: the reflection problem and the stochastic quasi-geostrophic equation}, Ph.D. Thesis, Bielefeld University, 2012. 

\end{thebibliography}
\end{document}